%% file: main.tex
\documentclass[11pt]{article}
\usepackage[margin=1.1in]{geometry}

\title{An index formula for families of end-periodic Dirac operators}
\date{}
\author{Alex R. Taylor\\University of Illinois Urbana-Champaign}

\usepackage[utf8]{inputenc}
\usepackage[T1]{fontenc}   
\usepackage[export]{adjustbox}
\usepackage[bottom]{footmisc}
\usepackage[shortlabels]{enumitem}
\usepackage{amsmath, amssymb, amsthm, amsfonts, fancyhdr, subcaption, graphicx, xcolor, mdframed, array, tabu, multirow, mathtools, latexsym, mathrsfs, tikz, epsfig, color, soul, pgfplots, wrapfig, nicefrac, booktabs, url, hyperref, accents, comment, marginnote, setspace, ifthen}
\usepackage{quiver}
\tikzcdset{every label/.append style = {font = \footnotesize}}
\graphicspath{ {./images/} }

\usepackage[final,nopatch=footnote]{microtype}

\usepackage{titlesec}
\titleformat{\section}
  {\normalfont\normalsize\bfseries}{\thesection.}{1em}{}
\titleformat{\subsection}[runin]{\normalfont\small\bfseries}{\thesubsection.}{5pt}{\addperiod}
\newcommand{\addperiod}[1]{#1.}

\newcommand\numberthis{\addtocounter{equation}{1}\tag{\theequation}}
\newtheorem{theorem}{Theorem}
\newtheorem*{theorem*}{Theorem}

\numberwithin{theorem}{section}
\numberwithin{equation}{section}

\newtheorem{corollary}[theorem]{Corollary}

\newtheorem{proposition}[theorem]{Proposition}
\newtheorem{lemma}[theorem]{Lemma}

\theoremstyle{definition}
\newtheorem{example}[theorem]{Example}
\AtBeginEnvironment{example}{%
  \pushQED{\qed}%
}
\AtEndEnvironment{example}{\popQED\endexample}

\newtheorem{remark}[theorem]{Remark}
\AtBeginEnvironment{remark}{%
  \pushQED{\qed}%
}
\AtEndEnvironment{remark}{\popQED\endremark}

\newlist{axioms}{enumerate}{1}
\setlist[axioms]{
    label=\textbf{(AP\arabic*)~}, % Custom label
    itemsep=0.5em, % Space between items
    parsep=0.2em, % Space between paragraphs within an item
    leftmargin=4em % Indentation from the left margin
}

\DeclareMathSymbol{\widetildesym}{\mathord}{largesymbols}{"65}
    \newcommand\lowerwidetildesym{%
      \text{\smash{\raisebox{-1.3ex}{%
        $\widetildesym$}}}}
            \newcommand\tildhat[1]{%
              \mathchoice
                {\accentset{\displaystyle\lowerwidetildesym}{#1}}
        {\accentset{\textstyle\lowerwidetildesym}{#1}}
        {\accentset{\scriptstyle\lowerwidetildesym}{#1}}
        {\accentset{\scriptscriptstyle\lowerwidetildesym}{#1}}
    }

\newcommand{\tild}
{\raise.17ex\hbox{$\scriptstyle\mathtt{\sim}$}}

\DeclareMathOperator{\supp}{supp}

\DeclareMathOperator{\Rg}{Rg}

\DeclareMathOperator{\Ker}{Ker}
\DeclareMathOperator{\Coker}{Coker} %% vector bundle
\DeclareMathOperator{\coker}{coker} %% vector space
\DeclareMathOperator{\ind}{ind} % Fredholm index
\DeclareMathOperator{\Ind}{Ind} % Index bundle

\DeclareMathOperator{\id}{id}

\DeclareMathOperator{\Hom}{Hom}
\DeclareMathOperator{\End}{End}

\DeclareMathOperator{\grad}{grad}

\DeclareMathOperator{\en}{end}
\DeclareMathOperator{\cl}{c\ell} %% Clifford multiplication
\DeclareMathOperator{\Cl}{Cl} %% Clifford bundle
\DeclareMathOperator{\rch}{\prescript{R}{}ch} %% renormalized Chern character
\DeclareMathOperator{\bch}{\prescript{b}{}ch} %% b-Chern character
\DeclareMathOperator{\ch}{ch} %% Chern character
\DeclareMathOperator{\rStr}{\prescript{R}{}Str} %% renormalized supertrace
\DeclareMathOperator{\rTr}{\prescript{R}{}Tr} %% renormalized trace
\DeclareMathOperator{\bStr}{\prescript{b}{}Str} %% b-supertrace
\DeclareMathOperator{\bTr}{\prescript{b}{}Tr} %% b-trace
\DeclareMathOperator{\Str}{Str}
\DeclareMathOperator{\Tr}{Tr}
\DeclareMathOperator{\tr}{tr}
\DeclareMathOperator{\str}{str}
\DeclareMathOperator{\vol}{vol}

\newcommand{\lessim}{\mathrel{\mathop{<}\limits_{\raisebox{0.6ex}{$\scriptstyle\sim$}}}}

\newcommand{\epdiff}{\mathrm{Diff}_{\pi,\mathrm{ep}}}

\newcommand{\epeta}{\widehat{\eta}_{\mathrm{ep}}}

\newcommand{\AS}{\mathrm{\mathbf{I}}}

\newcommand{\mathdash}{\relbar\mkern-9mu\relbar}

\newcommand{\con}{\nabla}

\newcommand*{\dt}[1]{\accentset{\mbox{\large\bfseries .}}{#1}}

\DeclareSymbolFont{toneitalic}{T1}{\familydefault}{m}{it}
\DeclareMathSymbol{\dir}{\mathord}{toneitalic}{"F0}

\newcommand\mf\mathfrak
\newcommand\mc\mathcal
\newcommand\mb\mathbb

\makeatletter
\newcommand*\bigcdot{\mathpalette\bigcdot@{.5}}
\newcommand*\bigcdot@[2]{\mathbin{\vcenter{\hbox{\scalebox{#2}{$\m@th#1\bullet$}}}}}
\makeatother

\pgfplotsset{compat=1.15} %% do I need to comment this line out?

\hypersetup{
  colorlinks % Defaults to red
}
\hypersetup{linkcolor=blue, citecolor=magenta}

\begin{document}

\maketitle
\thispagestyle{empty} %% removes page number from title page

\abstract{We derive a transgression formula for the renormalized Chern character of the Bismut superconnection in the context of end-periodic fiber bundles and families of end-periodic Clifford modules. The transgression is expressed in terms of the Fourier-Laplace transform of the Bismut superconnection using the renormalized supertrace of Mrowka-Ruberman-Saveliev. Consequently, we establish an index formula for families of Dirac operators on end-periodic manifolds. The index formula involves a new ``end-periodic eta form'' which generalizes both the Bismut-Cheeger eta form and the end-periodic eta invariant of Mrowka-Ruberman-Saveliev.}

\tableofcontents
\pagenumbering{arabic} 

%------------------------------
% body
%------------------------------
\input{sections/1-Intro}
\input{sections/2-EP-geometry}
\input{sections/3-FL-transform}
\input{sections/4-Renormalized-supertrace}
\input{sections/5-Transgression-formula}
\input{sections/6-Chern-character}

%------------------------------
% appendix and references
%------------------------------

\input{sections/references}
\end{document}

%% file: sections/1-Intro.tex
\section{Introduction and main results}\label{section:intro}

Given a closed Riemannian fiber bundle $Z \to B$, a Clifford module $E \to Z$, and a family of Dirac operators on the fibers, there is a natural class in the $K$-theory of the parameter space $B$ called the \emph{index bundle} which generalizes the Fredholm index of a single operator. The objective of a local index formula for families is to find a nice representative for the Chern character of the index bundle in de Rham cohomology. In 1985 Quillen \cite{quillen1} suggested an approach to this problem (later carried out by Bismut \cite{bismut}) in the category of closed manifolds using an infinite-dimensional, $\mb{Z}_{2}$-graded version of the transgression formula appearing in Chern-Weil theory, i.e., a formula for the variation of the Chern character of a path of connections, where one replaces connections with superconnections and the matrix exponential with the heat operator. One shows that
\[
\frac{d}{dt}\Str\left(e^{-A_{t}^{2}}\right) = - d_{B}\Str\left(\dt{A}_{t}e^{-A_{t}^{2}}\right)
\]
and then integrating over $t \in (0,\infty)$ obtains a representative in de Rham cohomology for the Chern character of the index bundle, that is
\[
\ch(\con_{0}) - \int_{Z/B}\AS(D(Z)) = - d_{B}\Str\left(\dt{A}_{t}e^{-A_{t}^{2}}\right).
\]
In this formula,
\begin{enumerate}[(i)]
\item The local index form $\AS(D(Z)) = \widehat{A}(T(Z/B))\wedge\ch'(E)$ is defined in terms of the local geometries of $Z$ and $E$. Here $\widehat{A}(T(Z/B))$ is the A-hat genus of the vertical tangent bundle and $\ch'(E)$ is the Chern character of the twisting curvature of $E \to Z$.
\item The Chern character $\ch(\con_{0})$ is a closed differential form on the base $B$ whose cohomology class equals the Chern character of the index bundle, and the index formula gives an explicit description of this representative.
\end{enumerate}
Of course, in order for this approach to succeed, one must construct a special superconnection $A_{t}$ so that the short-time limit of $\Str(e^{-A_{t}^{2}})$ exists. Bismut \cite{bismut} carried out this procedure via probabilistic methods to establish the families index formula in the case of closed manifolds. Berline and Vergne \cite{bv} and Donnelly \cite{don} then gave proofs of the theorem based on the classical expansion of heat kernels (also see Berline, Getzler, Vergne \cite[\S 10]{bgv}).

The local families index formula has seen generalizations to several categories of non-compact manifolds. Bismut and Cheeger \cite{bc1,bc2} established a local index formula for families in the context of manifolds with boundary,
\begin{equation}\label{eqn:cylendformula}
\ch(\con_{0}) - \int_{Z/B}\AS(D(Z)) = -\frac{1}{2}\widehat{\eta} - d_{B}\Str\left(\dt{A}_{t}e^{-A_{t}^{2}}\right)
\end{equation}
which generalizes the Atiyah-Patodi-Singer index theorem \cite{aps}. Here the boundary contribution $\widehat{\eta}$ is the Bismut-Cheeger eta form which generalizes the Atiyah-Patodi-Singer eta invariant. Melrose and Piazza \cite{mp} later proved the formula (\ref{eqn:cylendformula}) in the more general context of manifolds with b-metrics, which includes the case of manifolds with cylindrical ends. Other examples include the index formulas of \cite{tz}, \cite{liuma}, and \cite{ag}. The differential forms appearing in these index formulas often have interesting geometric interpretations, most notably the ``anomaly formula'' of Bismut-Freed for the curvature of the determinant line bundle associated with a family of Dirac operators, cf. \cite{quillen2} and \cite[\S 10.6]{bgv}. Further, as Freed \cite[\S 8]{freed} explains, the index theory for families has many interesting connections with mathematical physics.

Another interesting non-compact setting is that of end-periodic manifolds, where the non-compact end is allowed to have some periodically repeating topology. Taubes \cite{tau} originally introduced the index theory of elliptic operators on end-periodic manifolds to investigate exotic smooth structures on $\mb{R}^{4}$, and this end-periodic index theory has since found other interesting applications in the study of positive scalar curvature metrics and positive scalar curvature cobordism \cite{mrs1}, \cite{krs}, \cite{hm} and Seiberg-Witten invariants \cite{mrs2}. The end-periodic eta invariant arising from the index formula \cite[Theorem A]{mrs1} has played a fundamental role in these applications.

Thus we are motivated by the variety of interesting applications surrounding on the one hand the index theory of end-periodic operators, and the index theory of families of Dirac operators on the other. The main result of this article is a families index formula in the category of manifolds with periodic ends -- see Theorem \ref{mainthm1} and Theorem \ref{mainthm2}.

\begin{theorem*}
Let $B$ be a closed manifold. Let $M \to B$ be an end-periodic fiber bundle with compact part $Z$, and periodic end modeled on an infinite cyclic covering $\tildhat{Y}$ of an even-dimensional closed manifold $Y$. Included in the data of the periodic end is a function $f: \tildhat{Y} \to \mb{R}$ such that $W_{0} = f^{-1}[0,1]$ is the periodically repeating segment of the end.

Let $E \to M$ be a family of end-periodic Clifford modules, with $D$ the associated family of end-periodic Dirac operators, and suppose that the family $D$ is Fredholm. Then the Chern character of the index bundle of $D$ is represented in de Rham cohomology by the following differential form on $B$:
\begin{equation}\label{eqn:epformula}
\int_{Z/B}\AS(D(Z)) - \int_{W_{0}/B}f\,\AS(D(Y)) - \frac{1}{2}\epeta.
\end{equation}
Here, $\epeta$ is a differential form on $B$ whose degree 0 part is the end-periodic eta invariant of \cite[Theorem A]{mrs1}.
\end{theorem*}

The approach taken by \cite{mp} in proving the index formula (\ref{eqn:cylendformula}) is to derive a transgression formula for a renormalized supertrace adapted to the context of manifolds with cylindrical ends, namely the b-supertrace, and then apply the defect formula which describes the nonvanishing of the b-supertrace of a supercommutator. In proving the formula (\ref{eqn:epformula}) it becomes apparent that many of the distinguishing features as compared with (\ref{eqn:cylendformula}) originate from the presence of an additional term in the renormalized supertrace defect formula \cite[Theorem 4.6]{mrs1} in the end-periodic context. Most significantly, the end-periodic eta form $\epeta$ in (\ref{eqn:epformula}) involves a new term which arises in exactly this way. In \S\ref{section:epeta} we show that this new term can be interpreted geometrically in terms of the Hessian $\con^{T^{*}Y}df$; moreover, in \S\ref{section:hessian} we show that this Hessian measures the nontriviality of the periodic end in the sense that it vanishes if and only if the periodic end is actually cylindrical. \\

\noindent \textbf{Acknowledgements.} I am deeply grateful to my advisor, Pierre Albin, for providing invaluable guidance on this project, and for patiently reviewing innumerable earlier drafts. I would also like to thank Nikolai Saveliev for kindly answering my questions via email regarding end-periodic index theory.

%% file: sections/2-EP-geometry.tex
\section{End-periodic Riemannian fiber bundles}\label{section:EPgeometry}

We start by recounting the basic components of end-periodic geometry. For a more detailed exposition we refer the reader to \cite[\S 3]{tau} and \cite[\S 2]{mrs1}. Let $Y$ be any even-dimensional closed Riemannian manifold with nontrivial $H^{1}(Y,\mb{Z})$, called the \emph{furled-up manifold}. We choose a primitive cohomology class in $H^{1}(Y,\mathbb{Z})$ which specifies how to unfurl $Y$ to produce the periodic segment. This is equivalent to choosing a smooth map $f_{0}: Y \to S^{1}$, in which case the cohomology class is given by $f_{0}^{*}[d\theta] \in H^{1}(Y,\mathbb{Z})$. Corresponding to this class, the infinite cyclic covering space $p: \tildhat{Y} \to Y$ is defined via the pullback diagram
\begin{equation}\label{diagram:pullback}
\begin{tikzcd}
	\tildhat{Y} & {\mb{R}} \\
	Y & {S^{1}.}
	\arrow["f", from=1-1, to=1-2]
	\arrow["p"', from=1-1, to=2-1]
	\arrow["{e^{2\pi i t}}", from=1-2, to=2-2]
	\arrow["f_{0}", from=2-1, to=2-2]
\end{tikzcd}
\end{equation}
\noindent Thus $\tildhat{Y} = \{(t,y)\in \mb{R}\times Y: e^{2\pi it} = f_{0}(y)\}$. We will denote the projection $\tildhat{Y} \to \mb{R}$ by $f(t,y) = t$. If we fix a basepoint $y_{0} \in Y$, then $f_{0}$ is given by
\begin{align*}
&f_{0}: Y \to S^{1} \\
&f_{0}(y) = \exp\left(2\pi i\int_{\gamma(y)}f_{0}^{*}(d\theta)\right)
\end{align*}
where $\gamma(y)$ is any path starting at $y_{0} \in Y$ and ending at $y \in Y$. For any other choice of path, the difference of the integrals over the two paths yields an integral of $f_{0}^{*}(d\theta)$ over a closed loop based at $y_{0}$. This difference is necessarily an integer, so $f_{0}(y) \in S^{1}$ is well-defined.

The infinite cyclic cover $\tildhat{Y}$ will be the model for the end of an end-periodic manifold. To explain this, note that $\tildhat{Y}$ can be constructed rather explicitly as follows. Suppose for simplicity that $Y$ is connected. Define the periodic segment $W_{0} = f^{-1}[0,1] \subseteq \tildhat{Y}$, which is a connected submanifold with two boundary components. Then $\tildhat{Y}$ is diffeomorphic to the manifold obtained by recursively gluing copies of the segment $W_{0}$ together along the boundary components:
\[
\tildhat{Y} \simeq \bigsqcup_{k\in\mb{Z}}W_{k} = \tildhat{Y}_{+}\sqcup\tildhat{Y}_{-}
\]
where each $W_{k} = f^{-1}[k,k+1]$ is diffeomorphic to $W_{0}$, and $\tildhat{Y}_{\pm}$ are the two ends of $\tildhat{Y}$. An \emph{end-periodic manifold} with end modeled on $\tildhat{Y}$ is a smooth manifold $M$ with a compact submanifold $Z \subseteq M$ with boundary, such that $M\setminus Z$ is diffeomorphic to $\tildhat{Y}_{+}$. We let $\en(M) = M\setminus Z$ denote the \emph{periodic end} of $M$.

\begin{remark}\label{remark:epgeostuff}
It is worth mentioning a few things at this point:
\begin{enumerate}
\item In the special case where $\tildhat{Y} = \mb{R} \times \partial Z$, $W_{0} = [0,1] \times \partial Z$, $Y = S^{1}\times \partial Z$, and $f(t,q) = t$, then $M$ is just a manifold with a cylindrical end.

\item By allowing $Y$ to have multiple components we obtain the notion of an end-periodic manifold $M$ with multiple periodic ends. 

\item The projection $f: \tildhat{Y} \to \mb{R}$ in the diagram (\ref{diagram:pullback}) has the property that $f(t+k,y) = f(t,y) + k$ for any $k\in\mb{Z}$, and so the differential $df$ descends to a well-defined 1-form on $Y$, which is exactly $f_{0}^{*}(d\theta)$. Following \cite{mrs1}, the symbol $df$ will everywhere denote this (not necessarily exact) 1-form on $Y$.
\end{enumerate}
\end{remark}

A Riemannian metric $g_{Y}$ on $Y$ lifts to a metric $\tildhat{g} = p^{*}g_{Y}$ on $\tildhat{Y}$. A Riemannian metric $g$ on $M$ is called an \emph{end-periodic metric} if its restriction to $\en(M)$ coincides with $\tildhat{g}$ for some $g_{Y}$ on $Y$. In general we will always assume that an end-periodic manifold is equipped with an end-periodic metric.

Given any ordinary vector bundle $E_{Y} \to Y$ we construct a lifted vector bundle $\tildhat{E} = p^{*}E_{Y} \to \tildhat{Y}$ via the pullback diagram
\[\begin{tikzcd}
	\tildhat{E} & \tildhat{Y} \\
	{E_{Y}} & Y.
	\arrow[from=1-1, to=1-2]
	\arrow[from=1-1, to=2-1]
	\arrow["p", from=1-2, to=2-2]
	\arrow[from=2-1, to=2-2]
\end{tikzcd}\]
A vector bundle $E \to M$ is called an \emph{end-periodic vector bundle} if its restriction to $\en(M)$ coincides with $\tildhat{E} \to \tildhat{Y}$ for some $E_{Y} \to Y$. In the same vein, given a differential operator $P(Y)$ acting on sections of $E_{Y} \to Y$, we can define a lifted differential operator $\tildhat{P} = p^{*}P(Y)$ acting on sections of $\tildhat{E} \to \tildhat{Y}$. In particular we have
\[
\tildhat{P}(p^{*}u) = p^{*}(P(Y)u)
\]
for any $u \in \Gamma(Y;E_{Y})$. A differential operator $P$ acting on sections of an end-periodic vector bundle $E \to M$ is called an \emph{end-periodic differential operator} if its restriction to $\en(M)$ coincides with $\tildhat{P} = p^{*}P(Y)$ for some such differential operator $P(Y)$. We call $\tildhat{P}$ the \emph{normal operator} of $P$, also  denoted by $N(P) = \tildhat{P}$.

\begin{remark}[Bounded geometry]
A smooth Riemannian manifold has bounded geometry if its injectivity radius is positive and all covariant derivatives of its curvature tensor are uniformly bounded. An end-periodic manifold $M$ equipped with an end-periodic metric clearly has bounded geometry because the Riemannian metric on $\en(M)$ is the restriction of the lifted metric $\tildhat{g} = p^{*}g_{Y}$ from $Y$ and $p$ is a local isometry. For the same reason, it is clear that an end-periodic vector bundle $E \to M$ has bounded geometry.
\end{remark}

\begin{remark}[Clifford modules]
Our notion of \emph{Clifford module} coincides with the notion of Dirac bundle defined in \cite[Definition 5.2]{lm}. The data of a Clifford module $E \to M$ determines a distinguished Dirac-type operator $D: \Gamma(M;E) \to \Gamma(M;E)$ given by composing the Clifford action with the Clifford connection, cf. \cite[Proposition 3.42]{bgv}. In the presence of a $\mb{Z}_{2}$-grading $E = E^{+}\oplus E^{-}$ (compatible with the Clifford module data), the Dirac operator splits into the ``chiral parts'' $D^{\pm}: \Gamma(M;E^{\pm}) \to \Gamma(M;E^{\mp})$. Whereas $D$ is self-adjoint, the chiral operators satisfy $(D^{+})^{*} = D^{-}$. An important invariant associated with the Dirac operator is the \emph{Fredholm index}, $\ind D^{+} = \dim \ker D^{+} - \dim \coker D^{+}$.
\end{remark} 

Of course, the most important end-periodic differential operators for our purposes will be the end-periodic Dirac operators. Let $E_{Y} \to Y$ be a Clifford module with associated Dirac operator $D(Y)$. The Clifford module structure on $E_{Y} \to Y$ lifts to $\tildhat{E} \to \tildhat{Y}$, and the Dirac operator associated with this lifted structure is exactly the pullback $\tildhat{D} = p^{*}D(Y)$. Suppose that $M$ is an end-periodic manifold with end modeled on $\tildhat{Y} \to Y$. An \emph{end-periodic Clifford module} is by definition a Clifford module $E \to M$ which is also an end-periodic vector bundle, with data coming from
\begin{enumerate}[(i)]
\item A Clifford module structure on $E|_{Z} \to Z$.

\item A Clifford module structure on $E|_{\en(M)} \to \en(M)$ which coincides with that of $\tildhat{E} \to \tildhat{Y}$ on $\en(M)$ (and glues together smoothly with the data on $E|_{Z} \to Z$).
\end{enumerate}
Let $E \to M$ be an end-periodic Clifford module and let $D = D(M)$ be the associated Dirac operator. It follows from the definition of the Clifford module structure on $E \to M$ that $D$ coincides with $\tildhat{D}$ along $\en(M)$, and so $D$ is an end-periodic differential operator with normal operator $N(D) = \tildhat{D}$. We call such an operator an \emph{end-periodic Dirac operator}.

\begin{remark}
It is also natural to consider end-periodic differential operators of a slightly more general type, by requiring only that they be asymptotic to some periodic normal operator along the end, instead of coinciding exactly. Similarly, we could relax the definition of ``end-periodic metric'' by requiring it to be merely asymptotic to some periodic metric along the end. Such metrics have been considered for example in \cite{tau} and \cite{mpu}. All of our results extend in a straightforward manner to these more general metrics and operators.
\end{remark}

Now we will introduce the setting in which our index formula for families of end-periodic Dirac operators takes place. A fiber bundle $M_{z} \mathdash M \xrightarrow{\pi} B$ of smooth manifolds is called an \emph{end-periodic fiber bundle} if there is a fiber bundle $Y_{z} \mathdash Y \xrightarrow{\pi_{1}} B$ such that
\begin{enumerate}[(i)]
\item $M$ is an end-periodic manifold with furled-up manifold $Y$.
\item Each fiber $M_{z}$ is an end-periodic manifold with furled-up manifold $Y_{z}$.
\item The diffeomorphic identification of $\en(M)$ with $\tildhat{Y}_{+}$ is compatible with each of the identifications of $\en(M_{z})$ with $(\tildhat{Y}_{+})_{z}$ in that the following diagram commutes
\[
\begin{tikzcd}
	{\tildhat{Y}_{+}} & {\en(M)} \\
	Y & B.
	\arrow["\simeq", from=1-1, to=1-2]
	\arrow["p"', from=1-1, to=2-1]
	\arrow["\pi", from=1-2, to=2-2]
	\arrow["{\pi_{1}}"', from=2-1, to=2-2]
\end{tikzcd}
\]
\end{enumerate}
We will also assume that the base $B$ is closed. In this paper we will consider an end-periodic fiber bundle $M \xrightarrow{\pi} B$ and a family of end-periodic vector bundles $E_{z} \to M_{z}$ (which assemble into a total end-periodic vector bundle $E \to M$ with $E|_{M_{z}} = E_{z}$). We will use the notation $\epdiff^{\bigcdot}(M;E)$ to denote the space of vertical families of end-periodic differential operators on $E \to M$, and $\mathrm{Diff}_{\pi_{1}}^{\bigcdot}(Y;E_{Y})$ the space of vertical families of differential operators on $E_{Y} \to Y$.

For the fiber bundle $Y \to B$ we adopt the geometric setup of \cite[\S 10.1-10.3]{bgv}:
\begin{itemize}
\item Choose a metric $g_{B}$ on $TB$ and a metric $g_{Y}$ on $TY$. This choice of metric on $Y$ induces a connection on the fiber bundle $Y \to B$, i.e., a splitting $TY = \pi^{*}TB \oplus T(Y/B)$ by identifying $\pi^{*}TB$ with the orthogonal complement of $T(Y/B)$. Let $g_{Y/B}$ denote the induced metric on the vertical tangent bundle $T(Y/B)$ and let $\pi_{V}$ denote the projection onto the vertical component. Define the degenerate metric on $TY$ by
\begin{equation}\label{eqn:degenmetric}
g^{Y}_{0}(u,v) = g_{Y/B}(\pi_{V}u,\pi_{V}v).
\end{equation}
We define the degenerate Clifford bundle $\Cl_{0}(Y) = \Cl(T^{*}Y,g^{Y}_{0})$ with respect to this metric.

\item Suppose $E_{Y}$ is a Clifford module for the vertical cotangent bundle of $Y$ with Clifford action $\cl: \Cl(T^{*}(Y/B),g_{Y/B}) \to \End E_{Y}$ and Clifford connection $\nabla^{E_{Y}}$. From this Clifford module structure we get a vertical family of Dirac operators $D(Y) = (D(Y)^{z})_{z\in B}$.

\item Define the vector bundle $\mb{E}_{Y} = \pi^{*}\Lambda T^{*}B \otimes E_{Y} \to Y$ and degenerate Clifford action $m_{0}: \Cl_{0}(Y) \to \End \mb{E}_{Y}$ which acts on the differential form component of $\mb{E}_{Y}$ via exterior multiplication and acts on the $E_{Y}$-valued component of $\mb{E}_{Y}$ via the vertical Clifford action. More precisely, any $\alpha \in T^{*}Y$ decomposes uniquely into horizontal and vertical parts $\alpha = \alpha_{H} + \alpha_{V}$, and then
\[
m_{0}(\alpha)(\beta \otimes u) = (\alpha_{H} \wedge \beta)\otimes u + (-1)^{\deg \beta}\beta \otimes \cl(\alpha_{V})(u)
\]
for any $\beta \otimes u \in \Gamma(Y;\mb{E}_{Y})$. The action then extends in an obvious way. Thus $\mb{E}_{Y}$ is a Clifford module for the degenerate Clifford bundle $\Cl_{0}(Y)$.  The Clifford module $\mb{E}_{Y}$ is horizontally degenerate in the sense that $m_{0}(\alpha)^{2} = -g_{Y/B}(\alpha_{V},\alpha_{V})$, i.e., $m_{0}(\alpha_{H})^{2} = 0$.

\item With respect to the aforementioned degenerate Clifford module structure on $\mb{E}_{Y}$, we define a Clifford connection on $\mb{E}_{Y}$ by
\[
\nabla^{\mb{E}_{Y},0} = \pi^{*}\nabla^{B}\otimes 1 + 1 \otimes \nabla^{{E}_{Y}} + \frac{1}{2}m_{0}(\omega)
\]
where $\omega$ is the 2-tensor on $Y$ defined in \cite[Definition 10.5]{bgv} in terms of the second fundamental form and the curvature associated with the vertical connection $\nabla^{Y/B}$ and metric $g_{Y/B}$.

\item The Bismut superconnection $A(Y)$ on $E_{Y}$ is the Dirac operator on $\mb{E}_{Y}$ associated with (a) the degenerate Clifford action $m_{0}$ and (b) the Clifford connection $\nabla^{\mb{E}_{Y},0}$ on $\mb{E}_{Y}$. Thus, for any orthonormal local frame $(e_{j})$ for $TY$ with dual coframe $(e^{j})$, we define
\[
A(Y) = \sum_{j}m_{0}(e^{j})\con^{\mb{E}_{Y},0}_{e_{j}}.
\]
When restricted to $\Gamma(Y;E_{Y})$, the Bismut superconnection $A(Y)$ can be decomposed as
\[
A(Y) = D(Y) + A_{[1]}(Y) + A_{[2]}(Y)
\]
where $A_{[j]}(Y): \Gamma(Y;E_{Y}) \to \Gamma(Y;\pi^{*}\Lambda^{j}T^{*}B\otimes E_{Y})$ for each $j \geq 0$. The rescaled Bismut superconnection is given by
\[
A_{t}(Y) = t^{1/2}D(Y) + A_{[1]}(Y) + t^{-1/2}A_{[2]}(Y)
\]
for any $t > 0$. The Bismut superconnection $A(Y)$ is a first-order differential operator which is odd with respect to the $\mb{Z}_{2}$-grading on $\mb{E}_{Y}$. See \cite[Proposition 10.15]{bgv} for more details on the Bismut superconnection in the context of closed manifolds.
\end{itemize}

The following basic fact will enable us to transport the tools of end-periodic geometry to the family context.

\begin{proposition}\label{prop1}
Let $M \xrightarrow{\pi} B$ be an end-periodic fiber bundle. For a family of end-periodic vector bundles $E\to M$ with associated bundles $\tildhat{Y} \xrightarrow{p} Y$, $Y \xrightarrow{\pi_{1}} B$ and $\tildhat{Y} \xrightarrow{\pi_{2}} B$, and for any vector bundle $V \to B$, we have an isomorphism of vector bundles over $\en(M)$,
\[
E \otimes \pi^{*}V|_{\en(M)} \simeq p^{*}(E_{Y}\otimes \pi_{1}^{*}V)|_{\tildhat{Y}_{+}}
\]
where we identify $\en(M) \simeq \tildhat{Y}_{+}$. Thus $E \otimes \pi^{*}V \to M$ is an end-periodic vector bundle. Similarly, $\tildhat{E} \otimes \pi_{2}^{*}V \to \tildhat{Y}$ is an end-periodic vector bundle.
\end{proposition}
\begin{proof}
Since $E$ is an end-periodic vector bundle we have $E = p^{*}E_{Y}$ over $\en(M)$. Moreover, from the compatibility condition (iii) in the definition of the end-periodic fiber bundle $M \to B$ we have $\pi_{1}\circ p = \pi$ after identifying $\tildhat{Y}_{+}$ with $\en(M)$, and therefore
\[
p^{*}(E_{Y}\otimes \pi_{1}^{*}V) \simeq p^{*}E_{Y} \otimes p^{*}\pi_{1}^{*}V \simeq E \otimes \pi^{*}V
\]
over $\en(M) \simeq \tildhat{Y}_{+}$. Owing to this isomorphism, $E\otimes \pi^{*}V$ fits into the pullback diagram
\[\begin{tikzcd}
	{E\otimes\pi^{*}V} & {E_{Y}\otimes\pi_{1}^{*}V} \\
	{\en(M)} & Y
	\arrow[from=1-1, to=1-2]
	\arrow[from=1-1, to=2-1]
	\arrow["\mathrm{proj}", from=1-2, to=2-2]
	\arrow["p"', from=2-1, to=2-2]
\end{tikzcd}\]
which is to say that $E \otimes \pi^{*}V$ is an end-periodic vector bundle over $M$.
\end{proof}

Applying Proposition \ref{prop1} with $V = \Lambda T^{*}B$ we deduce that $\mb{E} = \pi^{*}\Lambda T^{*}B \otimes E \to M$ and $\tildhat{\mb{E}} = \pi_{2}^{*}\Lambda T^{*}B \otimes \tildhat{E} \to \tildhat{Y}$ are end-periodic vector bundles. As a result, all of the geometric data (e.g., metrics, Clifford module structure, Clifford connections, etc.) of $\mb{E}_{Y} \to Y$ lifts to $\tildhat{\mb{E}} \to \tildhat{Y}$ and then (together with the given data on $E|_{Z} \to Z$) yields end-periodic geometric data on $\mb{E} \to M$. We let $\tildhat{A}$ denote the associated Bismut superconnection on $\tildhat{\mb{E}} \to \tildhat{Y}$ and $A$ the associated Bismut superconnection on $\mb{E} \to M$. We call $A$ the \emph{end-periodic Bismut superconnection} on $\mb{E} \to M$, it is an end-periodic first-order differential operator which is odd with respect to the $\mb{Z}_{2}$-grading on $\mb{E}$.

A key feature of the Bismut superconnection $A$ is that it satisfies a generalized Lichnerowicz formula for the curvature $A^2$, see \cite[Theorem 10.17]{bgv} and equation (9.25) of \cite{mp}. Let $L = (L^{z})_{z \in B}$ denote the family of connection Laplacians associated with the connections $\con^{\mb{E},0}|_{M_{z}}$ on the Clifford modules $E_{z} \to M_{z}$. Let $S_{M/B}$ denote the scalar curvature of the fibers of $M \to B$, and let $F_{E}' \in \Omega(M,\End_{\Cl(T^{*}(M/B))}(E))$ denote the twisting curvature of $E$. Then
\begin{equation}\label{eqn:lich}
A^{2} = L + \frac{1}{4}S_{M/B} - \frac{1}{2}\sum_{a,b}m_{0}(e^{a})m_{0}(e^{b})F_{E}'(e_{a},e_{b}).
\end{equation}
where $(e^{a})$ is an orthonormal local frame for $T^{*}M$. Similar formulas hold of course for the Bismut superconnections $A(Y)$ on $\mb{E}_{Y} \to Y$ and $\tildhat{A}$ on $\tildhat{\mb{E}} \to \tildhat{Y}$. 

\begin{remark}[Supercommutator]\label{remark:supercomm}
The natural Lie bracket in the $\mb{Z}_{2}$-graded setting is the supercommutator. For operators $P,Q$ acting on sections of a $\mb{Z}_{2}$-graded vector bundle $E = E^{+}\oplus E^{-} \to M$ that are even or odd with respect to the $\mb{Z}_{2}$-grading, their supercommutator is defined as
\[
[P,Q] = PQ - (-1)^{|P||Q|}QP
\]
where we set $|P| = 0$ if $P$ is an even operator and $|P| = 1$ if $P$ is an odd operator.
\end{remark}

%% file: sections/3-FL-transform.tex
\section{Fourier-Laplace transform of the Bismut superconnection}\label{section:FLtransform}

One of the main tools we will use is the Fourier-Laplace (FL) transform introduced in the context of end-periodic geometry in \cite[Proposition 4.2]{tau}. Since $\tildhat{\mb{E}} \to \tildhat{Y}$ is an end-periodic vector bundle, we have for any $\xi \in \mb{C}^{\times}$ a well-defined FL transform, initially defined on compactly supported smooth sections by
\begin{align*}
&\mc{F}_{\xi}: C_{c}^{\infty}(\tildhat{Y};\tildhat{\mb{E}}) \to C^{\infty}(Y;\mb{E}_{Y}) \\
&(\mc{F}_{\xi}u)(p(x)) = \xi^{f(x)}\sum_{m\in\mb{Z}}\xi^{m}u(x+m).
\end{align*}
Here $x+m$ denotes the $\mb{Z}$-action on the covering space $\tildhat{Y}$, i.e., $(y,t)+m = (y,t+m)$, and the function $f:\tildhat{Y} \to \mb{R}$ comes from the pullback diagram (\ref{diagram:pullback}). Note that the right-hand side is a sum of vectors in the fiber $\mb{E}_{Y}|_{p(x)}$, because the fibers of $\tildhat{\mb{E}} = p^{*}\mb{E}_{Y}$ are by definition $\tildhat{\mb{E}}|_{x+m} = \mb{E}_{Y}|_{p(x+m)} = \mb{E}_{Y}|_{p(x)}$. Also, the sum is finite because $u$ is compactly supported.

For any vertical family of end-periodic differential operators $P \in \epdiff^{s}(M;\mb{E})$, let $N(P) \in \mathrm{Diff}_{\pi_{2}}^{s}(\tildhat{Y};\tildhat{\mb{E}})$ denote the normal operator. The FL transform extends by continuity to the Sobolev space $H^{k}(\tildhat{Y},\tildhat{\mb{E}})$ for any $k \geq 0$, see
\cite[Proposition 3.15]{tony}.
The FL transform defines a linear map
\begin{align*}
\mc{F}: C_{c}^{\infty}(\tildhat{Y},\tildhat{\mb{E}}) \to \Rg\mc{F} \subseteq \{v:S^{1} \to C^{\infty}(Y;\mb{E}_{Y})\}
\end{align*}
by $(\mc{F}u)(\xi) = \mc{F}_{\xi}u$, which also extends by continuity to a linear isomorphism $\mc{F}:H^{k}(\tildhat{Y};\tildhat{\mb{E}}) \xrightarrow{\simeq} L^{2}(S^{1},H^{k}(Y;\mb{E}_{Y}))$, with inverse given by
\begin{equation}\label{eqn:FLinverse}
\mc{F}^{-1}(v_{\theta})(x) = \frac{1}{2\pi i}\oint_{S^{1}}\xi^{-f(x)}v_{\xi}(p(x))\frac{d\xi}{\xi}.
\end{equation} Using the FL transform we may define another operator
\[
FL(N(P)): L^{2}(S^{1},H^{k}(Y;\mb{E}_{Y})) \to L^{2}(S^{1},H^{k-s}(Y;\mb{E}_{Y}))
\]
determined by the relation
\[
\mc{F}\circ N(P) = FL(N(P))\circ\mc{F}.
\]
Then for any $\xi \in S^1$, we get a vertical family of differential operators $P(Y)(\xi) \in \mathrm{Diff}_{\pi_{1}}^{s}(Y;\mb{E}_{Y})$ determined by the relation
\[
\mc{F}_{\xi}\circ N(P) = P(Y)(\xi)\circ \mc{F}_{\xi}.
\]
We call $P(Y)(\xi)$ the \emph{indicial family} of $P$. The operators $FL(N(P))$ and $P(Y)(\xi)$ are related by
\[
FL(N(P))(\mc{F}u)(\xi) = \mc{F}_{\xi}(N(P)u) = P(Y)(\xi)(\mc{F}_{\xi}u) \in H^{k-s}(Y;\mb{E}_{Y})
\]
for any $\xi \in S^{1}$ and $u \in C_{c}^{\infty}(\tildhat{Y};\tildhat{\mb{E}})$.

Taubes in \cite{tau} used the FL transform to characterize the Fredholmness of end-periodic Dirac operators. In contrast with the situation for closed manifolds, an end-periodic Dirac operator $D$ need not be Fredholm on $H^{1}(M;E)$; in addition to ellipticity, $D$ must be ``invertible at infinity'', a notion which is captured precisely by the invertibility of the indicial family $D(Y)(\xi)$. This is analogous to the case of manifolds with cylindrical ends, where a Dirac operator is Fredholm if and only if the Mellin transform of the indicial operator is invertible \cite[Theorem 5.40]{melrose}. Throughout the paper we will use the following result of Taubes, which we have paraphrased from \cite[Proposition 4.1, Lemma 4.3]{tau}.
\begin{theorem*}[Taubes]\label{taubes1}
Let $E \to M$ be an end-periodic Clifford module with end-periodic Dirac operator $D$ and induced Dirac operators $D(Y)$ and $\tildhat{D} = N(D)$. The following are equivalent:
\begin{enumerate}[(i)]
\item $D: H^{1}(M;E) \to L^{2}(M;E)$ is Fredholm.
\item $\tildhat{D}: H^{1}(\tildhat{Y};\tildhat{E}) \to L^{2}(\tildhat{Y};\tildhat{E})$ is invertible.
\item $D(Y)(\xi): H^{1}(Y;E_{Y}) \to L^{2}(Y;E_{Y})$ is invertible for every $\xi \in S^{1} \subseteq \mb{C}^{\times}$.
\end{enumerate}
\end{theorem*}
We note that the results of \cite{tau} also show that the end-periodic Dirac operator can always be made Fredholm on a suitable \emph{weighted} Sobolev space. Our results extend in a straightforward manner to the context of weighted Sobolev spaces, but for the sake of simplicity we will not include this detail. In the family context, Taubes' theorem tells us that a family of end-periodic Dirac operators $D = (D^{z})_{z\in B}$ is Fredholm on $H^{1}(M_{z};E_{z})$ (i.e., the index $\ind D^{z}$ is finite for every $z \in B$) if and only if the indicial families $D(Y)^{z}(\xi)$ are invertible for every $z \in B$ and $\xi \in S^{1} \subseteq \mb{C}^{\times}$. Throughout this paper we will be working under this assumption.

The following proposition summarizes some basic properties of the normal operator and the FL transform. The proof is a straightforward calculation.
\begin{proposition}[Properties of the FL transform]\label{FLprop2} \ \vspace{2mm} \\
Let $P,Q \in \epdiff^{\bigcdot}(M;\mb{E})$ with associated differential operators $P(Y),Q(Y) \in \mathrm{Diff}_{\pi_{1}}^{\bigcdot}(Y;\mb{E}_{Y})$ and $N(P),N(Q) \in \mathrm{Diff}_{\pi_{2}}^{\bigcdot}(\tildhat{Y};\tildhat{\mb{E}})$. Then 
\begin{enumerate}[a)]
\item $N(PQ) = N(P)\circ N(Q)$.
\item $[PQ](Y)(\xi) = P(Y)(\xi)\circ Q(Y)(\xi)$ for every $\xi \in S^{1}$.
\item $FL(N(P)N(Q)) = FL(N(P))\circ FL(N(Q))$.
\end{enumerate}
\end{proposition}

\noindent In the single-operator case (i.e., when the base $B$ of the fiber bundle is a point), the indicial family of an end-periodic Dirac operator is
\begin{equation}\label{eqn:indicialfamily}
D(Y)(\xi) = D(Y) - \log(\xi)\cl(df),
\end{equation}
for any $\xi \in S^1$, see equation (4.7) in \cite{tau}, or equation (9) in \cite{mrs1}. Note that $df$ is a smooth 1-form on $Y$, cf. Remark \ref{remark:epgeostuff}. There is a similar formula for the indicial family of the Bismut superconnection in the family case, which we will now derive.

\begin{proposition}\label{FLprop1}
Let $A$ denote the Bismut superconnection on $\mb{E} \to M$ and $A(Y)$ the Bismut superconnection on $\mb{E}_{Y} \to Y$. Thus we write
\begin{align*}
&A = D + A_{[1]} + A_{[2]} \\
&A(Y) = D(Y) + A_{[1]}(Y) + A_{[2]}(Y).
\end{align*}
The indicial family of the Bismut superconnection $A$ is given by
\begin{align*}
A(Y)(\xi) &= (D(Y) - \log(\xi)\cl(d_{Y/B}f)) + (A_{[1]}(Y) - \log(\xi)\pi^{*}d_{B}f) + A_{[2]}(Y) \\
&= D(Y)(\xi) + A_{[1]}(Y)(\xi) + A_{[2]}(Y) \\
&= A(Y) - \log(\xi)m_{0}(df).
\end{align*}
\end{proposition}
\begin{proof}
Since $\tildhat{E} = p^{*}E_{Y}$ we have
\[
\Gamma(\tildhat{Y};\tildhat{E}) = \Gamma(\tildhat{Y};p^{*}E_{Y}) \simeq C^{\infty}(\tildhat{Y})\otimes_{p^{*}C^{\infty}(Y)}p^{*}\Gamma(Y;E_{Y})
\]
so any compactly supported section of $\tildhat{\mb{E}}$ has the form $\tildhat{u} = \varphi\cdot p^{*}u$ for some $u \in \Gamma(Y;\mb{E}_{Y})$ and $\varphi \in C_{c}^{\infty}(\tildhat{Y})$. By definition of the lifted Bismut superconnection $\tildhat{A}$ we have $\tildhat{A}(p^{*}u) = p^{*}(A(Y)u)$ and so
\[
\tildhat{A}\tildhat{u} = \tildhat{A}(\varphi\cdot p^{*}u) = m_{0}(d\varphi)p^{*}u + \varphi\cdot p^{*}(A(Y)u),
\]
noting that the principal symbol of $A(Y)$ is the Clifford action $m_{0}$. Since $W_{0}$ and $Y$ are almost everywhere identified we will use the same symbol $x$ to denote a point in $W_{0}$ and its projection in $Y$.  We fix $x \in Y$ and calculate
\begin{align*}
&(\mc{F}_{\xi}\circ \tildhat{A}\tildhat{u})(x) = \mc{F}_{\xi}\bigg(m_{0}(d\varphi)p^{*}u + \varphi\cdot p^{*}(A(Y)u)\bigg)(x) \\
&= \xi^{f(x)}\sum_{k\in\mb{Z}}\xi^{k}\bigg(m_{0}(d\varphi|_{x+k})(p^{*}u)(x+k) + \varphi(x+k)p^{*}(A(Y)u)(x+k)\bigg).
\end{align*}
On the other hand,
\begin{align*}
&(A(Y)\circ\mc{F}_{\xi}\tildhat u)(x) = A(Y)\cdot\sum_{k\in\mb{Z}}\xi^{f(x)}\xi^{k}\tildhat{u}(x+k) \\
&= \log(\xi)m_{0}(df)(\mc{F}_{\xi}\tildhat{u})(x) + \sum_{k\in\mb{Z}}\xi^{f(x)}\xi^{k}\tildhat{A}\cdot(\varphi(x+k)(p^{*}u)(x+k)) \\
&= \log(\xi)m_{0}(df)(\mc{F}_{\xi}\tildhat{u})(x) \\
&\hspace{4mm} +\xi^{f(x)}\sum_{k\in\mb{Z}}\xi^{k}\bigg(m_{0}(d\varphi|_{x+k})(p^{*}u)(x+k) + \varphi(x+k)p^{*}(A(Y)u)(x+k)\bigg) \\
&= \log(\xi)m_{0}(df)(\mc{F}_{\xi}\tildhat{u})(x) + (\mc{F}_{\xi}\circ \tildhat{A})\tildhat{u}(x).
\end{align*}
where in the second line we make sense of the action of $A(Y)$ on $\tildhat{u}$ using the fact that $\tildhat{E}|_{x+k} = E_{Y}|_{x}$, where the action coincides with that of $\tildhat{A}$. Comparing these two equations we see that
\[
(\mc{F}_{\xi}\circ \tildhat{A})\tildhat{u} = (A(Y)-\log(\xi)m_{0}(df))\cdot \mc{F}_{\xi}\tildhat{u}
\]
which means that the indicial family is exactly $A(Y)(\xi) = A(Y) - \log(\xi)m_{0}(df)$.
\end{proof}

In the single operator case, the end-periodic eta invariant appearing in the index formula of Mrowka-Ruberman-Saveliev \cite[Theorem A]{mrs1} is expressed in terms of the indicial family of the end-periodic Dirac operator in question. The index formula states that the index of an end-periodic Dirac operator $D$ is given by
\[
\ind D^{+} = \int_{Z} \AS(D(Z)) - \int_{Y}f\,\AS(D(Y)) - \frac{1}{2}\eta_{\mathrm{ep}}(D(Y)).
\]
Here the end-periodic eta invariant $\eta_{\mathrm{ep}}(D(Y))$ generalizes the Atiyah-Patodi-Singer (APS) eta invariant; it is defined in equation (27) of \cite{mrs1} by the formula
\begin{align*}
\eta_{\mathrm{ep}}(D(Y)) &= \int_{0}^{\infty}\eta_{\mathrm{ep}}(D(Y))(t)\,dt \numberthis \label{eqn:epetainv} \\
&= \frac{1}{\pi i}\int_{0}^{\infty}\oint_{S^{1}}\Tr\left(\cl(df)D_{\xi}^{+}e^{-tD_{\xi}^{-}D_{\xi}^{+}}\right)\frac{d\xi}{\xi}\,dt
\end{align*}
where $D_{\xi} = D(Y)(\xi)$ is the indicial family from (\ref{eqn:indicialfamily}) associated with the end-periodic Dirac operator $D$. Whereas the APS eta invariant is a global spectral invariant of the boundary Dirac operator $D(\partial Z)$, the end-periodic eta invariant is a global spectral invariant of the entire indicial family. We will see that the indicial family of the end-periodic Bismut superconnection plays an analogous role in the index formula for families of end-periodic Dirac operators.

For a superconnection $A$ on $\mb{E} \to M$, the action on $\Gamma(M;E)$ decomposes into components
\[
A = \sum_{j=0}^{\dim B} A_{[j]}
\]
where $A_{[j]}: \Gamma(M;E) \to \Gamma(M;\pi^{*}\Lambda^{j}T^{*}B \otimes E)$. For $t > 0$ we define the rescaled superconnection
\begin{equation}\label{eqn:rescaledsuperconn}
A_{t} = \sum_{j \geq 0} t^{(1-j)/2}A_{[j]}.
\end{equation}
In general, if we have a linear operator $P$ on $\mb{E}$ which decomposes into components $P = \sum_{j} P_{[j]}$, then we will let $\delta_{t}P$ denote the operator obtained by applying the same rescaling, namely,
\[
\delta_{t}P = \sum_{j \geq 0} t^{(1-j)/2}P_{[j]}
\]
so in particular $A_{t} = \delta_{t}A$. Another important example is the rescaled degenerate Clifford action
\[
\delta_{t}m_{0}(df) = t^{1/2}\cl(d_{Y/B}f) + \pi^{*}d_{B}f, 
\]
as well as the rescaled version of the indicial family of the Bismut superconnection,
\begin{equation}\label{eqn:rescaledFL}
A_{t}(Y)(\xi) = A_{t}(Y) - \log(\xi)\delta_{t}m_{0}(df).
\end{equation}

In order to prove the index formula we will need to be able to calculate the short-time limit of the pointwise supertrace of the heat kernel associated with the indicial operator $A_{t}^{2}(Y)(\xi)$,
\[
\lim_{t \to 0^{+}}\str\left(K_{e^{-A_{t}^{2}(Y)(\xi)}}(x,x)\right)
\]
for any $\xi \in S^{1}$. Unfortunately, the operator $A(Y)(\xi)$ is not a Bismut superconnection with respect to the usual data on $\mb{E}_{Y}$, so the result of \cite[Theorem 10.23]{bgv} cannot be applied directly. However, we will show here that $A(Y)(\xi)$ when twisted by a suitable connection produces a Bismut superconnection on a twisted bundle, which will enable us to calculate the short-time limit.

First we recall the notion of a first-order differential operator twisted by a connection, see for example \cite[Chapter IV, \S 9, Theorem 3]{palais}. Let $E \to M^{n}$ be a vector bundle and let $D$ be a first-order differential operator acting on sections of $E$. Let $\mc{W} \to M$ be an auxiliary vector bundle with connection $\con^{\mc{W}}$ and consider the twisted bundle $E\otimes\mc{W} \to M$. Then there is a unique first-order differential operator $D_{\mc{W}}$ acting on sections of $E \otimes\mc{W}$ satisfying $D_{\mc{W}}(u_{1}\otimes u_{2}) = (Du_{1})\otimes u_{2}$ whenever $\con^{\mc{W}}u_{2} = 0$. The twisted operator can be locally constructed as follows. Let $\sigma_{1}(D) \in \Gamma(M;\Hom(T^{*}M,\End E))$ denote the principal symbol of $D$, take a local frame $(e_{j})$ for $TM$, and define
\[
D_{\mc{W}}(u\otimes v) = (Du)\otimes v + \sum_{j=1}^{n} \sigma_{1}(D)(e^{j})u_{1}\otimes \con^{\mc{W}}_{e_{j}}u_{2}.
\]
Evidently this defines a first-order differential operator with the property that $D_{\mc{W}}(u_{1}\otimes u_{2}) = (Du_{1})\otimes u_{2}$ whenever $\con^{\mc{W}}u_{2} = 0$. We will call $D_{\mc{W}}$ the \emph{twist of $D$ by the connection $\con^{\mc{W}}$}.

Now we will apply this twisting construction to show that the operator $A(Y)(\xi)$ can be twisted to produce a Bismut superconnection on a twisted bundle. For any $\xi \in \mb{C}^{\times}$ let $L_{\xi} = \mb{C} \times Y \to Y$ denote the trivial complex line bundle over $Y$ equipped with the standard metric $(z_{1},z_{2}) \mapsto z_{1}\overline{z_{2}}$ on the $\mb{C}$ fibers, and the connection
\[
\con^{L_{\xi}}_{X}u = Xu -\ln \xi \cdot df(X) u
\]
for any $X \in \Gamma(Y;TY)$ and $u \in C^{\infty}(Y;\mb{C})$. Then $\con^{L_{\xi}}$ is a flat connection on $\mb{C}\times Y \to Y$, and it is compatible with the aforementioned metric when $|\xi| = 1$. We apply the twisting construction with $\mathcal{W} = L_{\xi}$ and consider the twisted bundle $\mb{E}_{Y} \otimes L_{\xi} \to Y$.

\begin{proposition}\label{twistedprop}
Consider the twisted bundle $\mb{E}_{Y} \otimes L_{\xi} \to Y$, equipped with the following structures:
\begin{itemize}
\item The tensor product metric formed from the bundle metric on $\mb{E}_{Y}$ and the bundle metric on $L_{\xi}$.
\item The Clifford action $m_{0}\otimes 1$, acting on $\mb{E}_{Y}$ by $m_{0}$ and acting on $L_{\xi}$ by the identity.
\item The tensor product connection 
\[
\con^{\mb{E}_{Y}\otimes L_{\xi},0} = \con^{\mb{E}_{Y},0}\otimes 1 + 1 \otimes \con^{L_{\xi}}.
\]
\end{itemize}
Then we have the following:
\begin{enumerate}[(a)]
\item $\mb{E}_{Y} \otimes L_{\xi} \to Y$ is a degenerate Clifford module with $\mb{Z}_{2}$-grading $\mb{E}_{Y} \otimes L_{\xi} = (\mb{E}_{Y}^{+}\otimes L_{\xi})\oplus(\mb{E}_{Y}^{-}\otimes L_{\xi})$.
\item The tensor product connection on $\mb{E}_{Y}\otimes L_{\xi}$ is a Clifford connection, and a metric connection when $|\xi| = 1$.
\item The Bismut superconnection $\mb{A}$ on $\mb{E}_{Y}\otimes L_{\xi} \to Y$ determined by the aforementioned data is precisely the twist of $A(Y)$ by the connection $\con^{L_{\xi}}$ on $L_{\xi}$, and $\mb{A}$ is also the twist of $A(Y)(\xi)$ by the trivial connection $d$ on $L_{\xi}$.
\end{enumerate}
\end{proposition}
The conclusion (a) is immediate. The conclusion (b) follows from a straightforward calculation using the fact that $\con^{\mb{E}_{Y},0}$ is a Clifford connection. The proof of part (c) amounts to the following calculation. Let $\mb{A}$ denote the Bismut superconnection on $\mb{E}_{Y}\otimes L_{\xi} \to Y$. Take any local frame $(e_{j})$ for $TM$, then the action of $\mb{A}$ on any $u_{1}\otimes u_{2} \in \Gamma(Y;\mb{E}_{Y} \otimes L_{\xi})$ looks like
\begin{align*}
\mb{A}(u_{1}\otimes u_{2}) &= \sum_{j}m_{0}(e^{j})\con_{e_{j}}^{\otimes}(u_{1}\otimes u_{2}) \\
&= \sum_{j} m_{0}(e^{j})\left(\con_{e_{j}}^{\mb{E}_{Y},0}u_{1}\otimes u_{2} + u_{1}\otimes\con_{e_{j}}^{L_{\xi}}u_{2}\right) \\
&= A(Y)u_{1}\otimes u_{2} + \sum_{j}m_{0}(e^{j})u_{1}\otimes\con^{L_{\xi}}_{e_{j}}u_{2}.
\end{align*}
Note that the principal symbol of $A(Y)$ as a first-order differential operator on $\mb{E}$ is just the Clifford action $m_{0}$, and since $A(Y)(\xi)$ and $A(Y)$ differ by an order zero operator we also have $\sigma_{1}(A(Y)(\xi)) = \sigma_{1}(A(Y)) = m_{0}$. Thus the preceding calculation shows that $\mb{A}$ is the twist of $A(Y)$ by the flat connection $\con^{L_{\xi}}$ on $L_{\xi}$. Going a step further, we also see that
\begin{align*}
\mb{A}(u_{1}\otimes u_{2}) &= A(Y)u_{1}\otimes u_{2} + \sum_{j}m_{0}(e^{j})u_{1}\otimes\con^{L_{\xi}}_{e_{j}}u_{2} \\
&= A(Y)u_{1}\otimes u_{2} + \sum_{j}m_{0}(e^{j})u_{1}\otimes\left(e_{j}(u_{2})-\log(\xi)e_{j}(f)u_{2}\right) \\
&= A(Y)(\xi)u_{1}\otimes u_{2} + \sum_{j}m_{0}(e^{j})u_{1}\otimes e_{j}(u_{2})
\end{align*}
which shows that $\mb{A}$ is the twist of $A(Y)(\xi)$ by the trivial connection $d$ on $L_{\xi}$.

%% file: sections/4-Renormalized-supertrace.tex
\section{Renormalized supertrace}

\subsection{Asymptotically periodic smoothing operators}
In order to make sense of the renormalized trace on end-periodic manifolds we first need a suitable class of operators to which it can be applied. In \cite{mrs1} the authors show that their renormalized trace is well-defined on operators of the form $D^{m}e^{-tD^{2}}$ for $m \geq 0$, which is all that is required to establish the index formula for end-periodic Dirac operators. Similarly, in our case too we will only need to apply the renormalized supertrace to certain smoothing operators. In this subsection we will define the notion of an asymptotically periodic smoothing operator. In the next subsection we will show that the renormalized supertrace is well-defined on these operators, and then show that the smoothing operators we are concerned with are asymptotically periodic.

Note that non-local operators constructed from end-periodic data do not necessarily have the property that they coincide with the induced normal operator along the end. For example, if $D$ is an end-periodic Dirac operator on $E \to M$ with normal operator $\tildhat{D}$, and $K$ and $\tildhat{K}$ are the corresponding heat kernels, then $K$ is in general asymptotic to $\tildhat{K}$ along $\en(M)$ (i.e., $K$ is not just the restriction of $\tildhat{K}$ along the end), see \cite[Proposition 10.6]{mrs1}. This leads to the notion of an asymptotically periodic smoothing operator, which we introduce below based on the heat kernel estimates of \cite[\S 10]{mrs1}.

Let $M^{n} \to B$ be an end-periodic fiber bundle and $E \to M$ a family of end-periodic vector bundles. Recall that we denote by $W_{0}$ the ``first copy'' of the fundamental segment along $\en(M)$, i.e., $W_{0} = f^{-1}[0,1]$. Given a vertical family of smoothing operators $H = (H^{z})_{z\in B}$ on $\mb{E} \to M$ we have a family of smooth integral kernels $K_{H}^{z} \in \Gamma(M_{z} \times M_{z},\Lambda T_{z}^{*}B \otimes E_{z}\boxtimes E_{z}^{*})$ such that
\[
(H^{z}u)(x) = \int_{M_{z}} K_{H}^{z}(x,x')u(x')\,dx'
\]
for any section $u \in \Gamma(M;E)$. For the sake of brevity we will often drop the $z$-superscript and write
\[
(Hu)(x) = \int_{M/B} K_{H}(x,x')u(x')\,dx',
\]
which one may interpret as a differential form on $B$ given by $(Hu)(x): z \mapsto (H^{z}u)(x)$.

We say that $H$ is \emph{asymptotically periodic} if there exists a family of periodic smoothing operators $(\tildhat{H}^{z})_{z\in B} = (p^{*}H(Y)^{z})_{z\in B}$ on $\tildhat{E} \to \tildhat{Y}$ with integral kernels $\tildhat{K}^{z}$ such that the following estimates hold:
\begin{axioms}

%(Prop 10.2)
\item For any $z \in B$ there exist constants $C, \gamma > 0$ such that 
\[
|K^{z}(x,y)| \leq Ce^{-\gamma d^{2}(x,y)}
\]
for all $x,y \in \bigsqcup_{k \geq 1}W_{k}^{z}$, where $d(x,y)$ is the Riemannian distance between $x$ and $y$.

%(Prop 10.6)
\item  For any $z \in B$ there exist constants $C,\gamma > 0$ such that
\[
|K^{z}(x,x) - \tildhat{K}^{z}(x,x)| \leq C e^{-\gamma d^{2}(x,W_{0})}
\]
for all $x \in \bigsqcup_{k\geq 1}W_{k}^{z}$.

%(Prop 10.9)
\item  For any $z \in B$ there exist constants $C,\gamma > 0$ such that
\[
|K^{z}(x,y)-\tildhat{K}^{z}(x,y)| \leq Ce^{-\gamma d^{2}}
\]
for all $x,y \in \bigsqcup_{k\geq 1}W_{k}^{z}$, where $d = \min\{d(x,W_{0}),d(y,W_{0})\}$. 
\end{axioms}
The constants $C,\gamma$ may depend on $z$ but not on $x$ or $y$. Also, they are not required to be the same for each estimate.

Note that an immediate consequence of (AP3) is that the operator $\tildhat{H}$ to which $H$ is asymptotic is uniquely determined, because if $\tildhat{H}_{1}$ and $\tildhat{H}_{2}$ are two periodic operators satisfying (AP3) then their difference $\tildhat{H}_{1} - \tildhat{H}_{2}$ must decay along $\en(M)$, but it is also periodic so it must be identically zero. Therefore, for an asymptotically periodic smoothing operator $H$ we say that $\tildhat{H} = N(H)$ is \emph{the normal operator} of $H$.

Our principal objective in this section is to show that the heat operator $e^{-sA^2}$ associated with the end-periodic Bismut superconnection $A$ is an asymptotically periodic smoothing operator, i.e., the heat kernel $K$ satisfies the conditions (AP1), (AP2), (AP3) where $\tildhat{K}$ is the heat kernel for the normal operator $\tildhat{A}$. Our proofs will for the most part follow those of \cite[\S 10]{mrs1}, except of course we will work with the Bismut superconnection $A$ and families of heat kernels.

Let $A$ denote the end-periodic Bismut superconnection on $\mb{E} \to M$ adapted to an end-periodic Dirac operator $D$ on $E \to M$. Consider the heat operator $e^{-sA^{2}}$ with smooth integral kernel $K(s,x,y)$. As a function of $s$ and $x$ with $y$ fixed, $K$ solves the initial value problem
\begin{equation}\label{2eqn1}
\left(\frac{\partial}{\partial s} + A^{2}\right)K(s,x,y) = 0, \hspace{3mm} \lim_{s\to 0^{+}}K(s,x,y) = \delta_{y}\cdot\id
\end{equation}
where $A$ acts on the $x$ variable and $\delta_{y} = \delta(x-y)$ is the delta distribution supported at $y$. The following estimate, which is immediate from \cite{don}, serves as the foundation for the heat kernel estimates we will establish in this section.

\begin{proposition}\label{prop:don}
Let $M^{n} \to B$ be a Riemannian fiber bundle with bounded geometry, and let $E \to M$ be a Clifford module with bounded geometry. Let $A$ denote the Bismut superconnection adapted to a family of Dirac operators $D = (D^{z})_{z\in B}$ on $E$, and let $K^{z}(s,x,y)$ denote the family of heat kernels. For any $z \in B$ and any $T > 0$ there exists a constant $C>0$ such that
\[
\left|\frac{\partial^{i}}{\partial s^{i}}\con_{x}^{j}\con_{y}^{k}K^{z}(s,x,y)\right| \leq Cs^{-n/2-i-|j|-|k|}e^{-d^2{(x,y)}/4s}
\]
for all $s \in (0,T]$ and $x,y \in M^{z}$, and the constant $C$ depends only on $T$ and $z$. Here $j$ and $k$ are multi-indices.
\end{proposition}
\begin{proof}
This estimate follows from \cite[Proposition 4.1]{don}, which is stated in the context of families of Dirac operators and manifolds/vector bundles with bounded geometry. Also see \cite[Proposition 10.1]{mrs1} where this estimate is stated for the case of a single Dirac operator.
\end{proof}

We will use the following lemma in the proof of Proposition \ref{prop:ap1}.
\begin{lemma}\label{laplacelemma}
Let $(M,g)$ be a Riemannian manifold, and let $E$ be a vector bundle with bundle metric $(\cdot,\cdot)$ and metric connection $\con$. Let $L = \con^{*}\con$ denote the connection Laplacian, and let $\Delta = d^{*}d = -*d*d$ denote the scalar Laplacian acting on $C^{\infty}(M)$. Then for any section $u \in \Gamma(M;E)$ we have
\[
\Delta(|u|^{2}) = 2(Lu,u) - 2|\con u|^{2}.
\]
In particular, when $E = M \times \mathbb{R}$ is trivial with flat metric and flat connection $d$ acting on scalar functions, we obtain
\[
\Delta(|u|^{2}) = 2u\Delta u - 2|du|^{2}
\]
for any $u \in C^{\infty}(M)$.
\end{lemma}
\begin{proof}
Owing to the compatibility of the connection with the bundle metric we have 
\[
d(|u|^{2}) = d(u,u) = 2(\con u, u) \implies \Delta (|u|^{2}) = d^{*}d(|u|^{2}) = 2d^{*}(\con u, u).
\]
In an orthonormal local frame $(e_{j})$ for $TM$ write $(\con u, u) = \sum_{j}(\con_{e_{j}}u,u)e^{j}$. Then
\begin{align*}
d^{*}(\con u, u) &= -\sum_{j}\con_{e_{j}}\cdot(\con_{e_{j}}u,u) \\
&= -\sum_{j}\left((\con_{e_{j}}\con_{e_{j}}u,u) + (\con_{e_{j}}u,\con_{e_{j}}u) \right) \\
&= (Lu,u) - |\con u|^{2}
\end{align*}
and this completes the derivation.
\end{proof}

\begin{proposition}\label{prop:ap1}
Let $K(s,x,y)$ denote the integral kernel for the heat operator $e^{-sA^2}$. There exists constants $C,\alpha,\gamma > 0$ such that
\[
|K^{z}(s,x,y)| \leq Ce^{\alpha s}s^{-n/2}e^{-\gamma d^{2}(x,y)/s}
\]
for all $x,y \in \en(M^{z})$ and all $s > 0$. In particular, for any $s > 0$ the heat operator $e^{-sA^{2}}$ satisfies (AP1).
\end{proposition}
\begin{proof}
The estimate is obtained by following the proof of \cite[Proposition 10.2]{mrs1} on each fiber, replacing the Dirac operator the Bismut superconnection $A$, and working with families of heat kernels. The same proof goes through with little change because the Bismut superconnection satisfies the Lichnerowicz formula $A^{2} = L + R$, cf. (\ref{eqn:lich}). Here $R$ is given in terms of the scalar curvature of the fibers of $M \to B$ and the twisting curvature of $E$, and $L = (L^{z})_{z \in B}$ is the family of connection Laplacians on the fibers $E_{z} \to M_{z}$ associated with the connections $\con^{\mb{E},0}|_{M_{z}}$. 

We can apply Lemma \ref{laplacelemma} on each fiber $E_{z} \to M_{z}$ with the scalar Laplacian $\Delta = d^{*}d$ acting on smooth functions on $M_{z}$, and $u = K^{z}$ the heat kernel along the fiber over $z \in B$ (and similarly we can apply the scalar version of Lemma \ref{laplacelemma} with $u = |K^{z}|$ on each fiber). We combine the resulting identities with Kato's inequality
\begin{equation}
\big|d|K^{z}|\big| \leq |\con K^{z}|
\end{equation}
which can be applied to a section of any Riemannian vector bundle with compatible connection, in particular our section $K^{z}$, see \cite[Lemma 10.4]{mrs1}. 
This yields the differential inequality
\begin{align*}
|K^{z}|\frac{\partial}{\partial s}|K^{z}| + |K^{z}|\Delta(|K^{z}|) + (RK^{z},K^{z}) &\leq 0
\end{align*}
and therefore
\begin{align*}
&|K^{z}|\frac{\partial}{\partial s}|K^{z}| + |K^{z}|\Delta(|K^{z}|) \leq -(RK^{z},K^{z}) \leq \|R\||K^{z}|^{2} \\
&\implies \frac{\partial}{\partial s}|K^{z}| + \Delta|K^{z}| \leq \|R\||K^{z}| \leq \alpha|K^{z}| \numberthis \label{2eqn8}
\end{align*}
where we bound the curvature term $\|R\| \leq \alpha$ using the bounded geometry of $M$. The rest of the proof then follows as in \cite[Proposition 10.2]{mrs1}.
\end{proof}

The following lemma is the analogue in the family context of \cite[Lemma 7.13]{roe98} which is used in the proof of \cite[Proposition 10.6]{mrs1}.
\begin{lemma}\label{lemma:curvatureaction}
Let $A$ denote the Bismut superconnection on $\mb{E} \to M$. For any smooth function $h \in C^{\infty}(M)$ and section $u \in \Gamma(M;\mb{E})$ we have
\[
A^{2}(hu) = m_{0}(\con^{T^{*}M}dh)u-2\con_{\grad h}^{\mb{E},0}u + hA^{2}u.
\]
\end{lemma}
\begin{proof}
Take any local orthonormal frame $(e_{j})$ for $TM$. By definition we have
\begin{align*}
A(hu) &= \sum_{j\geq 1}m_{0}(e^{j})\con_{e_{j}}^{\mb{E},0}(hu) \\
&= \sum_{j\geq 1}m_{0}(e^{j})(e_{j}(h)u + h\con_{e_{j}}^{\mb{E},0}u) \\
&= m_{0}(dh)u + hAu.
\end{align*}
Hence
\begin{align*}
A^{2}(hu) &= A(m_{0}(dh)u) + A(hAu) \\
&= [A,m_{0}(dh)]u + h A^{2}u,
\end{align*}
where the bracket denotes a supercommutator with respect to the $\mb{Z}_{2}$-grading on $\mb{E}$, cf. Remark \ref{remark:supercomm}. Since $\con^{\mb{E},0}$ is compatible with the Clifford action $m_{0}$, we have
\[
[A,m_{0}(dh)] = m_{0}(\con^{T^{*}M}dh) - 2\con^{\mb{E},0}_{\grad h}
\]
and this completes the derivation of the formula. Note that the same calculation also works for the Bismut superconnection $A(Y)$ on $\mb{E}_{Y} \to Y$ and $\tildhat{A}$ on $\tildhat{\mb{E}} \to \tildhat{Y}$.
\end{proof}

\begin{lemma}\label{mrs10.5}
Let $K(s,x,y)$ denote the integral kernel for the heat operator $e^{-sA^2}$. There exists constants $C,\alpha,\gamma > 0$ such that
\[
|\con^{\mb{\tildhat{E}},0} \tildhat{K}^{z}(s,x,y)| \leq Ce^{\alpha s}s^{-n/2-1}e^{-\gamma d^{2}(x,W_{0})/s}
\]
for all $x,y \in \bigsqcup_{k\geq 1}W_{k}^{z}$ and all $s > 0$.
\end{lemma}
\begin{proof}
The estimate follows by applying the same argument as in the proof of Proposition \ref{prop:ap1} to the section $\con^{\mb{\tildhat{E}},0} \tildhat{K}^{z}$. Also see the proof of \cite[Proposition 10.5]{mrs1} where this estimate is stated for the case of a single Dirac operator.
\end{proof}

\begin{proposition}\label{prop:ap2}
Let $K(s,x,y)$ denote the integral kernel for the heat operator $e^{-sA^2}$, and let $\tildhat{K}(s,x,y)$ denote the integral kernel for the heat operator $e^{-s\tildhat{A}^2}$. For any $z \in B$ there exists constants $C,\alpha,\gamma > 0$ such that
\[
|K^{z}(s,x,x) - \tildhat{K}^{z}(s,x,x)| \leq Ce^{\alpha s}e^{-\gamma d^{2}(x,W_{0})/s}
\]
for all $x \in \bigsqcup_{k\geq 1}W_{k}^{z}$ and $s > 0$. In particular, for any $s > 0$ the heat operator $e^{-sA^{2}}$ satisfies (AP2).
\end{proposition}
\begin{proof}
The estimate follows by adapting the argument used in the proof of \cite[Proposition 10.6]{mrs1} to the family setting: using property (AP1) and Lemma \ref{mrs10.5}, which we established above for the heat operator $e^{-sA^{2}}$, together with Lemma \ref{lemma:curvatureaction}, which provides the family version of the curvature identity for the Dirac operator.
\end{proof}

Finally, we establish property (AP3) for the heat operator. The argument is similar to the proof of Proposition \ref{prop:ap2} for property (AP2).
\begin{proposition}\label{prop:ap3}
Let $K(s,x,y)$ denote the integral kernel for the heat operator $e^{-sA^2}$, and let $\tildhat{K}(s,x,y)$ denote the integral kernel for the heat operator $e^{-s\tildhat{A}^2}$. For any $z \in B$ and any given $T > 0$ there exists constants $C, \gamma > 0$ such that
\[
|K^{z}(s,x,y) - \tildhat{K}^{z}(s,x,y)| \leq Ce^{-\gamma d^{2}/s} 
\]
for all $0 < s \leq T$ and $x,y \in \bigsqcup_{k\geq 1}W_{k}^{z}$, where $d = \min\{d(x,W_{0}),d(y,W_{0})\}$. In particular, for any $s > 0$ the heat operator $e^{-sA^{2}}$ satisfies (AP3).
\end{proposition}
\begin{proof}
One may follow the proof of Proposition \ref{prop:ap2}, and in this case since $0 < s \leq T$, replace the function $e^{\alpha s}$ with a constant upper bound.
\end{proof}

\begin{corollary}\label{cor:ap}
Let $A_{t}$ denote the rescaled end-periodic Bismut superconnection, cf. equation (\ref{eqn:rescaledsuperconn}). For any $s,t > 0$, the heat operator $e^{-sA_{t}^{2}}$ is an asymptotically periodic smoothing operator with periodic normal operator $e^{-s\tildhat{A}_{t}^2}$.
\end{corollary}

\begin{proposition}\label{prop:moreapops}
Let $P \in \epdiff^{\bigcdot}(M;\mb{E})$. Then $Pe^{-sA_{t}^2}$ is also a vertical family of asymptotically periodic smoothing operators, with normal operator $N(P)e^{-s\tildhat{A}_{t}^{2}}$.

In particular, it follows that for any $s,t > 0$, the following smoothing operators are asymptotically periodic:
\[
\mathrm{(i)} \ A_{t}e^{-sA_{t}^{2}} \hspace{3mm} \mathrm{(ii)} \ \frac{d}{dt}(A_{t}^{2})e^{-sA_{t}^{2}} \hspace{3mm} \mathrm{(iii)} \ \mc{Q} = \dt{A}_{t}e^{-sA_{t}^{2}}.
\]
\end{proposition}
\begin{proof}
Write $H = e^{-sA_{t}^{2}}$ with normal operator $N(H) = e^{-s\tildhat{A}_{t}^{2}}$. It's clear that the family $PH$ consists of smoothing operators, because the algebra of smoothing operators is a two-sided ideal inside the algebra of bounded operators. The family of smoothing kernels for $PH$ satisfy the estimates (AP1)-(AP3) with normal operator $N(PH) = N(P)e^{-s\tildhat{A}_{t}^{2}}$ because as we have shown above the heat kernels satisfy these estimates, and Proposition \ref{prop:don} yields estimates on the covariant derivatives of the heat kernels. For example, suppose for simplicity that $P$ is a first-order differential operator, then locally we can write
\[
P = \sum_{j}a_{j}\con^{\mb{E},0}_{X_{j}}
\]
for some coefficient endomorphisms $a_{j} \in \Gamma(M;\End\mb{E})$ and vector fields $X_{j} \in \Gamma(M;TM)$. Over the end of $M$ we have $P = N(P) = p^{*}P(Y)$ so the endomorphisms $a_{j}$ are uniformly bounded. By Proposition \ref{prop:ap3}, for any $z \in B$ we have constants $C, \gamma > 0$ such that
\[
|K_{H}^{z}(x,y) - K_{N(H)}^{z}(x,y)| \leq Ce^{-\gamma d^{2}}
\]
for all $x,y \in \bigsqcup_{k \geq 1}W_{k}^{z}$, with $d = \min\{d(x,W_{0}),d(y,W_{0})\}$. Using the estimate on covariant derivatives given by Proposition \ref{prop:don} and then following the argument of Proposition \ref{prop:ap3} again, we deduce that
\[
|\con^{E}_{X_{j}}\cdot(K^{z}_{H}(x,y) - K^{z}_{N(H)}(x,y))| \leq C_{2}\exp(-\gamma_{2} d^{2})
\]
for all $x,y \in \bigsqcup_{k \geq 1}W_{k}^{z}$ and each $j$. Using this estimate together with the uniform boundedness of the coefficients of $P$ we estimate
\begin{align*}
|K_{PH}^{z}(x,y) - K_{N(P)N(H)}^{z}(x,y)| &= |PK_{H}^{z}(x,y) - N(P)K_{N(H)}^{z}(x,y)| \\
&= |PK_{H}^{z}(x,y) - PK_{N(H)}^{z}(x,y)| \\
&\leq \sum_{j}|a_{j}^{z}(x)||\con^{E}_{X_{j}}\cdot(K^{z}_{H}(x,y) - K^{z}_{N(H)}(x,y))| \\
&\leq C_{3}\exp(-\gamma_{3} d^{2})
\end{align*}
which is exactly the estimate (AP3). Similar arguments can be used to show that (AP1) and (AP2) hold as well.
\end{proof}

\begin{proposition}\label{prop:normalop1}
Let $H_{1}$ and $H_{2}$ be asymptotically periodic smoothing operators. Then $N(H_{1}H_{2}) = N(H_{1})N(H_{2})$.
\end{proposition}
\begin{proof}
By definition, the normal operators are given by
\[
N(H_{j})(p^{*}u) = p^{*}(H_{j}(Y)u), \ j = 1,2
\]
and so the proof amounts to a direct calculation as in Proposition \ref{FLprop2}.
\end{proof}

\begin{proposition}\label{prop:normalop2}
Let $A_{t}$ denote the rescaled end-periodic Bismut superconnection. Consider the heat operator $H = e^{-sA_{t}^{2}}$ which by Corollary \ref{cor:ap} is asymptotically periodic with normal operator $\tildhat{H} = p^{*}(e^{-A_{t}^{2}(Y)}) = e^{-s\tildhat{A}_{t}^{2}}$. Let $H(Y)(\xi)$ denote the indicial family of $H$. Then we have
\[
H(Y)(\xi) = e^{-sA_{t}^{2}(Y)(\xi)}.
\]
\end{proposition}
\begin{proof}
By definition, the heat operator $e^{-sA_{t}^{2}(Y)(\xi)}$ maps any $u \in \Gamma(Y;\mb{E}_{Y})$ to a solution of the heat equation
\[
\begin{cases}
&\displaystyle\left(\frac{\partial}{\partial s} + A_{t}^{2}(Y)(\xi)\right)e^{-sA_{t}^{2}(Y)(\xi)}u = 0 \numberthis \label{eqn:FLheateqn} \\
&\displaystyle\lim_{s \to 0^{+}}e^{-sA_{t}^{2}(Y)(\xi)}u = u.
\end{cases}
\]
We will show that $H(Y)(\xi)u$ also solves this equation. Fix $\xi \in S^{1}$ and suppose that $u = \mc{F}_{\xi}v$. Then
\begin{align*}
H(Y)(\xi)u = H(Y)(\xi)\mc{F}_{\xi}v
= \mc{F}_{\xi}\circ \tildhat{H}v
= \mc{F}_{\xi}\circ e^{-s\tildhat{A}_{t}^{2}}v.
\end{align*}
Furthermore
\[
\left(\frac{\partial}{\partial s} + \tildhat{A}_{t}^{2}\right)e^{-s\tildhat{A}_{t}^{2}}v = 0
\]
and from Proposition \ref{FLprop2} we know that the indicial family of the curvature $A_{t}^{2}$ is $A_{t}^{2}(Y)(\xi) = (A_{t}(Y)(\xi))^2$, so taking the FL transform yields
\begin{align*}
&\left(\frac{\partial}{\partial s}\mc{F}_{\xi} + A_{t}^{2}(Y)(\xi)\circ\mc{F}_{\xi}\right)e^{-s\tildhat{A}_{t}^{2}}v = 0 \\
&\implies \left(\frac{\partial}{\partial s} + A_{t}^{2}(Y)(\xi)\right)H(Y)(\xi)u = 0.
\end{align*}
which is equation (\ref{eqn:FLheateqn}). Moreover, $H(Y)(\xi)u$ satisfies the initial condition in equation (\ref{eqn:FLheateqn}). We have
\[
\lim_{s \to 0^{+}} e^{-s\tildhat{A}_{t}^{2}}v = v
\]
so taking the FL transform yields
\[
\lim_{s \to 0^{+}} H(Y)(\xi)u = \lim_{s \to 0^{+}} \mc{F}_{\xi}e^{-s\tildhat{A}_{t}^{2}}v = \mc{F}_{\xi}v = u.
\]
Thus $H(Y)(\xi)u$ solves the heat equation (\ref{eqn:FLheateqn}) with the same initial condition. By uniqueness of solutions of the heat equation we deduce that $H(Y)(\xi)u = e^{-sA_{t}^{2}(Y)(\xi)}u$ for every $u \in \Rg \mc{F}_{\xi}$. Moreover the image of $\mc{F}$ is dense cf. \cite[Proposition 3.15]{tony}, and the heat operators are continuous, so we conclude that $H(Y)(\xi) = e^{-sA_{t}^{2}(Y)(\xi)}$ as claimed.
\end{proof}

\subsection{Renormalized supertrace defect formula}

Let $X \to B$ be a Riemannian fiber bundle with closed fibers and let $E \to X$ be a family of vector bundles. Given a vertical family of smoothing operators $H = (H^{z})_{z\in B}$ on $\mb{E} = E \otimes \pi^{*}\Lambda T^{*}B \to X$ we have a family of smooth integral kernels $K_{H}^{z} \in \Gamma(X_{z} \times X_{z},\Lambda T_{z}^{*}B \otimes E_{z}\boxtimes E_{z}^{*})$ such that
\[
(H^{z}u)(x) = \int_{X_{z}} K_{H}^{z}(x,x')u(x')\,dx'.
\]
for any section $u \in \Gamma(X;E)$. In particular, for any $z \in B$ and $x \in X_{z}$ we have $K_{H}^{z}(x,x) \in \Lambda T^{*}_{z}B \otimes (\End E_{z})|_{x}$ so it makes sense to take the usual matrix trace of the endomorphism part. Thus the trace of $H$ yields a differential form on $B$ defined by fiber integrals of pointwise traces,
\begin{align*}
&(\Tr H)(z) = \int_{X_{z}}\tr_{E}(K_{H}^{z}(x,x))\,dx \\
&\mbox{ i.e., \ } \Tr H = \int_{X/B} \tr_{E}(K_{H}(x,x))\,dx.
\end{align*}
Moreover, if $E \to X$ is a family of $\mb{Z}_{2}$-graded vector bundles, then we get an induced $\mb{Z}_{2}$-grading on $\mb{E}$ given by $\mb{E} = \mb{E}^{+}\oplus\mb{E}^{-}$ with
\begin{comment}
\[
\mb{E}^{\pm} = \pi^{*}\Lambda T^{*}B \otimes E^{\pm}.
\]
\end{comment}
\begin{align*}
&\mb{E}^{+} = \left(\pi^{*}\Lambda^{\mathrm{ev}}T^{*}B \otimes E^{+}\right) \oplus \left(\pi^{*}\Lambda^{\mathrm{odd}}T^{*}B \otimes E^{-}\right), \\
&\mb{E}^{-} = \left(\pi^{*}\Lambda^{\mathrm{ev}}T^{*}B \otimes E^{-}\right) \oplus \left(\pi^{*}\Lambda^{\mathrm{odd}}T^{*}B \otimes E^{+}\right).
\end{align*}
With respect to this grading we define the supertrace of $H$ by
\[
\Str(H) = \Tr(H^{+}) - \Tr(H^{-})
\]
where $H^{\pm}$ are the restrictions of $H$ to $\mb{E}^{\pm}$. Thus the supertrace of $H$ is a differential form on $B$ given by
\begin{align*}
& \Str(H): B \to \Lambda T^{*}B \\
& \Str(H)(z) = \Str_{E|_{M_{z}}}(H_{z}) = \Tr_{E|_{M_{z}}}(H_{z}^{+}) - \Tr_{E|_{M_{z}}}(H_{z}^{-}).
\end{align*}
See \cite[pp. 309-314]{bgv} and \cite[Lemma 9.14]{bgv} for reference.

In general when the fiber bundle $M \to B$ is non-compact, the (super)trace of a smoothing operator may not be well-defined, and asymptotically periodic smoothing operators in particular need not be trace-class. In the context of end-periodic manifolds there is a coherent way of renormalizing the trace, see \cite[\S 3.1]{mrs1} where they define the end-periodic renormalized trace in the single-operator case. We will adapt this renormalized trace to the family context and show that it is well-defined on the space of asymptotically periodic smoothing operators.

Let $M \to B$ be an end-periodic fiber bundle, and $E \to M$ a family of end-periodic Clifford modules. For any vertical family of asymptotically periodic smoothing operators $H$ we define the renormalized trace
\[
\rTr H = \lim_{N\to\infty}\left[\int_{Z_{N}/B}\tr(K_{H}(x,x))\,dx - (N+1)\int_{W_{0}/B}\tr(K_{N(H)}(x,x))\,dx\right]
\]
where $Z_{N}$ is the submanifold of $M$ composed of the compact part $Z$ together with the first $N$ segments of the end of $M$. The renormalized supertrace $\rStr H$ is then obtained using the same formula but with the pointwise $\str$ instead of $\tr$. Equivalently, $\rStr$ is also given by
\[
\rStr H = \rTr(H^{+}) - \rTr(H^{-}).
\]
The renormalized supertrace $\rStr$ generalizes the $b$-supertrace of \cite{mp} equation (12.1) (also see \cite[Definition 4.64]{melrose}). Much like other renormalized traces, the renormalized trace $\rTr$ does not vanish on commutators; rather, we have a \emph{trace defect formula} which expresses $\rTr[P,Q]$ solely in terms of the indicial families of $P$ and $Q$. The formula of \cite[Theorem 4.6]{mrs1} explains the renormalized trace defect for a single smoothing operator of the form $D^{m}e^{-tD^2}$ where $D$ is an end-periodic Dirac operator. Similarly, in the $\mb{Z}_{2}$-graded context, the renormalized supertrace $\rStr$ does not vanish on supercommutators, but there is a supertrace defect formula that expresses $\rStr[P,Q]$ in terms of indicial families. The $\bStr$ defect for families of $\mb{Z}_{2}$-graded smoothing operators in the category of manifolds with cylindrical ends is explained by \cite[Proposition 9]{mp}.

The renormalized (super)trace is well-defined for families of asymptotically periodic smoothing operators.
\begin{proposition}
Let $M \to B$ be an end-periodic fiber bundle and $E \to M$ a family of end-periodic vector bundles. If $H$ is a vertical family of asymptotically periodic smoothing operators on $E \to M$, then the limit as $N \to \infty$ in the definition of $\rTr(H)$ converges to a well-defined differential form $\rTr(H) \in \Omega^{\bigcdot}(B)$. 
\end{proposition}
\begin{proof}
This follows just as in \cite[Lemma 3.1]{mrs1} using the fact that our asymptotically periodic operators satisfy the estimate (AP2) by definition.
\end{proof}

Recall that in the $\mb{Z}_{2}$-graded setting the bracket $[P,Q]$ is a supercommutator, cf. Remark \ref{remark:supercomm}. The following proposition explains the renormalized supertrace defect in the end-periodic $\mb{Z}_{2}$-graded family context.
\begin{proposition}[Supertrace defect formula]\label{strdefect}
Let $M \to B$ be an end-periodic fiber bundle and $E \to M$ a family of end-periodic vector bundles. Let $P$ and $Q$ be vertical families of asymptotically periodic smoothing operators. Then:
\begin{align*}
\rStr[P,Q] &= \frac{1}{2\pi i}\oint_{|\xi|=1}\int_{W_{0}/B}f(x)\str\left(K_{[P_{\xi},Q_{\xi}]}(x,x)\right)dx\frac{d\xi}{\xi} \\
& \hspace{5mm} - \frac{1}{2\pi i}\oint_{|\xi|=1} \Str\left(\frac{\partial P_{\xi}}{\partial \xi} Q_{\xi}\right) d\xi
\end{align*}
where the bracket is understood to be a supercommutator, i.e.,
\[
\str\left(K_{[P_{\xi},Q_{\xi}]}\right) = \str\left(K_{P_{\xi}Q_{\xi}}-(-1)^{|P||Q|}K_{Q_{\xi}P_{\xi}}\right).
\]
\end{proposition}
\begin{proof}
Following the proof of the trace commutator formula \cite[Theorem 4.6]{mrs1} in the family context using the estimates (AP1), (AP2) and (AP3) with the ungraded chiral parts of $P$ and $Q$, we deduce the $\rTr$ defect formula in the family context,
\begin{align*}
\rTr[P^{-},Q^{+}] &= \frac{1}{2\pi i}\oint_{|\xi|=1}\int_{W_{0}/B}f(x)\tr\left(K_{P_{\xi}^{-}Q_{\xi}^{+}}(x,x)-K_{Q_{\xi}^{+}P_{\xi}^{-}} (x,x)\right)dx\,\frac{d\xi}{\xi} \\
&\hspace{5mm} - \frac{1}{2\pi i}\oint_{|\xi|=1}\Tr\left(\frac{\partial P_{\xi}^{-}}{\partial\xi}Q_{\xi}^{+}\right)d\xi
\end{align*}
and similarly for $\rTr[P^{+},Q^{-}]$.

Now our supertrace defect formula is proved by breaking the operators $P$ and $Q$ into their chiral parts, then applying the preceding formula to each part. If $P$ and $Q$ have opposite parities then $[P,Q]$ is odd and both sides of the formula vanish, so it suffices to assume that $P$ and $Q$ are both even or both odd. If $P$ and $Q$ are both odd, we have
\[
\rStr[P,Q] = \rTr[P^{-},Q^{+}] - \rTr[P^{+},Q^{-}]
\]
and applying the family version of the $\rTr$ defect formula to each term yields
\begin{align*}
\rStr[P,Q] &= \frac{1}{2\pi i}\oint_{|\xi|=1}\int_{W_{0}/B}f(x)\tr(P_{\xi}^{-}Q_{\xi}^{+}(x,x) - Q_{\xi}^{+}P_{\xi}^{-}(x,x))\,dx\frac{d\xi}{\xi} \\
& \hspace{5mm} - \frac{1}{2\pi i}\oint_{|\xi|=1}\Tr\left(\frac{\partial P_{\xi}^{-}}{\partial\xi} Q_{\xi}^{+}\right) d\xi \\
& \hspace{5mm} - \frac{1}{2\pi i}\oint_{|\xi|=1}\int_{W_{0}/B}f(x)\tr(P_{\xi}^{+}Q_{\xi}^{-}(x,x) - Q_{\xi}^{-}P_{\xi}^{+}(x,x))\,dx\frac{d\xi}{\xi} \\
& \hspace{5mm} + \frac{1}{2\pi i}\oint_{|\xi|=1}\Tr\left(\frac{\partial P_{\xi}^{+}}{\partial\xi} Q_{\xi}^{-}\right) d\xi.
\end{align*}
Note that
\begin{align*}
\str(P_{\xi}Q_{\xi}) &= \tr(P_{\xi}^{-}Q_{\xi}^{+}) - \tr(P_{\xi}^{+}Q_{\xi}^{-}) \\
\str(Q_{\xi}P_{\xi}) &= \tr(Q_{\xi}^{-}P_{\xi}^{+}) - \tr(Q_{\xi}^{+}P_{\xi}^{-}) \\
\Str\left(\frac{\partial P_{\xi}}{\partial \xi}Q_{\xi}\right) &= \Tr\left(\frac{\partial P_{\xi}^{-}}{\partial \xi}Q_{\xi}^{+}\right) - \Tr\left(\frac{\partial P_{\xi}^{+}}{\partial \xi}Q_{\xi}^{-}\right)
\end{align*}
so grouping together terms in the preceding formula according to this pattern yields the statement of the proposition. The case where $P$ and $Q$ are both even operators is similar.
\end{proof}

In order to establish the transgression formula we will also need to apply the supertrace defect formula to supercommutators of the form $[A_{t},Q]$ where $Q$ is an asymptotically periodic smoothing operator. Since $A_{t}$ is not a smoothing operator, this warrants some justification.
\begin{proposition}\label{prop:alsoholds}
Let $A_{t}$ be the rescaled end-periodic Bismut superconnection. The $\rStr$ defect formula also holds in the case where $Q$ is the asymptotically periodic smoothing operator $\mc{Q} = \dt{A}_{t}e^{-A_{t}^{2}}$ and $P = A_{t}$.
\end{proposition}
\begin{proof}
For any $s >0$ consider the smoothing operator $P_{s} = A_{t}e^{-sA_{t}^{2}}$, which is asymptotically periodic by Proposition \ref{prop:moreapops}. We can apply the supertrace defect formula of Proposition \ref{strdefect} to obtain a defect formula for $\rStr[P_{s},Q]$. Since the integral kernels are continuous in $s$, and the limit as $N \to \infty$ in the definition of $\rTr(P_{s}Q)$ converges uniformly in $s$ on bounded time intervals, the renormalized trace $\rTr(P_{s}Q)$ is a continuous function of $s > 0$. Similarly, $\rTr(QP_{s})$ is continuous in $s$, hence $\rStr[P_{s},Q]$ is also continuous in $s$. Therefore, we can interchange limits to calculate
\[
\rStr[A_{t},Q] = \lim_{s \to 0^{+}} \rStr[P_{s},Q]
\]
which yields the desired result. 
\end{proof}

\subsection{Uniformity for the rescaled heat operator}
In the single-operator case, the estimate of \cite[Corollary 10.8]{mrs1} ensures that the renormalized supertrace $\rStr(e^{-tD^2})$ converges uniformly on bounded time intervals $t \in (0, T]$. This uniform convergence justifies interchanging limits to compute the short-time limit of $\rStr(e^{-tD^2})$ in \cite[Proposition 5.2]{mrs1}. The analogue of \cite[Corollary 10.8]{mrs1} in the family context is Proposition \ref{prop:ap3}, which similarly guarantees uniform convergence of the renormalized supertrace $\rStr(e^{-sA^2})$ for $s \in (0, T]$. However, when computing the short-time limit of the renormalized Chern character $\rch(A_{t}) = \rStr(e^{-A_{t}^2})$, we fix $s = 1$ and let $t \to 0^+$. Unlike the single-operator case, where the exponent is simply $tD^2$, the curvature $A_{t}^{2}$ of the rescaled Bismut superconnection includes additional $t$-dependent terms with distinct scaling.

In other words, although we have good control over the $s$-dependence in estimates for $e^{-sA^2}$, we require refined estimates for $e^{-sA_{t}^2}$ that control the additional $t$-dependence. Crucially, we need these estimates to yield uniform convergence of $\rch(A_{t})$ on bounded time intervals $t \in (0, T]$ so that we can again justify interchanging limits when taking $t \to 0^+$ to compute the short-time limit of $\rch(A_{t})$. We will establish this in Proposition \ref{uniformityprop} below. In \S\ref{section:short-time} we will use Proposition \ref{uniformityprop} to show that the short-time limit of $\rch(A_{t})$ converges to the integral of a local index form.

\begin{remark}
Throughout this section we will use the notation
\[
a(t) = O(t^{\infty})
\]
to mean that $a(t)/t^{q} \to 0$ as $t \to 0^{+}$ for any $q > 0$. We will also introduce the notation
\[
a(t) \lessim b(t)
\]
to mean that $a(t) \leq C(t)b(t)$ for some factor $C(t) > 0$ which converges to a finite value as $t \to 0^{+}$.
\end{remark}

%% Using the coarea formula in the form of Corollary 2.3 here
%% https://www3.nd.edu/~lnicolae/Coarea.pdf

\begin{lemma}\label{gaussianlemma}
Let $M$ be a manifold of bounded geometry, with injectivity radius $\iota = \mathrm{inj}(M) > 0$. Let $B_{\iota}(m)$ denote the geodesic ball of radius $\iota$ centered at $m \in M$. Fix any $\gamma > 0$. There exists constants $C,\beta > 0$ independent of $m \in M$ such that, as $t\to 0^{+}$,
\begin{align*}
&\mathrm{(i)} \hspace{3mm} \int_{B_{\iota}(m)} e^{-\gamma d^{2}(m,z)/t}\,dz \leq C\left(\frac{t}{\gamma}\right)^{1/2}e^{\frac{\beta^{2}t}{4\gamma}}\int_{v(0)}^{v(\iota)}e^{-v^{2}}\,dv = O(t^{1/2}) \\
&\mathrm{(ii)} \hspace{3mm} \int_{M\setminus B_{\iota}(m)}e^{-\gamma d^{2}(m,z)/t}\,dz \leq C\left(\frac{t}{\gamma}\right)^{1/2}e^{\frac{\beta^{2} t}{4\gamma}}\int_{v(\iota)}^{\infty}e^{-v^2}\,dv = O(t^{\infty})
\end{align*}
where we define $v(r) = \left(\frac{\gamma}{t}\right)^{1/2}\left(r-\frac{\beta t}{2\gamma}\right)$. In particular this implies that 
\[
\int_{M}e^{-\gamma d^{2}(m,z)/t}\,dz \leq O(t^{1/2})
\]
as $t \to 0^{+}$.
\end{lemma}
\begin{proof}
Since $M$ has bounded geometry, it has positive injectivity radius $\iota > 0$. We decompose the integral over $M$ into two parts,
\[
\int_{M} e^{-\gamma d^{2}(m,z)/t}\,dz = \int_{B_{\iota}(m)}e^{-\gamma d^{2}(m,z)/t}\,dz + \int_{M\setminus B_{\iota}(m)} e^{-\gamma d^{2}(m,z)/t}\,dz.
\]
We can use the coarea formula to calculate the integral over $M\setminus B_{\iota}(m)$ as follows. Let $S_{r}(m) = \partial B_{r}(m)$ denote the sphere of radius $r$ centered at $m$ inside $M$, which is the preimage of $r$ under the Riemannian distance function $d(m,\cdot)$. Note that $d(m,\cdot)$ is smooth almost everywhere, because it is smooth away from the cut locus of $m$, which has measure zero. Moreover, the rate of increase of distance along a geodesic is exactly $|\nabla d(m,z)| = 1$. Therefore, by the coarea formula,
\begin{align*}
\int_{M\setminus B_{\iota}(m)}e^{-\gamma d^{2}(m,z)/t}\,dz = \int_{\iota}^{\infty}\int_{S_{r}(m)}\frac{e^{-\gamma r^2/t}}{|\nabla d(m,z)|}\,d\sigma_{r}\,dr = \int_{\iota}^{\infty}\int_{S_{r}(m)}e^{-\gamma r^2/t}\,d\sigma_{r}\,dr
\end{align*}
where $d\sigma_{r}$ is the induced surface measure. The Bishop-Gromov inequality implies that geodesic balls have at most exponential surface area growth, i.e., $\vol(S_{r}(m)) \leq Ce^{\beta r}$ for some constants $C,\beta > 0$, therefore
\begin{align*}
\int_{\iota}^{\infty}\int_{S_{r}(m)}e^{-\gamma r^2/t}\,d\sigma_{r}\,dr &\leq C\int_{\iota}^{\infty}e^{-\gamma r^2/t}e^{\beta r}\,dr = Ce^{\frac{\beta^{2}t}{4\gamma}}\int_{\iota}^{\infty}e^{-\frac{\gamma}{t}\left(r-\frac{\beta t}{2\gamma}\right)^{2}}\,dr
\end{align*}
where we complete the square in the exponent. Making the change of variables $v =\left(\frac{\gamma}{t}\right)^{1/2}\left(r-\frac{\beta t}{2\gamma}\right)$ yields the bound (ii) in the statement of the proposition. Since $v(\iota) \to \infty$ as $t \to 0^{+}$, the Gaussian integral over the tail $\{v \geq v(\iota)\}$ decays like $O(t^{\infty})$ as $t \to 0^{+}$ by the well-known estimate
\[
\int_{a}^{\infty} e^{-v^2}\,dv \leq \frac{e^{-a^2}}{a\sqrt{\pi}}
\]
and so the upper bound decays like $O(t^{\infty})$ as $t\to 0^{+}$. 

The integral over $B_{\iota}(m)$ can be calculated in the same way, which yields the estimate
\begin{align*}
\int_{B_{\iota}(m)} e^{-\gamma d^{2}(m,z)/t}\,dz \leq C\left(\frac{t}{\gamma}\right)^{1/2}e^{\frac{\beta^{2}t}{4\gamma}}\int_{v(0)}^{v(\iota)}e^{-v^2}\,dv = O(t^{1/2})
\end{align*}
where $v(0) \to 0$ and $v(\iota) \to \infty$ as $t\to 0^{+}$. The Gaussian integral converges to $\sqrt{\pi}/2$, from which we obtain the $O(t^{1/2})$ decay as $t\to 0^{+}$.
\end{proof}

\begin{proposition}\label{uniformityprop}
Let $A$ be the end-periodic Bismut superconnection adapted to a family of end-periodic Dirac operators $D$, with normal operator $\tildhat{A} = p^{*}A(Y)$. Let $A_{t}$ denote the rescaled Bismut superconnection. Then we have
\[
|K_{e^{-A_{t}^{2}}}(x,x) - K_{e^{-\tildhat{A}_{t}^{2}}}(x,x)| \leq O(t^{\infty})
\]
for all $x \in \bigsqcup_{k\geq 1}W_{k}$ and $t \in (0,T]$. Consequently, the limit
\[
\rch(A_{t}) = \rStr(e^{-A_{t}^2}) = \lim_{N \to \infty} s_{N}(t),
\]
converges uniformly on bounded time intervals $t \in (0,T]$.
\end{proposition}
\begin{proof}
First we explain how the convergence of the renormalized supertrace follows from the estimate in question. Indeed, the same procedure as in the proof of \cite[Lemma 3.1]{mrs1} shows how the convergence of the renormalized supertrace $\rStr(e^{-A_{t}^{2}})$ translates to the convergence of the series
\[
\sum_{k=0}^{\infty}\int_{W_{0}/B}\tr(K_{e^{-A_{t}^{2}}}(x+k,x+k)-K_{e^{-\tildhat{A}_{t}^{2}}}(x+k,x+k))\,dx.
\]
The estimate in the statement of Proposition \ref{uniformityprop} implies in particular that this series converges uniformly on bounded time intervals $t \in (0,T]$. 

The Lichnerowicz formula for the vertical family of Dirac operators (applied on each fiber) yields
\[
D^2 = (\con^{E})^{*}(\con^{E}) + \frac{1}{4}S_{M/B} + \frac{1}{2}\sum_{\alpha,\beta}\cl(e^{\alpha})\cl(e^{\beta})K_{E}'(e_{\alpha},e_{\beta})
\]
where $(e^{\alpha})$ is a vertical orthonormal frame for $T^{*}(M/B)$. Comparing with the Lichnerowicz formula (\ref{eqn:lich}) for the Bismut superconnection, we see that the difference $R = A^{2} - D^{2}$ is an endomorphism of $\mathbb{E}$ given in terms of the curvature (and covariant derivatives of the curvature), and involves no derivatives. In particular, since $E$ and $M$ have bounded geometry, the components of $R$ are all uniformly bounded.  Rescaling via (\ref{eqn:rescaledsuperconn}) yields $A_{t}^{2} = tD^{2} + t^{1/2}\delta_{t}R$, and we will write $R_{t} = t^{-1/2}\delta_{t}R$ for brevity. The precise structure of the operator $R_{t}$ is not important, the only thing we will use here is the fact that
\begin{equation}\label{eqn:curvaturebound}
|R_{t}| \leq r(t^{-1}) = \sum_{j=1}^{4}c_{j}t^{-j/2}
\end{equation}
for some constants $c_{j} > 0$, using the uniform boundedness of $R$. According to the Volterra series for the heat operator, cf. \cite{bgv} equation (2.5), we have
\[
e^{-A_{t}^2} = e^{-tD^2} + \sum_{m \geq 1} (-t)^{m}I_{m}
\]
where $I_{m}$ is the operator given by
\[
I_{m} = \int_{\Delta^{m}}e^{-\sigma_{0}tD^2}R_{t}e^{-\sigma_{1}tD^2}R_{t}\cdots R_{t}e^{-\sigma_{m}tD^2}\,d\sigma
\]
where $\Delta_{m}$ is the standard $m$-simplex
\[
\Delta_{m} = \{(\sigma_{0},\ldots,\sigma_{m}) \in \mb{R}^{m+1}: \sum_{j=0}^{m}\sigma_{j} = 1, \sigma_{j} \geq 0\}.
\]
Similarly, for the heat operator associated with the Bismut superconnection $\tildhat{A}$, we have
\[
e^{-\tildhat{A}_{t}^2} = e^{-t\tildhat{D}^2} + \sum_{m \geq 1} (-t)^{m} \tildhat{I}_{m}.
\]
The sum over $m \geq 1$ is finite because $I_{m}$ is a sum of terms of degree at least $m$ with respect to the grading of the differential forms on $B$. For the remainder of the proof we will fix the notation $\mc{K}(t,x,y) = K_{e^{-tD^{2}}}(x,y)$ and $\tildhat{\mc{K}}(t,x,y) = K_{e^{-t\tildhat{D}^{2}}}(x,y)$. We seek to estimate the difference of heat kernels
\begin{align*}
|K_{e^{-A_{t}^{2}}}(x,x) - K_{e^{-\tildhat{A}_{t}^{2}}}(x,x)| &\leq |\mc{K}(t,x,x) - \tildhat{\mc{K}}(t,x,x)| \\
&\hspace{4mm} + \sum_{m \geq 1} t^{m} |K_{I_{m}}(t,x,x) - K_{\tildhat{I}_{m}}(t,x,x)| \numberthis \label{eqn:uniformity1}
\end{align*}
for $x \in \bigsqcup_{k\geq 1}W_{k}$ and $t \in (0,T]$. The estimate of \cite[Corollary 10.8]{mrs1} yields
\begin{equation}\label{eqn:uniformity2}
|\mc{K}(t,x,x) - \tildhat{\mc{K}}(t,x,x)| \leq Ce^{-\gamma d^{2}(x,W_{0})/t} = O(t^{\infty})
\end{equation}
for all $x \in \bigsqcup_{k\geq 1}W_{k}$ and $t \in (0,T]$, so it remains to estimate the sum over $m\geq 1$. Consider the first term $m=1$. We can express the difference $I_{1} - \tildhat{I}_{1}$ as a telescoping sum,
\begin{align*}
I_{1} - \tildhat{I}_{1} &= \int_{0}^{1}e^{-\sigma tD^2}R_{t}e^{-(1-\sigma)tD^2}\,d\sigma - \int_{0}^{1}e^{-\sigma t\tildhat{D}^2}\tildhat{R}_{t}e^{-(1-\sigma)t\tildhat{D}^2}\,d\sigma\\
&= \int_{0}^{1}(e^{-\sigma tD^2}-e^{-\sigma t\tildhat{D}^2})R_{t}e^{-(1-\sigma)tD^2}\,d\sigma + \int_{0}^{1}e^{-\sigma t\tildhat{D}^{2}}(R_{t}-\tildhat{R}_{t})e^{-(1-\sigma)tD^2}\,d\sigma \\
&\hspace{6mm} + \int_{0}^{1}e^{-\sigma t\tildhat{D}^2}\tildhat{R}_{t}(e^{-(1-\sigma)tD^2}-e^{-(1-\sigma)t\tildhat{D}^{2}})\,d\sigma
\end{align*}
but $R_{t} = \tildhat{R}_{t}$ along $\en(M)$ because $A^2$ is an end-periodic differential operator, so the middle term vanishes and we are left with
\begin{align*}
I_{1} - \tildhat{I}_{1} &= \int_{0}^{1}(e^{-\sigma tD^2}-e^{-\sigma t\tildhat{D}^2})R_{t}e^{-(1-\sigma)tD^2}\,d\sigma \\
&\hspace{4mm} + \int_{0}^{1}e^{-\sigma t\tildhat{D}^2}R_{t}(e^{-(1-\sigma)tD^2}-e^{-(1-\sigma)t\tildhat{D}^{2}})\,d\sigma.
\end{align*}
Hence the difference of integral kernels is bounded above by
\begin{align*}
&|K_{I_{1}}(t,x,x) - K_{\tildhat{I}_{1}}(t,x,x)| \\
&\hspace{4mm}\leq \int_{0}^{1}\int_{M}|\mc{K}(\sigma t,x,y)-\tildhat{\mc{K}}(\sigma t,x,y)|\cdot|R_{t}\mc{K}((1-\sigma)t,y,x)|\,dy\,d\sigma \\
&\hspace{6mm} + \int_{0}^{1}\int_{M}|\tildhat{\mc{K}}(\sigma t,x,y)R_{t}|\cdot|\mc{K}((1-\sigma)t,y,x)-\tildhat{\mc{K}}((1-\sigma)t,y,x)|\,dy\,d\sigma \\
&= (*) + (**).
\end{align*}
Owing to the clear symmetry, the same argument can be used to estimate both terms. We focus on the first term $(*)$ for the moment. Using the off-diagonal estimate \cite[Proposition 10.9]{mrs1} we have
\[
|\mc{K}(\sigma t,x,y)-\tildhat{\mc{K}}(\sigma t,x,y)| \leq C_{1}e^{-\gamma d^{2}/\sigma t} \leq C_{1}e^{-\gamma d^{2}/t} 
\]
where $d = \min\{d(x,W_{0}),d(y,W_{0})\}$. Moreover, applying \cite[Proposition 10.2]{mrs1} together with the curvature bound (\ref{eqn:curvaturebound}), we have
\begin{align*}
|R_{t}\mc{K}((1-\sigma)t,y,x)| &\leq C_{2}r(t^{-1})e^{-\alpha(1-\sigma)t}t^{-n/2}e^{-\gamma d^{2}(x,y)/(1-\sigma)t} \\
&\leq C_{2}r(t^{-1})e^{-\alpha(1-\sigma)t}t^{-n/2}e^{-\gamma d^{2}(x,y)/t} 
\end{align*}
where $d(x,y)$ is the Riemannian distance on $M$. Thus the term $(*)$ is estimated by
\begin{align*}
(*) &\leq C_{3}r(t^{-1})t^{-n/2}\int_{0}^{1}\int_{M}e^{-\gamma d^{2}/t}e^{-\alpha(1-\sigma)t}e^{-\gamma d^{2}(x,y)/t}\,dy\,d\sigma \\
&= C_{3}\left(\frac{1-e^{-\alpha t}}{\alpha t}\right)r(t^{-1})t^{-n/2}\int_{M}e^{-\gamma d^{2}/t}e^{-\gamma d^{2}(x,y)/t}\,dy \\
&\lessim r(t^{-1})t^{-n/2}\int_{M}e^{-\gamma d^{2}/t}e^{-\gamma d^{2}(x,y)/t}\,dy \numberthis \label{eqn:uniformity3}
\end{align*}
where we note that the factor $(1-e^{-\alpha t})/(\alpha t)$ approaches $1$ as $t \to 0^{+}$. In order to examine this integral over $M$, we split $M$ into two submanifolds:
\begin{align*}
&M_{0}(x) = \{y \in M : d(y,W_{0}) \leq d(x,W_{0})\} \\
&M_{1}(x) = \{y \in M : d(y,W_{0}) \geq d(x,W_{0})\}
\end{align*}
so that $d^2 = d^2(y,W_{0})$ on $M_{0}(x)$ and $d^2 = d^2(x,W_{0})$ on $M_{1}(x)$. Then breaking (\ref{eqn:uniformity3}) into two parts, we have
\begin{align*}
(*) &\lessim r(t^{-1})t^{-n/2}\int_{M} e^{-\gamma d^{2}/ t}e^{-\gamma d^{2}(x,y)/t}\,dy \\
&= r(t^{-1})t^{-n/2}\left(\int_{M_{0}(x)}e^{-\gamma d^2(y,W_{0})/ t}e^{-\gamma d^{2}(x,y)/t}\,dy\right) \\
&\hspace{6mm} + r(t^{-1})t^{-n/2}e^{-\gamma d^2(x,W_{0})/t}\left(\int_{M_{1}(x)}e^{-\gamma d^{2}(x,y)/t}\,dy\right).
\end{align*}
For $y \in M_{0}(x)$ we have $d^{2}(y,W_{0}) + d^{2}(x,y) \geq \frac{1}{2}d^{2}(x,W_{0})$ by the triangle inequality, so we can estimate
\begin{align*}
\int_{M_{0}(x)} e^{-\gamma d^2(y,W_{0})/t}e^{-\gamma d^{2}(x,y)/t}\,dy \leq \vol(M_{0}(x))e^{-\gamma d^{2}(x,W_{0})/2t} = O(t^{\infty}).
\end{align*}
Lastly, using Lemma \ref{gaussianlemma} we can bound the integral over $M_{1}(x)$ by a function which decays like $O(t^{1/2})$ as $t \to 0^{+}$. Putting all of this together yields
\begin{align*}
|K_{I_{1}}(t,x,x) - K_{\tildhat{I}_{1}}(t,x,x)| &\lessim r(t^{-1})t^{-n/2}\cdot O(t^{\infty}) \\
&\hspace{4mm} + r(t^{-1})t^{-n/2}e^{-\gamma d(x,W_{0})/t}\cdot O(t^{1/2}) \\
&= O(t^{\infty})
\end{align*}
because $r(t^{-1})$ grows at most polynomially as $t\to 0^{+}$.

A similar argument holds for the remaining terms with $m > 1$ in (\ref{eqn:uniformity1}), just with a longer telescoping sum and with more terms of the same type already discussed. Thus for every $m \geq 1$ we have an estimate of the form
\[
|K_{I_{m}}(t,x,x) - K_{\tildhat{I}_{m}}(t,x,x)| \lessim p(t^{-1})\cdot O(t^{\infty}) = O(t^{\infty})
\]
where $p(t^{-1})$ grows at most polynomially as $t \to 0^{+}$. Combining this estimate with (\ref{eqn:uniformity1}) and (\ref{eqn:uniformity2}) we obtain
\[
|K_{e^{-A_{t}^{2}}}(x,x) - K_{e^{-\tildhat{A}_{t}^{2}}}(x,x)| \leq O(t^{\infty})
\]
as claimed.
\end{proof}

\subsection{Special case: the b-supertrace defect}

A particularly important case is that of manifolds with cylindrical ends, where the renormalized trace coincides with the $b$-trace defined in \cite[\S 4.20]{melrose}. Melrose used the $b$-trace and its defect formula to give a proof of the Atiyah-Patodi-Singer index theorem for manifolds with $b$-metrics, wherein the eta invariant arises naturally from the transgression formula for the variation of $\bStr(e^{-tD^{2}})$. In the setting of \emph{families} of Dirac operators on manifolds with boundary, the analogous $b$-supertrace and its defect play a central role in \cite{mp}, where they are used to prove the index formula \cite[Theorem 1]{mp}. There, the Bismut–Cheeger eta form emerges naturally from the transgression formula for the variation of $\bStr(e^{-A_{t}^2})$, with $A_{t}$ the rescaled Bismut superconnection. In \S\ref{section:epeta} we will show that our end-periodic eta form reduces to the Bismut–Cheeger eta form in the cylindrical-end case. In this section we will introduce the $b$-trace and $b$-supertrace in preparation for that calculation.

Let $M \to B$ be an end-periodic fiber bundle and $E \to M$ a family of end-periodic Clifford modules. Suppose that the end of $M$ is cylindrical, meaning $\tildhat{Y} = \mb{R}\times \partial Z$, $Y = S^{1} \times \partial Z$, and $f(r,q) = r$. Recall that for any vertical family of asymptotically periodic smoothing operators $H$ we defined the renormalized trace
\[
\rTr H = \lim_{N\to\infty}\left[\int_{Z_{N}/B}\tr(K_{H}(x,x))\,dx - (N+1)\int_{W_{0}/B}\tr(K_{N(H)}(x,x))\,dx\right]
\]
where $Z_{N}$ is the submanifold of $M$ composed of the compact part $Z$ together with the first $N$ segments of the end of $M$. In the case where the end of $M$ is cylindrical, it is straightforward to see that the renormalized trace reduces to the $b$-trace $\bTr(H)$ of \cite[Definition 4.64]{melrose} and the renormalized supertrace reduces to the $b$-supertrace $\bStr(H)$ of equation (12.1) of \cite{mp}. We will use the notation $\bTr(H)$ when working on a manifold with cylindrical end.

As in the general end-periodic case, the $b$-trace defect $\bTr[P,Q]$ is expressed in terms of the indicial families of the operators $P$ and $Q$. The key difference in the cylindrical setting is that a slightly simpler indicial family can be used -- one that arises from the Mellin transform of the normal operator instead of the Fourier-Laplace transform -- and this is how the $\bTr$ and $\bStr$ defects are formulated in \cite[Proposition 5.9]{melrose} and \cite[Proposition 9]{mp}. The Mellin transform is given by 
\begin{align*}
&\mc{M}_{\lambda}: C_{c}^{\infty}(\tildhat{Y};\tildhat{E}) \to C^{\infty}(\partial Z;E_{\partial Z}) \\
&(\mc{M}_{\lambda}u)(q) = \int_{0}^{\infty} r^{-i\lambda}u(r,q)\,\frac{dr}{r}
\end{align*}
for any $\lambda \in \mb{R}$, cf. \cite[\S 5.1]{melrose}. This is just the Fourier transform after an exponential change of coordinates. Take any $P \in \epdiff^{s}(M;\mb{E})$ with normal operator $N(P) \in \mathrm{Diff}_{\pi}^{s}(\tildhat{Y};\tildhat{\mb{E}})$. For any $\lambda \in \mb{R}$ we define an operator $P(\partial Z)(\lambda) \in \mathrm{Diff}_{\pi}^{s}(\partial Z;\mb{E}_{\partial Z})$ by the relation
\[
\mc{M}_{\lambda}\circ N(P) = P(\partial Z)(\lambda)\circ \mc{M}_{\lambda}.
\]
In \cite{melrose} the family of operators \{$P(\partial Z)(\lambda)\}_{\lambda \in \mb{R}}$ is called the \emph{indicial family} of $P$. Thus we have so far defined two indicial families, $P_{\xi}(Y) = P(Y)(\xi)$ and $P_{\lambda}(\partial Z) = P(\partial Z)(\lambda)$, obtained from the Fourier-Laplace transform and the Mellin transform respectively. When we want to distinguish the two we will say that $P_{\lambda}(\partial Z)$ is the indicial family induced by the Mellin transform, and $P_{\xi}(Y)$ is the indicial family induced by the Fourier-Laplace transform. The $b$-supertrace defect in the cylindrical end case is expressed in terms of the (Mellin transform induced) indicial family,
\begin{equation}\label{eqn:bstrdefect}
\bStr[P,Q] = -\frac{1}{2\pi i}\int_{\mb{R}}\Str\left(\frac{\partial}{\partial\theta}[P_{\theta}(\partial Z)]Q_{\theta}(\partial Z)\right)d\theta,
\end{equation}
cf. \cite[Proposition 9]{mp}. The following proposition explains the relation between the indicial families induced by the Fourier-Laplace and Mellin transforms.
\begin{proposition}\label{newprop}
For a family of end-periodic differential operators $P\in\epdiff^{\bigcdot}(M;\mb{E})$ with normal operator $N(P) = \tildhat{P}$ we have
\begin{equation}\label{neweqn}
K_{\tildhat{P}}(x,y) = \frac{1}{2\pi i}\oint_{S^{1}}\xi^{f(y)-f(x)}K_{P_{\xi}(Y)}(p(x),p(y))\frac{d\xi}{\xi}
\end{equation}
for any $x,y \in \tildhat{Y}$. 

When $M$ is a manifold with cylindrical end, i.e., with end modeled on $\tildhat{Y} = \mb{R} \times \partial Z$, in terms of the Mellin transform we have
\begin{equation}\label{neweqn2}
K_{\tildhat{P}}((r,q),(r',q')) = \frac{1}{2\pi}\int_{\mathrm{Im}\,\theta=a}e^{i(r-r')\theta}K_{P_{\theta}(\partial Z)}(q,q')\,d\theta,
\end{equation}
for any $a \in \mb{R}$ as in the Mellin inversion theorem \cite[Theorem 5.1]{melrose}. Moreover, in the cylindrical end case we have the identity
\begin{equation}\label{neweqn3}
\frac{1}{2\pi i}\oint_{S^{1}}\Tr(P_{\xi}(Y))\frac{d\xi}{\xi} = \frac{1}{2\pi}\int_{\mb{R}}\Tr(P_{\theta}(\partial Z))\,d\theta.
\end{equation}
\end{proposition}
\begin{proof}
The formula (\ref{neweqn}) follows from the same direct calculation as in the single-operator case, cf. \cite[Proposition 2.5]{mrs1}. The formula (\ref{neweqn2}) is derived from the same calculation again, but replacing the Fourier-Laplace transform with the Mellin transform. The identity (\ref{neweqn3}) is then obtained by restricting the integral kernel formulas to the diagonal $x=y$ and taking the trace.
\end{proof}

%% file: sections/5-Transgression-formula.tex
\section{Transgression formula and eta form}\label{section:transgression}

\subsection{Transgression formula for periodic ends}\label{section:trans2}
In this section we derive a transgression formula for the renormalized Chern character of the  rescaled end-periodic Bismut superconnection adapted to a family of end-periodic Dirac operators; i.e., a formula for the variation of the renormalized Chern character $\rch(A_{t})$. 

We first exhibit a broad outline of how the transgression formula is derived. By Duhamel's formula we have
\begin{align*}
\frac{d}{dt}\rStr(e^{-A_{t}^{2}}) &= -\rStr\left[A_{t},\dt{A_{t}}e^{-A_{t}^{2}}\right] - \int_{0}^{1}\rStr\left[e^{-sA_{t}^{2}},\frac{d}{dt}(A_{t}^{2})e^{-(1-s)A_{t}^{2}}\right]ds \\
&= -\rStr\left[A_{t}-A_{[1]},\dt{A_{t}}e^{-A_{t}^{2}}\right] - \rStr\left[A_{[1]},\dt{A_{t}}e^{-A_{t}^{2}}\right] \\
&\hspace{4mm} - \int_{0}^{1}\rStr\left[e^{-sA_{t}^{2}},\frac{d}{dt}(A_{t}^{2})e^{-(1-s)A_{t}^{2}}\right]ds.
\end{align*}
Applying the $\rStr$ defect formula to each of these three terms produces a total of six defect terms, all depending on the indicial family $A_{t}(Y)(\xi)$:
\begin{align*}
\frac{d}{dt}\rStr(e^{-A_{t}^{2}}) = (\alpha_{1}(t)-\alpha_{2}(t)) + (\beta_{1}(t)-\beta_{2}(t)) + (\gamma_{1}(t)-\gamma_{2}(t)).
\end{align*}
Our derivation of the transgression formula will be completed by showing that
\begin{align*}
\mathrm{(i)} \hspace{3mm} &\alpha_{1}(t) + \beta_{1}(t) + \gamma_{1}(t) = \frac{-1}{2\pi i}\frac{d}{dt}\oint_{|\xi|=1}\int_{W_{0}/B}f(x)\str\left(K_{e^{-A_{t}^{2}(Y)(\xi)}}(x,x)\right)dx\,\frac{d\xi}{\xi}, \\
\mathrm{(ii)} \hspace{3mm} &\alpha_{2}(t) + \gamma_{2}(t) = -\frac{1}{2}\epeta(t), \mbox{ and } \\
\mathrm{(iii)}  \hspace{3mm} &\beta_{2}(t) = \mbox{exact form}.
\end{align*}

\begin{remark} In the case of manifolds with cylindrical ends, the transgression formula for the $b$-Chern character $\bch(A_{t}) = \bStr(e^{-A_{t}^{2}})$ has $\alpha_{1} = \beta_{1} = \gamma_{1} = 0$, $\gamma_{2} = 0$, and the Bismut-Cheeger eta form is $\frac{1}{2}\widehat{\eta}(t) = -\alpha_{2}(t)$, cf. \cite[Proposition 11]{mp}.
\end{remark}

\noindent We start with the following calculation using Duhamel's formula:
\begin{comment}
\begin{align*}
\frac{d}{dt}\rch(A_{t}) &= \frac{1}{2\pi i}\frac{d}{dt}\oint_{|\xi|=1}\int_{W_{0}/B}f(x)\str\left(K_{e^{-A_{t}^{2}(Y)(\xi)}}(x,x)\right)dx\frac{d\xi}{\xi} \\
& \hspace{6mm} - \frac{1}{2}\epeta(t) - d_{B}\left(\frac{1}{2\pi i}\oint_{|\xi|=1}\Str\left(\frac{\partial \mc{Q}_{\xi}}{\partial\xi}\right)  d\xi\right).
\end{align*}
\end{comment}
\begin{align*}
\frac{d}{dt}\rch(A_{t}) &= \frac{d}{dt}\rStr\left(e^{-A_{t}^{2}}\right) \\
&= -\int_{0}^{1}\rStr\left(e^{-sA_{t}^{2}}\frac{d}{dt}(A_{t}^{2})e^{-(1-s)A_{t}^{2}}\right)ds \\
&= -\rStr\left(\frac{d}{dt}(A_{t}^{2})e^{-A_{t}^{2}}\right) - \int_{0}^{1}\rStr\left[e^{-sA_{t}^{2}},\frac{d}{dt}(A_{t}^{2})e^{-(1-s)A_{t}^{2}}\right]ds \\
&= -\rStr\left[A_{t},\dt{A_{t}}e^{-A_{t}^{2}}\right] - \gamma(t)
\end{align*}
where we have introduced the shorthand
\[
\gamma(t) = \int_{0}^{1}\rStr\left[e^{-sA_{t}^{2}},\frac{d}{dt}(A_{t}^{2})e^{-(1-s)A_{t}^{2}}\right]ds
\]
which is the correction that appears when we commute the two terms under $\rStr$. We will fix the following notation:
\begin{itemize}
\item $\mc{P} = A_{t} - A_{[1]}$ and $\mc{P}_{\xi} = \mc{P}(Y)(\xi)$.
\item $\mc{Q} = \dt{A_{t}}e^{-A_{t}^{2}}$ and $\mc{Q}_{\xi} = \mc{Q}(Y)(\xi)$.
\item $\displaystyle H_{1} = e^{-sA_{t}^{2}(Y)(\xi)}$ and $\displaystyle H_{2} = \frac{d}{dt}(A_{t}^{2}(Y)(\xi))e^{-(1-s)A_{t}^{2}(Y)(\xi)}$.
\end{itemize}
Also note that
\begin{align*}
\mc{P}_{\xi} &= \mc{P}(Y)(\xi) = A_{t}(Y)(\xi)-A_{[1]}(Y)(\xi), \\
\mc{Q}_{\xi} &= \mc{Q}(Y)(\xi) = \dt{\mc{P}_{\xi}}e^{-A_{t}^{2}(Y)(\xi)}.
\end{align*}
Returning to the time derivative of $\rch(A_{t})$ and splitting $A_{t} = (A_{t}-A_{[1]}) + A_{[1]}$, we have
\begin{align*}
\frac{d}{dt}\rch(A_{t}) &= -\rStr\left[A_{t},\dt{A_{t}}e^{-A_{t}^{2}}\right] - \gamma(t)\\
&= -\rStr\left[A_{t}-A_{[1]},\dt{A_{t}}e^{-A_{t}^{2}}\right] -\rStr\left[A_{[1]},\dt{A_{t}}e^{-A_{t}^{2}}\right] - \gamma(t) \\
&= -\rStr[\mc{P},\mc{Q}] - \rStr[A_{[1]},\mc{Q}] - \gamma(t) \\
&= -\alpha(t) - \beta(t) - \gamma(t) \numberthis \label{eqn:firsttrans}
\end{align*}
where we introduce the abbreviations
\begin{align*}
&\alpha(t) = \rStr[\mc{P},\mc{Q}] = \rStr\left[A_{t}-A_{[1]},\dt{A_{t}}e^{-A_{t}^{2}}\right], \\
&\beta(t) = \rStr[A_{[1]},\mc{Q}] = \rStr\left[A_{[1]},\dt{A_{t}}e^{-A_{t}^{2}}\right].
\end{align*}
Let us focus on the first term $\alpha(t)$ on the right-hand side of (\ref{eqn:firsttrans}). By the supertrace defect formula we have
\begin{align*}
\alpha(t) &= \frac{1}{2\pi i}\oint_{|\xi|=1}\int_{W_{0}/B}f(x)\str(K_{\mc{P}_{\xi}\mc{Q}_{\xi}}(x,x) + K_{\mc{Q}_{\xi}\mc{P}_{\xi}}(x,x))\,dx\frac{d\xi}{\xi} \\
& \hspace{5mm} - \frac{1}{2\pi i}\oint_{|\xi|=1} \Str\left(\frac{\partial \mc{P}_{\xi}}{\partial\xi} \mc{Q}_{\xi}\right)d\xi \\
&\stackrel{\mathrm{def}}{=} \alpha_{1}(t) - \alpha_{2}(t). \numberthis \label{eqn:secondtrans}
\end{align*}
The following lemma is the key step in understanding the term $\alpha_{1}(t)$.
\begin{lemma}\label{translemma}
We have
\begin{align*}
\str(K_{\mc{P}_{\xi}\mc{Q}_{\xi}} + K_{\mc{Q}_{\xi}\mc{P}_{\xi}}) &= -\frac{d}{dt}\str\left(K_{e^{-A_{t}^{2}(Y)(\xi)}}\right) - \str\left(K_{A_{[1]}(Y)(\xi)\mc{Q}_{\xi}} + K_{\mc{Q}_{\xi}A_{[1]}(Y)(\xi)}\right) \\
&\hspace{6mm} - \int_{0}^{1}\str(K_{H_{1}H_{2}}-K_{H_{2}H_{1}})\,ds
\end{align*}
where we have set
\[
H_{1} = e^{-sA_{t}^{2}(Y)(\xi)} \mbox{ \hspace{2mm} and \hspace{2mm} } H_{2} = \frac{d}{dt}(A_{t}^{2}(Y)(\xi))e^{-(1-s)A_{t}^{2}(Y)(\xi)}.
\]
\end{lemma}
\begin{proof}
By Duhamel's formula we have
\begin{align*}
\frac{d}{dt}\str\left(K_{e^{-A_{t}^{2}(Y)(\xi)}}\right) &= \str\left(K_{\frac{d}{dt}\left(e^{-A_{t}^{2}(Y)(\xi)}\right)}\right) \\
&= -\int_{0}^{1}\str(K_{H_{1}H_{2}})\,ds \\
&= -\int_{0}^{1}\str(K_{H_{2}H_{1}})\,ds - \int_{0}^{1}\str(K_{H_{1}H_{2}}-K_{H_{2}H_{1}})\,ds \\
&= -\str\left(K_{\frac{d}{dt}(A_{t}^{2}(Y)(\xi))e^{-A_{t}^{2}(Y)(\xi)}}\right) - \int_{0}^{1}\str(K_{H_{1}H_{2}}-K_{H_{2}H_{1}})\,ds.
\end{align*}
The first term on the right-hand side here is
\begin{align*}
&\str\left(K_{\frac{d}{dt}(A_{t}^{2}(Y)(\xi))e^{-A_{t}^{2}(Y)(\xi)}}\right) \\
&\hspace{4mm} = \str\left(K_{(A_{t}(Y)(\xi)\dt{\mc{P}}_{\xi}+\dt{\mc{P}}_{\xi}A_{t}(Y)(\xi))e^{-A_{t}^{2}(Y)(\xi)}}\right) \\
&\hspace{4mm} = \str\left(K_{((\mc{P}_{\xi}+A_{[1]}(Y)(\xi))\dt{\mc{P}}_{\xi}+\dt{\mc{P}}_{\xi}(\mc{P}_{\xi}+A_{[1]}(Y)(\xi)))e^{-A_{t}^{2}(Y)(\xi)}}\right) \\
&\hspace{4mm} = \str\left(K_{(\mc{P}_{\xi}\dt{\mc{P}}_{\xi}+\dt{\mc{P}}_{\xi}\mc{P}_{\xi})e^{-A_{t}^{2}(Y)(\xi)}}\right) +\str\left(K_{[A_{[1]}(Y)(\xi),\dt{\mc{P}}_{\xi}]e^{-A_{t}^{2}(Y)(\xi)}} \right) \\
&\hspace{4mm} \stackrel{\mathrm{def}}{=} (1a) + (1b).
\end{align*}
Thus
\[
\frac{d}{dt}\str\left(K_{e^{-A_{t}^{2}(Y)(\xi)}}\right) = -(1a) -(1b) - \int_{0}^{1}\str(K_{H_{1}H_{2}}-K_{H_{2}H_{1}})\,ds.
\]
Solving for (1a) yields
\[
(1a) = -\frac{d}{dt}\str\left(K_{e^{-A_{t}^{2}(Y)(\xi)}}\right) - (1b) - \int_{0}^{1}\str(K_{H_{1}H_{2}}-K_{H_{2}H_{1}})\,ds.
\]
On the other hand, the term $\alpha_{1}(t)$ in equation (\ref{eqn:secondtrans}) involves
\begin{align*}
\str(K_{\mc{P}_{\xi}\mc{Q}_{\xi}} + K_{\mc{Q}_{\xi}\mc{P}_{\xi}}) &= \str\left(K_{\mc{P}_{\xi}\dt{\mc{P}}_{\xi}e^{-A_{t}^{2}(Y)(\xi)}} + K_{\dt{\mc{P}}_{\xi}e^{-A_{t}^{2}(Y)(\xi)}\mc{P}_{\xi}}\right) \\
&= \str\left(K_{(\mc{P}_{\xi}\dt{\mc{P}}_{\xi} + \dt{\mc{P}}_{\xi}\mc{P}_{\xi})e^{-A_{t}^{2}(Y)(\xi)}}\right) + \str\left(K_{\dt{\mc{P}}_{\xi}[e^{-A_{t}^{2}(Y)(\xi)},\mc{P}_{\xi}]} \right) \\
&\stackrel{\mathrm{def}}{=} (1a) + (1c).
\end{align*}
Then plugging in our previous expression for (1a) here we get
\[
\str(K_{\mc{P}_{\xi}\mc{Q}_{\xi}} + K_{\mc{Q}_{\xi}\mc{P}_{\xi}}) = -\frac{d}{dt}\str\left(K_{e^{-A_{t}^{2}(Y)(\xi)}}\right) - (1b) + (1c) - \int_{0}^{1}\str(K_{H_{1}H_{2}}-K_{H_{2}H_{1}})\,ds.
\]
Now we look at the expression $-(1b) + (1c)$ and use the following algebraic fact: if $A$ is odd, $B$ is odd, and $C$ is even, then
\[
[B,A]C = A[B,C] + [B,AC].
\]
This identity yields
\begin{align*}
-(1b) &= -\str\left(K_{[A_{[1]}(Y)(\xi),\dt{\mc{P}}_{\xi}]e^{-A_{t}^{2}(Y)(\xi)}} \right) \\
&= -\str\left(K_{\dt{\mc{P}}_{\xi}\left[A_{[1]}(Y)(\xi),e^{-A_{t}^{2}(Y)(\xi)}\right]}\right) - \str\left( K_{\left[A_{[1]}(Y)(\xi),\dt{\mc{P}}_{\xi}e^{-A_{t}^{2}(Y)(\xi)} \right]}\right)
\end{align*}
and therefore
\begin{align*}
-(1b) + (1c) &= -\str\left(K_{[A_{[1]}(Y)(\xi),\dt{\mc{P}}_{\xi}]e^{-A_{t}^{2}(Y)(\xi)}} \right) + \str\left(K_{\dt{\mc{P}}_{\xi}[e^{-A_{t}^{2}(Y)(\xi)},\mc{P}_{\xi}]} \right)  \\
&= -\str\left(K_{\dt{\mc{P}}_{\xi}\left[A_{[1]}(Y)(\xi),e^{-A_{t}^{2}(Y)(\xi)}\right]}\right) - \str\left( K_{\left[A_{[1]}(Y)(\xi),\dt{\mc{P}}_{\xi}e^{-A_{t}^{2}(Y)(\xi)} \right]}\right) \\
&\hspace{6mm} + \str\left(K_{\dt{\mc{P}}_{\xi}[e^{-A_{t}^{2}(Y)(\xi)},\mc{P}_{\xi}]} \right) \\
&= \str\left(K_{\dt{\mc{P}}_{\xi}\left[e^{-A_{t}^{2}(Y)(\xi)},A_{[1]}(Y)(\xi)\right]}\right) - \str\left( K_{\left[A_{[1]}(Y)(\xi),\dt{\mc{P}}_{\xi}e^{-A_{t}^{2}(Y)(\xi)} \right]}\right) \\
&\hspace{6mm} + \str\left(K_{\dt{\mc{P}}_{\xi}[e^{-A_{t}^{2}(Y)(\xi)},\mc{P}_{\xi}]} \right),
\end{align*}
where in the first term on the last line, we have changed the sign by swapping the operators inside the supercommutator (noting that $A_{[1]}(Y)(\xi)$ is odd and $e^{-A_{t}^{2}(Y)(\xi)}$ is even). We can reunite the first and last terms here to obtain
\[
-(1b) + (1c) = \str\left(K_{\dt{\mc{P}}_{\xi}[e^{-A_{t}^{2}(Y)(\xi)},A_{[1]}(Y)(\xi) + \mc{P}_{\xi}]}\right) - \str \left(K_{\left[A_{[1]}(Y)(\xi),\dt{\mc{P}}_{\xi}e^{-A_{t}^{2}(Y)(\xi)} \right]}\right),
\]
but of course we have $A_{[1]}(Y)(\xi) + \mc{P}_{\xi} = A_{t}(Y)(\xi)$ by definition, so the first supercommutator on the right-hand side will vanish, and we are left with
\[
-(1b) + (1c) = -\str\left(K_{\left[A_{[1]}(Y)(\xi),\dt{\mc{P}}_{\xi}e^{-A_{t}^{2}(Y)(\xi)} \right]}\right) = -\str\left(K_{A_{[1]}(Y)(\xi)\mc{Q}_{\xi}} + K_{\mc{Q}_{\xi}A_{[1]}(Y)(\xi)}\right),
\]
which establishes the lemma.
\end{proof}
Returning to the term $\alpha_{1}(t)$ and applying Lemma \ref{translemma}, we have
\begin{align*}
\alpha_{1}(t) &= \frac{1}{2\pi i}\oint_{|\xi|=1}\int_{W_{0}/B}f(x)\str(K_{\mc{P}_{\xi}\mc{Q}_{\xi}}(x,x) + K_{\mc{Q}_{\xi}\mc{P}_{\xi}}(x,x))\,dx\frac{d\xi}{\xi} \\
&= \frac{1}{2\pi i}\oint_{|\xi|=1}\int_{W_{0}/B}f(x)\bigg[-\frac{d}{dt}\str\left(K_{e^{-A_{t}^{2}(Y)(\xi)}}(x,x)\right) \\
&\hspace{6mm} -\str\left(K_{A_{[1]}(Y)(\xi)\mc{Q}_{\xi}}(x,x) + K_{\mc{Q}_{\xi}A_{[1]}(Y)(\xi)}(x,x)\right) \\
&\hspace{6mm} -\int_{0}^{1}\str(K_{H_{1}H_{2}}(x,x)-K_{H_{2}H_{1}}(x,x))\,ds\bigg]\,dx\frac{d\xi}{\xi}.
\end{align*}
We compare this identity with
\begin{align*}
&\beta(t) = \rStr[A_{[1]},\mc{Q}] \\
&= \frac{1}{2\pi i}\oint_{|\xi|=1}\int_{W_{0}/B} f(x)\str\left(K_{A_{[1]}(Y)(\xi)\mc{Q}_{\xi}}(x,x) + K_{\mc{Q}_{\xi}A_{[1]}(Y)(\xi)}(x,x)\right)dx\frac{d\xi}{\xi} \\
&\hspace{6mm}-\frac{1}{2\pi i}\oint_{|\xi|=1} \Str\left(\frac{\partial}{\partial \xi}\left(A_{[1]}(Y)(\xi)\right)\mc{Q}_{\xi}\right)d\xi \\
&\stackrel{\mathrm{def}}{=} \beta_{1}(t) - \beta_{2}(t),
\end{align*}
and, noting that $H_{1}$ and $H_{2}$ are both even operators, 
\begin{align*}
\gamma(t) &= \int_{0}^{1}\rStr\left[e^{-sA_{t}^{2}},\frac{d}{dt}(A_{t}^{2})e^{-(1-s)A_{t}^{2}}\right]ds \\
&= \int_{0}^{1}\oint_{|\xi|=1}\int_{W_{0}/B} f(x)\str(K_{H_{1}H_{2}}(x,x)-K_{H_{2}H_{1}}(x,x))\,dx\frac{d\xi}{\xi}ds \\
&\hspace{6mm}-\frac{1}{2\pi i}\int_{0}^{1}\oint_{|\xi|=1} \Str\left(\frac{\partial H_{1}}{\partial \xi}H_{2}\right) d\xi\,ds \\
&\stackrel{\mathrm{def}}{=} \gamma_{1}(t) - \gamma_{2}(t).
\end{align*}
Thus, we see that
\begin{align*}
\alpha_{1}(t) &= -\frac{1}{2\pi i}\frac{d}{dt}\oint_{|\xi|=1}\int_{W_{0}/B}f(x)\str\left(K_{e^{-A_{t}^{2}(Y)(\xi)}}(x,x)\right)dx\frac{d\xi}{\xi} \\
&\hspace{8mm} - \beta_{1}(t) - \gamma_{1}(t).
\end{align*}
Substituting this expression for $\alpha_{1}(t)$ back into the transgression formula (\ref{eqn:firsttrans}), cancellations occur and we obtain
\begin{align*}
\frac{d}{dt}\rch(A_{t}) &= -\alpha(t) - \beta(t) - \gamma(t) \\
&= -\alpha_{1}(t) + \alpha_{2}(t) - \beta_{1}(t) + \beta_{2}(t) - \gamma_{1}(t) + \gamma_{2}(t) \\
&= \frac{1}{2\pi i}\frac{d}{dt}\oint_{|\xi|=1}\int_{W_{0}/B}f(x)\str\left(K_{e^{-A_{t}^{2}(Y)(\xi)}}(x,x)\right)dx\frac{d\xi}{\xi} \\
&\hspace{6mm} + \beta_{2}(t) + \alpha_{2}(t) + \gamma_{2}(t)
\end{align*}
which in more detail is 
\begin{align*}
\frac{d}{dt}\rch(A_{t}) &= \frac{1}{2\pi i}\frac{d}{dt}\oint_{|\xi|=1}\int_{W_{0}/B}f(x)\str\left(K_{e^{-A_{t}^{2}(Y)(\xi)}}(x,x)\right)dx\frac{d\xi}{\xi} \\
& \hspace{6mm} + \frac{1}{2\pi i}\oint_{|\xi|=1} \Str\left(\frac{\partial}{\partial \xi}\left(A_{[1]}(Y)(\xi)\right)\mc{Q}_{\xi}\right)d\xi \\
&\hspace{6mm}+\frac{1}{2\pi i}\oint_{|\xi|=1} \Str\left(\frac{\partial \mc{P}_{\xi}}{\partial\xi} \mc{Q}_{\xi}\right)d\xi \\
&\hspace{6mm}+\frac{1}{2\pi i}\int_{0}^{1}\oint_{|\xi|=1} \Str\left(\frac{\partial H_{1}}{\partial \xi}H_{2}\right) d\xi\,ds. \numberthis \label{transformula}
\end{align*}
Looking at equation (\ref{transformula}), our first term is a family version of the first term on the right-hand side of (23) in \cite{mrs1}. We will argue that our second term $\beta_{2}(t)$ is an exact form, and the last two terms together yield the end-periodic eta form.

\begin{proposition}\label{exactprop}
The term $\beta_{2}(t)$ on the right-hand side of equation (\ref{transformula}),
\[
\beta_{2}(t) = \frac{1}{2\pi i}\oint_{|\xi|=1} \Str\left(\frac{\partial}{\partial \xi}\left(A_{[1]}(Y)(\xi)\right)\mc{Q}_{\xi}\right)d\xi \in \Omega^{\bigcdot}(B)
\]
is an exact form on $B$.
\end{proposition}
\begin{proof}
\begin{comment}
Note that $A_{[1]}(Y)(\xi)$ is a connection on the twisted bundle $E_{Y}\otimes L_{\xi} \to Y$ (see Lemma \ref{twistedlemma}).
\end{comment}
Decompose the connection as $A_{[1]}(Y)(\xi) = d + \varphi$ where $\varphi$ is a bundle map, and integrate by parts:
\begin{align*}
\oint_{|\xi|=1} \Str\left(\frac{\partial}{\partial \xi}\left(A_{[1]}(Y)(\xi)\right)\mc{Q}_{\xi}\right) d\xi &= -\oint_{|\xi|=1} \Str\left(A_{[1]}(Y)(\xi)\frac{\partial \mc{Q}_{\xi}}{\partial \xi}\right)d\xi \\
&= -\oint_{|\xi|=1} \Str\left(d\cdot\frac{\partial \mc{Q}_{\xi}}{\partial \xi}\right) + \Str\left[\varphi,\frac{\partial \mc{Q}_{\xi}}{\partial\xi} \right] d\xi \\
&= -d_{B}\left(\oint_{|\xi|=1}\Str\left(\frac{\partial \mc{Q}_{\xi}}{\partial\xi}\right) d\xi\right)
\end{align*}
because $\Str\left[\varphi,\frac{\partial \mc{Q}_{\xi}}{\partial\xi} \right] = 0$. For comparison see the analogous argument in \cite[Lemma 9.15]{bgv} and the last line in the proof of \cite[Proposition 11]{mp}.
\end{proof}

The proof of the transgression formula is completed by defining the time-dependent \emph{end-periodic eta form} as the combination of the two remaining terms in equation (\ref{transformula}) (with the opposite sign, to be consistent with \cite[Proposition 11]{mp}). Thus we define the (time-dependent) end-periodic eta form by
\begin{align*}
\frac{1}{2}\epeta(t) &= -\alpha_{2}(t) - \gamma_{2}(t) \\
&= -\frac{1}{2\pi i}\oint_{|\xi|=1} \Str\left(\frac{\partial \mc{P}_{\xi}}{\partial\xi} \mc{Q}_{\xi}\right)  d\xi - \frac{1}{2\pi i}\int_{0}^{1}\oint_{|\xi|=1} \Str\left(\frac{\partial H_{1}}{\partial \xi}H_{2}\right)  d\xi\,ds.
\end{align*}
In the next section we will clarify this object, and then state the end result in Theorem \ref{mainthm1}. Depending upon the particular geometry of the fiber bundle $Y \to B$ it is possible that $\epeta(t)$ simplifies significantly; for example, we will show that when $Y = S^{1} \times \partial Z$ is the furled-up manifold for a cylindrical end, $\epeta(t)$ reduces to the time-dependent eta form of \cite[Proposition 11]{mp}. The integral over $t \in (0,\infty)$ of $\epeta(t)$ yields the end-periodic eta form appearing in our index formula.

\subsection{The end-periodic eta form}\label{section:epeta}
In this section we will present a geometric interpretation for the term $\gamma_{2}(t)$ involved in the definition of the end-periodic eta form. We will also show that the latter reduces to the Bismut-Cheeger eta form when the periodic end is cylindrical, and that the degree zero component $(\epeta)_{[0]}$ is exactly the end-periodic eta invariant of \cite{mrs1}.

We will abbreviate $A_{t}(Y)(\xi) = A_{t}(\xi)$. Recall the precise structure of indicial family of the rescaled end-periodic Bismut superconnection from equation (\ref{eqn:rescaledFL}). We have
\begin{align*}
& \mc{P}_{\xi} = A_{t}(\xi)-A_{[1]}(Y)(\xi) = t^{1/2}D(Y)(\xi)- t^{-1/2}A_{[2]}(Y) \\
&\implies \frac{\partial \mc{P}_{\xi}}{\partial\xi} = \frac{\partial}{\partial\xi}(A_{t}(\xi)-A_{[1]}(Y)(\xi)) = \frac{-t^{1/2}}{\xi}\cl(d_{Y/B}f).
\end{align*}
Therefore the first term in the definition of $\epeta(t)$ is
\begin{align*}
\alpha_{2}(t) &= -\frac{1}{2\pi i}\oint_{|\xi|=1} \Str\left(\frac{\partial \mc{P}_{\xi}}{\partial\xi} \mc{Q}_{\xi}\right)  d\xi \\
&= \frac{-1}{2\pi i}\oint_{|\xi|=1}\Str\left(\frac{\partial}{\partial\xi}(A_{t}(\xi)-A_{[1]}(Y)(\xi))\dt{A}_{t}(\xi)e^{-A_{t}^{2}(\xi)}\right)d\xi \\
&= \frac{-1}{2\pi i}\oint_{|\xi|=1}t^{1/2}\Str\left(\cl(d_{Y/B}f)\dt{A}_{t}(\xi)e^{-A_{t}^{2}(\xi)}\right)  \frac{d\xi}{\xi}.
\end{align*}

\noindent Note that this is very similar to the object in \cite{mp} equation (13.13) which, as \cite[Proposition 11]{mp} argues in the cylindrical end context, is equal to the time-dependent Bismut-Cheeger eta form $\frac{1}{2}\widehat{\eta}(t)$. We will show in this section that it does in fact generalize the time-dependent Bismut-Cheeger eta form. Now turning to the second term in the definition of $\epeta(t)$, we recall that we defined
\begin{align*}
2\pi i \gamma_{2}(t) &= \int_{0}^{1}\oint_{|\xi|=1} \Str\left(\frac{\partial H_{1}}{\partial \xi}H_{2}\right)  d\xi\,ds \\
&= \int_{0}^{1}\oint_{|\xi|=1} \Str\left( \frac{\partial}{\partial\xi}\left(e^{-sA_{t}^{2}(\xi)}\right)\frac{d}{dt}(A_{t}^{2}(\xi))e^{-(1-s)A_{t}^{2}(\xi)}\right)d\xi\,ds.
\end{align*}
\noindent As we will explain, it is natural to think of this object in relation to the operator $[A_{t}(\xi),\mc{B}(\xi)]$ where we set
\[
\mc{B}(\xi) = \frac{\partial}{\partial\xi}A_{t}(\xi) = -\xi^{-1}\delta_{t}m_{0}(df)
\]
so that
\begin{align*}
\frac{\partial}{\partial \xi}(A^{2}_{t}(\xi)) &= A_{t}(\xi)\frac{\partial}{\partial\xi}(A_{t}(\xi)) + \frac{\partial}{\partial\xi}(A_{t}(\xi))A_{t}(\xi) \\
&= [A_{t}(\xi),\mc{B}(\xi)] \\
&= -\xi^{-1}[A_{t}(\xi),\delta_{t}m_{0}(df)].
\end{align*}
We will rewrite the second term $\gamma_{2}(t)$ in the definition of $\epeta(t)$ so as to exhibit its relation with the operator $[A_{t}(\xi),\delta_{t}m_{0}(df)]$. The following proposition explains the geometric significance of the operator $[A_{t}(\xi),\delta_{t}m_{0}(df)]$.
\begin{proposition}\label{prop:hessian}
With $A(Y)$ the Bismut superconnection on the furled-up fiber bundle $Y \to B$, we have
\begin{equation}\label{eqn:hessian}
[A(Y),m_{0}(df)] = m_{0}(\con^{T^{*}Y}df) - 2\con_{(d_{Y/B}f)^{\#}}^{\mb{E}_{Y},0}
\end{equation}
and
\[
[A(Y)(\xi),m_{0}(df)] = m_{0}(\con^{T^{*}Y}df) - 2\con_{(d_{Y/B}f)^{\#}}^{\mb{E}_{Y},0} + 2\ln(\xi)|d_{Y/B}f|^{2},
\]
where $(d_{Y/B}f)^{\#}$ is the dual vector field with respect to the metric $g_{Y}^{0}$ on $Y$, cf.  (\ref{eqn:degenmetric}). Moreover, if the periodic end is cylindrical, i.e., $\tildhat{Y} = \mb{R} \times \partial Z$ and $f(r,q) = r$, then
\[
m_{0}(\con^{T^{*}Y}df) = \con_{(d_{Y/B}f)^{\#}}^{\mb{E}_{Y},0} = 0
\]
hence also $[A(Y),m_{0}(df)] = 0$.
\end{proposition}
\begin{proof}
We start by establishing the identity (\ref{eqn:hessian}). The Bismut superconnection is by definition
\[
A(Y) = \sum_{j}m_{0}(e^{j})\con^{\mb{E}_{Y},0}_{e_{j}}
\]
for any local orthonormal frame $(e_{j})$ for $TY$. Note that
\[
[\con_{e_{j}}^{\mb{E}_{Y},0},m_{0}(df)] = m_{0}(\con_{e_{j}}^{T^{*}Y}df)
\]
because $\con^{\mb{E}_{Y},0}$ is a Clifford connection, i.e., compatible with the Clifford action $m_{0}$.  Moreover, since $A(Y)$ and $m_{0}(df)$ are odd operators their supercommutator is an anticommutator. Therefore
\begin{align*}
[A(Y),m_{0}(df)] &= A(Y)m_{0}(df) + m_{0}(df)A(Y) \\
&= \sum_{j}m_{0}(e^{j})\con^{\mb{E}_{Y},0}_{e_{j}}m_{0}(df) + m_{0}(df)m_{0}(e^{j})\con^{\mb{E}_{Y},0}_{e_{j}} \\
&= \sum_{j} m_{0}(e^{j})\con^{\mb{E}_{Y},0}_{e_{j}}m_{0}(df) + \left(-m_{0}(e^{j})m_{0}(df)-2g_{Y}^{0}(df,e^{j})\right)\con^{\mb{E}_{Y},0}_{e_{j}} \\
&= \sum_{j} m_{0}(e^{j})\left(\con^{\mb{E}_{Y},0}_{e_{j}}m_{0}(df)-m_{0}(df)\con^{\mb{E}_{Y},0}_{e_{j}}\right) - 2g_{Y}^{0}(df,e^{j})\con^{\mb{E}_{Y},0}_{e_{j}} \\
&= \sum_{j} m_{0}(e^{j})[\con^{\mb{E}_{Y},0}_{e_{j}},m_{0}(df)]-2g_{Y}^{0}(df,e^{j})\con^{\mb{E}_{Y},0}_{e_{j}} \\
&= \sum_{j} m_{0}(e^{j})m_{0}(\con_{e_{j}}^{T^{*}Y}df)-2g_{Y}^{0}(df,e^{j})\con^{\mb{E}_{Y},0}_{e_{j}} \\
&= m_{0}(\con^{T^{*}Y}df) - 2\con_{(d_{Y/B}f)^{\#}}^{\mb{E}_{Y},0}.
\end{align*}
Then the identity for $[A(Y)(\xi),m_{0}(df)]$ is an immediate consequence of (\ref{eqn:hessian}), because
\begin{align*}
[A(Y)(\xi),m_{0}(df)] &= [A(Y),m_{0}(df)] - 2\ln(\xi)m_{0}(df)^{2} \\
&= [A(Y),m_{0}(df)] + 2\ln(\xi)|d_{Y/B}f|^{2}.
\end{align*}
Suppose the periodic end is cylindrical. Then $f(r,q) = r$ is just the lateral coordinate along the cylinder and $df$ is a parallel covector field on $\tildhat{Y}$, so the Hessian of $f$ vanishes and $m_{0}(\con^{T^{*}Y}df) \equiv 0$. Moreover, $\grad f = \partial_{r}$ and by \cite{mp} equation (1.9) the connection on $\tildhat{E} \to \tildhat{Y}$ is translation invariant, in other words we have $\con^{\tildhat{E}}_{\partial_{r}} =p^{*}\con^{E_{Y}}_{\partial_{r}} = 0$, hence
\begin{align*}
\con_{\partial_{r}}^{\mb{E}_{Y},0} &= \pi^{*}\con^{T^{*}B}_{\partial_{r}}\otimes 1 + 1 \otimes \con^{E_{Y}}_{\partial_{r}} + \frac{1}{2}m_{0}(\partial_{r}\lrcorner\,\omega) \\
&= 0 + \frac{1}{2}m_{0}(\partial_{r}\lrcorner\,\omega).
\end{align*}
Note that the left-hand side of this equation is a derivation (i.e. satisfies a product rule with respect to multiplication by scalar functions) whereas the right-hand side is $C^{\infty}(Y)$-linear, so we must have $\con_{\partial_{r}}^{\mb{E}_{Y},0} = 0$. Thus $[A(Y),m_{0}(df)] = 0$ when the end is cylindrical.
\end{proof}

We will show later (cf. Proposition \ref{prop:hessian2}) that the vanishing of the Hessian $m_{0}(\con^{T^{*}Y}df)$ is not just necessary but also sufficient to deduce that the end-periodic fiber bundle is cylindrical. 

Now we will clarify the differential form $\gamma_{2}(t)$ by rewriting it so as to exhibit its relation with the operator $[A_{t}(\xi),\delta_{t}m_{0}(df)]$. First, apply Duhamel's formula to the $\xi$-derivative of the heat operator:
\begin{align*}
2\pi i \gamma_{2}(t) = -\int_{0}^{1}\int_{0}^{1}\oint_{S^{1}}s\Str\bigg(e^{-suA_{t}^{2}(\xi)}& [A_{t}(\xi),\mc{B}(\xi)]e^{-s(1-u)A_{t}^{2}(\xi)} \\
&\cdot \frac{d}{dt}(A_{t}^{2}(\xi))e^{-(1-s)A_{t}^{2}(\xi)}\bigg)\,d\xi\,ds\,du
\end{align*}
which after cycling these even operators under the $\Str$ becomes
\begin{align*}
2\pi i \gamma_{2}(t) = -\int_{0}^{1}\int_{0}^{1}\oint_{S^{1}}s\Str\bigg(& e^{-s(1-u)A_{t}^{2}(\xi)} \frac{d}{dt}(A_{t}^{2}(\xi)) \\
&\hspace{3mm} \cdot e^{-(1-s(1-u))A_{t}^{2}(\xi)}[A_{t}(\xi),\mc{B}(\xi)]\bigg)\,d\xi\,ds\,du. \numberthis \label{eqn:newtermdoubleint}
\end{align*}
We introduce the shorthand
\[
F(s(1-u)) = \Str\bigg( e^{-s(1-u)A_{t}^{2}(\xi)} \frac{d}{dt}(A_{t}^{2}(\xi))\cdot e^{-(1-s(1-u))A_{t}^{2}(\xi)}[A_{t}(\xi),\mc{B}(\xi)]\bigg).
\]
Making the change of variables $v = s(1-u)$, the double integral (\ref{eqn:newtermdoubleint}) becomes
\begin{align*}
\int_{0}^{1}&\int_{0}^{1}sF(s(1-u))\,ds\,du = -\int_{0}^{1}\int_{0}^{s}sF(v)\cdot\frac{-1}{s}\,dv\,ds \\
&= \int_{0}^{1}\int_{0}^{s}F(v)\,dv\,ds = \int_{0}^{1}\int_{v}^{1}F(v)\,ds\,dv = \int_{0}^{1}(1-v)F(v)\,dv.
\end{align*}
Therefore 
\begin{align*}
2\pi i \gamma_{2}(t) &= -\int_{0}^{1}\int_{0}^{1}\oint_{S^{1}}sF(s(1-u))\,d\xi\,ds\,du \\
&= -\int_{0}^{1}\oint_{S^{1}}(1-v)F(v)\,d\xi\,dv \\
&= \int_{0}^{1}\oint_{S^{1}}(v-1)\Str\left(e^{-vA_{t}^{2}(\xi)}\frac{d}{dt}(A_{t}^{2}(\xi))e^{-(1-v)A_{t}^{2}(\xi)}[A_{t}(\xi),\mc{B}(\xi)]\right)d\xi\,dv \\
&= \oint_{S^{1}}\Str\left(\int_{0}^{1}(v-1)e^{-vA_{t}^{2}(\xi)}\frac{d}{dt}(A_{t}^{2}(\xi))e^{-(1-v)A_{t}^{2}(\xi)}\,dv\cdot[A_{t}(\xi),\mc{B}(\xi)]\right)d\xi \\
&= \oint_{S^{1}}\Str\left(\frac{d}{dt}\int_{0}^{1}e^{-uA_{t}^{2}(\xi)}\,du\cdot[A_{t}(\xi),\mc{B}(\xi)]\right)d\xi \\
&= \oint_{S^{1}}\Str\left(\dt{\mc{H}}_{t}(\xi)[A_{t}(\xi),\mc{B}(\xi)]\right)d\xi
\end{align*}
where we define the operator
\[
\mc{H}_{t}(\xi) = \int_{0}^{1}e^{-uA_{t}^{2}(\xi)}\,du.
\]
Note that
\[
[A_{t}(\xi),\mc{B}(\xi)]e^{-uA_{t}^{2}(\xi)} = [A_{t}(\xi)\mc{B}(\xi),e^{-uA_{t}^{2}(\xi)}] + [A_{t}(\xi)e^{-uA_{t}^{2}(\xi)},\mc{B}(\xi)]
\]
for any $u > 0$, hence also
\[
[A_{t}(\xi),\mc{B}(\xi)]\mc{H}_{t}(\xi) = [A_{t}(\xi)\mc{B}(\xi),\mc{H}_{t}(\xi)] + [A_{t}(\xi)\mc{H}_{t}(\xi),\mc{B}(\xi)]
\]
and therefore
\[
\Str\left([A_{t}(\xi),\mc{B}(\xi)]\mc{H}_{t}(\xi)\right) = 0.
\]
Differentiating this latter equation with respect to $t$ yields
\begin{align*}
0 &= \frac{d}{dt}\oint_{S^{1}}\Str\left([A_{t}(\xi),\mc{B}(\xi)]\mc{H}_{t}(\xi)\right)d\xi \\
&= \oint_{S^{1}}\Str\left(\frac{d}{dt}[A_{t}(\xi),\mc{B}(\xi)]\mc{H}_{t}(\xi)\right)d\xi + \oint_{S^{1}}\Str\left([A_{t}(\xi),\mc{B}(\xi)]\dt{\mc{H}}_{t}(\xi)\right)d\xi \\
&= \oint_{S^{1}}\Str\left(\frac{d}{dt}[A_{t}(\xi),\mc{B}(\xi)]\mc{H}_{t}(\xi)\right)d\xi + 2\pi i \gamma_{2}(t),
\end{align*}
hence
\begin{align*}
2\pi i \gamma_{2}(t) &= -\oint_{S^{1}}\Str\left(\frac{d}{dt}[A_{t}(\xi),\mc{B}(\xi)]\mc{H}_{t}(\xi)\right)d\xi \\
&= \oint_{S^{1}}\Str\left(\frac{d}{dt}[A_{t}(\xi),\delta_{t}m_{0}(df)]\mc{H}_{t}(\xi)\right)\frac{d\xi}{\xi},
\end{align*}
and this completes our derivation of the transgression formula.

\begin{theorem}[Transgression formula]\label{mainthm1}
Let $M \to B$ be an end-periodic fiber bundle and let $E \to M$ be a family of end-periodic Clifford modules, with $D$ the associated family of end-periodic Dirac operators. Let $A_{t}$ be the rescaled end-periodic Bismut superconnection on $\mb{E} \to M$ adapted to $D$. Then we have the following transgression formula for the renormalized Chern character:
\begin{align*}\numberthis \label{eqn:finaltrans}
\frac{d}{dt}\rch(A_{t}) &= \frac{1}{2\pi i}\frac{d}{dt}\oint_{|\xi|=1}\int_{W_{0}/B}f(x)\str\left(K_{e^{-A_{t}^{2}(Y)(\xi)}}(x,x)\right)dx\frac{d\xi}{\xi} \\
& \hspace{6mm} - \frac{1}{2}\epeta(t) -d_{B}\left(\frac{1}{2\pi i}\oint_{|\xi|=1}\Str\left(\frac{\partial \mc{Q}_{\xi}}{\partial\xi}\right)d\xi\right),
\end{align*}
where the end-periodic eta form is given by
\begin{align*}
\epeta(t) &= \frac{-1}{\pi i}\oint_{|\xi|=1}t^{1/2}\Str\left(\cl(d_{Y/B}f)\dt{A}_{t}(Y)(\xi)e^{-A_{t}^{2}(Y)(\xi)}\right)\frac{d\xi}{\xi} \\
&\hspace{6mm} - \frac{1}{\pi i}\oint_{|\xi|=1}\Str\left(\frac{d}{dt}[A_{t}(Y)(\xi),\delta_{t}m_{0}(df)]\mc{H}_{t}(\xi)\right)\frac{d\xi}{\xi},
\end{align*}
with $\displaystyle \mc{H}_{t}(\xi) = \int_{0}^{1}e^{-uA_{t}^{2}(\xi)}\,du.$
\end{theorem}

Next we will show that the end-periodic eta form reduces to the Bismut-Cheeger eta form when the periodic end is cylindrical. In this case, the furled-up manifold is simply $Y = S^{1}\times \partial Z$ with infinite cyclic cover $\tildhat{Y} = \mb{R} \times \partial Z$. Our definition of $\epeta(t)$ involves the $\gamma_{2}(t)$ term which we initially defined as
\begin{align*}
\gamma_{2}(t) = \frac{1}{2\pi i}\int_{0}^{1}\int_{|\xi|=1} \Str\left( \frac{\partial}{\partial\xi}\left(e^{-sA_{t}^{2}(Y)(\xi)}\right)\frac{d}{dt}(A_{t}^{2}(Y)(\xi))e^{-(1-s)A_{t}^{2}(Y)(\xi)}\right)d\xi\,ds.
\end{align*}
In this integrand, the only contributing factors are those that produce a term which is \emph{even} in $\xi$, because the terms which are \emph{odd} in $\xi$ will integrate to zero by symmetry. So we can simplify this expression by looking at the involved operators and identifying the parts that are even and odd in $\xi$.

\begin{remark}[Even and odd functions in $\xi \in S^{1}$]
When we say that a function $h$ on $S^{1}\subseteq\mb{C}$ is \emph{odd} in $\xi$ we mean that it satisfies $h(\overline{\xi}) = -h(\xi)$ for all $\xi \in S^{1}$. For example, $\xi \mapsto \ln(\xi)$ is an odd function, whereas $\xi \mapsto \ln(\xi)^{2}$ is even. A straightforward calculation shows that if $h$ is odd in $\xi$ then
\[
\oint_{S^{1}}h(\xi)\frac{d\xi}{\xi} = 0.\vspace{-5mm}
\]
\end{remark}

Denote the lateral coordinate along the cylindrical end by $f(r,q) = r$ where $r \in [0,\infty)$. We also let $r$ denote the periodic coordinate along the $S^{1}$-direction of $Y$. In this case the indicial family of the Bismut superconnection is
\begin{align*}
A_{t}(Y)(\xi) &= A_{t}(Y) - \ln(\xi)\delta_{t}m_{0}(dr) \\
&= A_{t}(Y) - \ln(\xi)(t^{1/2}\cl(d_{Y/B}r)+\pi^{*}d_{B}r),
\end{align*}
with curvature
\begin{align*}
A_{t}^{2}(Y)(\xi) &= A_{t}^{2}(Y) - \ln(\xi)\delta_{t}[A(Y),m_{0}(dr)] - t\ln(\xi)^{2}|d_{Y/B}r|^2 \\
&= A_{t}^{2}(Y) - t\ln(\xi)^{2}
\end{align*}
where we have used that $[A(Y), m_{0}(dr)] = 0$ in the cylindrical end case by Proposition \ref{prop:hessian}. Owing to this simplification, we have
\[
e^{-A_{t}^{2}(Y)(\xi)} = e^{-t\ln(\xi)^{2}}e^{-A_{t}^{2}(Y)} = \mbox{even in }\, \xi.
\]
Moreover,
\begin{align*}
&\frac{\partial}{\partial \xi}(e^{-sA_{t}^{2}(Y)(\xi)}) = -2st\frac{\ln(\xi)}{\xi}e^{-st\ln(\xi)^{2}}e^{-sA_{t}^{2}(Y)} = \frac{1}{\xi}(\mbox{odd in }\, \xi), \\
&\frac{d}{dt}(A_{t}^{2}(Y)(\xi)) = \frac{d}{dt}(A_{t}^{2}(Y))-\ln(\xi)^{2}
= \mbox{even in }\, \xi.
\end{align*}
Therefore
\begin{equation}\label{eqn:cylparity2}
\gamma_{2}(t) = -\frac{1}{2\pi i}\int_{0}^{1}\oint_{S^{1}}\Str\left(\mathrm{odd}\cdot\mathrm{even}\cdot\mathrm{even}\right)\frac{d\xi}{\xi}\,ds = 0
\end{equation}
which is zero because the integrand is completely odd in $\xi$. 

Since $\gamma_{2}(t) = 0$ in the cylindrical end case, the $\epeta(t)$ form reduces to
\begin{align*}
\frac{1}{2}\epeta(t) &= \frac{-1}{2\pi i}\oint_{|\xi|=1} \Str\left(\frac{\partial \mc{P}_{\xi}}{\partial\xi} \mc{Q}_{\xi}\right) d\xi \\
&= \frac{-1}{2\pi i}\oint_{|\xi|=1}t^{1/2}\Str\left(\cl(dr)\dt{A}_{t}(Y)(\xi)e^{-A_{t}^{2}(Y)(\xi)}\right)\frac{d\xi}{\xi}.
\end{align*}
We introduce the shorthand $\mc{R} = \cl(dr)\dt{A_{t}}e^{-A_{t}^{2}}$ for the following calculation. Using Proposition \ref{newprop} we calculate
\begin{align*}
\frac{1}{2\pi i}\oint_{S^{1}}\Str(\mc{R}_{\xi}(Y))\frac{d\xi}{\xi} &\stackrel{(\ref{neweqn})}{=} \int_{W_{0}/B}\str(K_{\tildhat{\mc{R}}}(x,x))\,dx \\
&= \int_{0}^{1}\int_{\partial Z/B}\str(K_{\tildhat{\mc{R}}}((r,q),(r,q)))\,dq\,dr \\
&\stackrel{(\ref{neweqn2})}{=} \int_{0}^{1}\int_{\partial Z/B}\frac{1}{2\pi}\int_{\mb{R}}\str(K_{\mc{R}_{\theta}(\partial Z)}(q,q))\,d\theta\,dq\,dr \\
&= \frac{1}{2\pi}\int_{\mb{R}}\int_{\partial Z/B}\str(K_{\mc{R}_{\theta}(\partial Z)}(q,q))\,dq\,d\theta \\
&= \frac{1}{2\pi}\int_{\mb{R}}\Str(\mc{R}_{\theta}(\partial Z))\,d\theta.
\end{align*}
Looking at the Mellin transform-induced indicial family with parameter $\theta$ where $\xi = e^{i\theta}$, it is shown in \cite[Proposition 7]{mp} and \cite[Proposition 11]{mp} that
\begin{align*}
&A_{t}(\partial Z)(\theta) = A_{t}(\partial Z) -i\theta m_{0}(dr)\\
&A_{t}^{2}(\partial Z)(\theta) = A_{t}^{2}(\partial Z) + t\theta^{2} \\
&e^{-A_{t}^{2}(\partial Z)(\theta)} = e^{-t\theta^{2}}e^{-A_{t}^{2}(\partial Z)}.
\end{align*}
Therefore
\begin{align*}
\mc{R}_{\theta}(\partial Z) &= \cl(dr)\dt{A_{t}}(\partial Z)(\theta)e^{-A_{t}^{2}(\partial Z)(\theta)} \\
&= \cl(dr)\left[\frac{-i}{2}\theta t^{-1/2}\cl(dr)+\dt{A_{t}}(\partial Z)\right]e^{-t\theta^{2}}e^{-A_{t}^{2}(\partial Z)},
\end{align*}
hence
\begin{align*}
\epeta(t) &= -\frac{1}{\pi i}\oint_{S^{1}}t^{1/2}\Str(\mc{R}_{\xi}(Y))\frac{d\xi}{\xi} \\
&= -\frac{1}{\pi}\int_{\mb{R}}\Str(\mc{R}_{\theta}(\partial Z))\,d\theta \\
&= -\frac{1}{\pi}\int_{\mb{R}}\frac{i}{2}\theta e^{-t\theta^{2}}\,d\theta\cdot \Str(e^{-A_{t}^{2}(\partial Z)}) \\
&\hspace{12mm} - \frac{1}{\pi}\int_{\mb{R}}e^{-t\theta^{2}}\,d\theta \cdot t^{1/2}\Str\left(\cl(dr)\dt{A_{t}}(\partial Z)e^{-A_{t}^{2}(\partial Z)}\right) \\
&= 0 - \frac{1}{\sqrt{\pi}}\Str\left(\cl(dr)\dt{A_{t}}(\partial Z)e^{-A_{t}^{2}(\partial Z)}\right)
\end{align*}
which is exactly the expression for the eta form in \cite{mp} equation (13.14).

\begin{remark}
As shown in \cite[\S 6.3]{mrs1}, the end-periodic eta invariant reduces to the classical APS eta invariant when the end is cylindrical. This is the degree zero part of what we have just shown above. The argument of \cite{mrs1} relies on the fact that the eta invariant can be readily expressed in terms of the boundary spectrum, and calculations can be carried out using a complete orthonormal system of eigenvectors for the boundary Dirac operator. In contrast, our proof that the end-periodic eta form reduces to the Bismut–Cheeger eta form follows a rather different approach, because the eta form includes differential form components of higher degree that are not amenable to spectral methods.
\end{remark}

Now we will show that the degree zero component of the end-periodic eta form -- the function $(\epeta)_{[0]} \in C^{\infty}(B)$ -- is the fiberwise end-periodic eta invariant of \cite[Theorem A]{mrs1}. In particular this implies that the index formula of Theorem \ref{mainthm1} reduces to the index formula of \cite[Theorem A]{mrs1} in the single operator case.

Considering a family of end-periodic Dirac operators $D = (D^{z})_{z\in B}$, the (time-dependent) end-periodic eta form in Theorem \ref{mainthm1} is
\begin{align*}
\epeta(t) &= \frac{-1}{\pi i}\oint_{|\xi|=1}t^{1/2}\Str\left(\cl(d_{Y/B}f)\dt{A}_{t}(Y)(\xi)e^{-A_{t}^{2}(Y)(\xi)}\right)\frac{d\xi}{\xi} \\
&\hspace{6mm} - \frac{1}{\pi i}\oint_{|\xi|=1}\Str\left(\frac{d}{dt}[A_{t}(Y)(\xi),\delta_{t}m_{0}(df)]\mc{H}_{t}(\xi)\right)\frac{d\xi}{\xi}.
\end{align*}
By replacing the Bismut superconnection $A$ with the Dirac opertor $D$, it is clear that the degree zero component of the first term is the fiberwise end-periodic eta invariant $z\mapsto \eta_{\mathrm{ep}}(D^{z}(Y_{z}))(t)$, cf. (\ref{eqn:epetainv}). It remains to show that the degree zero component of the second term is zero, i.e.,
\[
\int_{0}^{1}\rStr\left[e^{-stD^{2}},D^{2}e^{-(1-s)tD^2}\right]ds = 0.
\]
This can be shown by the following direct calculation. We set $H_{1} = e^{-stD^{2}}$ and $H_{2} = D^{2}e^{-(1-s)tD^{2}}$, and write $H_{j}(\xi) = H_{j}(Y)(\xi)$ and $D_{\xi} = D(Y)(\xi)$ for brevity. The supertrace defect formula yields
\begin{align*}
\rStr[H_{1},H_{2}] &= \frac{1}{2\pi i}\int_{W_{0}}\oint_{S^{1}}f(x)\str(K_{H_{1}(\xi)H_{2}(\xi)}(x,x)-K_{H_{2}(\xi)H_{1}(\xi)}(x,x))\frac{d\xi}{\xi}\,dx \\
&\hspace{4mm} -\frac{1}{2\pi i}\oint_{S^{1}}\Str\left(\frac{\partial}{\partial\xi}(H_{1}(\xi))\cdot H_{2}(\xi)\right)d\xi.
\end{align*}
Note that $H_{1}(\xi)H_{2}(\xi) = H_{2}(\xi)H_{1}(\xi)$,  so the first term in $\rStr[H_{1},H_{2}]$ vanishes. As for the second term in $\rStr[H_{1},H_{2}]$, we have
\[
D_{\xi}^{2} = D^{2}-\ln(\xi)[D,\cl(d_{Y/B}f)]-\ln(\xi)^{2}|d_{Y/B}f|^{2} = D^{2} + \mc{C}(\xi)
\]
where we write $\mc{C}(\xi) = -\ln(\xi)[D,\cl(d_{Y/B}f)]-\ln(\xi)^{2}|d_{Y/B}f|^{2}$. Therefore
\begin{align*}
\frac{\partial}{\partial\xi}(H_{1}(\xi)) = -\int_{0}^{1}ste^{-ustD_{\xi}^{2}}\partial_{\xi}(\mc{C}(\xi))e^{-(1-u)stD_{\xi}^{2}}\,du
\end{align*}
and
\begin{align*}
\oint_{S^{1}}& \Str\left(\frac{\partial}{\partial\xi}(H_{1}(\xi))\cdot H_{2}(\xi)\right)d\xi \\
&= -\int_{0}^{1}\oint_{S^{1}}st\Str\left(e^{-ustD_{\xi}^{2}}\partial_{\xi}(\mc{C}(\xi))e^{-(1-u)stD_{\xi}^{2}}D_{\xi}^{2}e^{-tD_{\xi}^{2}} \right)d\xi\,du \\
&= -st\oint_{S^{1}}\Str\left(\partial_{\xi}(\mc{C}(\xi))D_{\xi}^{2}e^{-tD_{\xi}^{2}}\right)d\xi.
\end{align*}
Note that
\[
\partial_{\xi}(\mc{C}(\xi)) = \frac{2}{\xi}(\cl(\con^{T^{*}Y}df)-\ln(\xi)|d_{Y/B}f|^{2})
\]
and so
\begin{align*}
\oint_{S^{1}}\Str\left(\partial_{\xi}(\mc{C}(\xi))D_{\xi}^{2}e^{-tD_{\xi}^{2}}\right)d\xi &= 2\oint_{S^{1}}\Str\left(\cl(\con^{T^{*}Y}df)D_{\xi}^{2}e^{-tD_{\xi}^{2}}\right)\frac{d\xi}{\xi} \\
&\hspace{4mm} -2|d_{Y/B}f|^{2}\oint_{S^{1}}\Str(D_{\xi}^{2}e^{-tD_{\xi}^{2}})\ln(\xi)\,\frac{d\xi}{\xi}.
\end{align*}
Moreover we see that
\[
\Str\left(\cl(\con^{T^{*}Y}df)D_{\xi}^{2}e^{-tD_{\xi}^{2}}\right) = 0
\]
because the operator inside the supertrace is odd with respect to the $\mb{Z}_{2}$-grading, and
\[
\Str(D_{\xi}^{2}e^{-tD_{\xi}^{2}}) = \frac{1}{2}\Str[D_{\xi},D_{\xi}e^{-tD_{\xi}^{2}}] = 0.
\]
In total this means
\[
\oint_{S^{1}} \Str\left(\frac{\partial}{\partial\xi}(H_{1}(\xi))\cdot H_{2}(\xi)\right)d\xi = 0
\]
and hence that $\rStr[H_{1},H_{2}] = 0$. Therefore the degree zero component of the (time-dependent) end-periodic eta form is
\begin{align*}
\epeta(t)_{[0]} &= \frac{-1}{\pi i}\oint_{|\xi|=1}t^{1/2}\Str\left(\cl(d_{Y/B}f)\frac{1}{2t^{1/2}}D_{\xi}e^{-tD_{\xi}^{2}}\right)\frac{d\xi}{\xi} \\
&= \frac{-1}{\pi i}\oint_{|\xi|=1}\Tr\left(\cl(d_{Y/B}f)D_{\xi}^{+}e^{-tD^{-}_{\xi}D_{\xi}^{+}}\right)\frac{d\xi}{\xi}\numberthis\label{degreezero}
\end{align*}
which is the function $z \mapsto \eta_{\mathrm{ep}}(D^{z}(Y_{z}))(t) \in \mb{R}$, mapping any $z \in B$ to the scalar (time-dependent) end-periodic eta invariant of the fiber over $z$. 

\subsection{The Hessian of a periodic end}\label{section:hessian}

In \S \ref{section:epeta} we showed that the term $\gamma_{2}(t)$ in the end-periodic eta form can be expressed in terms of the operator
\[
[A(Y),m_{0}(df)] = m_{0}(\con^{T^{*}Y}df) - 2\con_{(d_{Y/B}f)^{\#}}^{\mb{E}_{Y},0}.
\]
Moreover we showed in Proposition \ref{prop:hessian} that if $[A(Y),m_{0}(df)]$ vanishes then the Hessian $m_{0}(\con^{T^{*}Y}df)$ vanishes, i.e., $d_{Y/B}f$ is a parallel 1-form along the fibers of $Y \to B$, which is reminiscent of the parallel 1-form $dr$ along a cylindrical end. On the other hand, if $[A(Y),m_{0}(df)]$ vanishes then the end-periodic eta form reduces to
\[
\epeta(t) = \frac{-1}{\pi i}\oint_{|\xi|=1}t^{1/2}\Str\left(\cl(dr)\dt{A}_{t}(Y)(\xi)e^{-A_{t}^{2}(Y)(\xi)}\right)\frac{d\xi}{\xi}
\]
which is structurally similar to the Bismut-Cheeger eta form in the cylindrical end case. This leads to the question of whether the periodic end is actually cylindrical in this case. We will show that the periodic end is indeed cylindrical whenever $m_{0}(\con^{T^{*}Y}df) \equiv 0$; in other words, when the periodic end is noncylindrical then $m_{0}(\con^{T^{*}Y}df) \not\equiv 0$ which suggests that the new term in the end-periodic eta form is in general nonzero.

We recall from \S \ref{section:EPgeometry} the construction of the infinite cyclic covering space $\tildhat{Y} \to Y$. Let $(Y,g_{Y})$ be a smooth, closed, connected Riemannian manifold. One chooses a primitive cohomology class in $H^{1}(Y,\mathbb{Z})$, or equivalently a smooth map $f_{0}: Y \to S^{1}$. Corresponding to this class, the infinite cyclic covering space $p: \tildhat{Y} \to Y$ is defined via the pullback diagram
\[
\begin{tikzcd}
	\tildhat{Y} & \mb{R} \\
	Y & {S^{1}.}
	\arrow["f", from=1-1, to=1-2]
	\arrow["p"', from=1-1, to=2-1]
	\arrow["\exp", from=1-2, to=2-2]
	\arrow["f_{0}"', from=2-1, to=2-2]
\end{tikzcd}
\]
From the pullback diagram we see that the differential $df \in \Omega^{1}(\tildhat{Y})$ descends to the 1-form $f_{0}^{*}(d\theta)$ on $Y$, in other words we have $df = p^{*}(f_{0}^{*}(d\theta))$. We lift the geometric data from $Y$ up to $\tildhat{Y}$ via pullback: the metric $g$ on $TY$ lifts to a metric on $T\tildhat{Y}$, and the Levi-Civita connection on $T\tildhat{Y}$ with respect to the lifted metric is the lift of the Levi-Civita connection on $TY$.

Motivated by the remarks at the beginning of this subsection we will consider the special case where $df$ is a parallel 1-form on $\tildhat{Y}$, i.e., $\con^{T^{*}\tildhat{Y}} df = 0$ with respect to the Levi-Civita connection on $\tildhat{Y}$. Note that $p$ is a Riemannian covering map when $\tildhat{Y}$ is equipped with the pullback metric $p^{*}g_{Y}$, and so $df$ is parallel on $\tildhat{Y}$ if and only if $f_{0}^{*}(d\theta)$ is parallel on $Y$.

\begin{proposition}\label{prop:hessian2}
Suppose that $df$ is a parallel 1-form on $\tildhat{Y}$. Then $f_{0}: Y \to S^{1}$ defines a smooth fiber bundle $F - Y \xrightarrow{f_{0}} S^{1}$, and the infinite cyclic covering space is a cylinder $\tildhat{Y} = \mb{R} \times F$.
\end{proposition}
\begin{proof}
Differentiating the pullback diagram yields
\[
d(\exp)\circ df = df_{0}\circ dp.
\]
The fact that $df$ is parallel (and nonzero) implies that $df$ is a nonvanishing 1-form, hence $f$ is a smooth submersion. Since the covering maps $\exp$ and $p$ are local diffeomorphisms, $f_{0}$ is also a smooth submersion. Moreover, $f_{0}$ is surjective because its image $f_{0}(Y)$ is both open and closed in $S^{1}$, and $f_{0}$ is a proper map because $Y$ is compact. By Ehresmann's fibration theorem \cite[Theorem 9.3]{voisin} it follows that $f_{0}: Y \to S^{1}$ is a smooth locally trivial fibration, which in this case implies that it is in fact a smooth fiber bundle, because $Y$ and $S^{1}$ are smooth compact manifolds.

Now the infinite cyclic covering space is by definition the pullback of the fiber bundle $Y \to S^{1}$ along $\exp: \mathbb{R} \to S^{1}$, so $\tildhat{Y}$ is a smooth fiber bundle over $\mathbb{R}$. But $\mathbb{R}$ is contractible so $\tildhat{Y}$ is a trivial fiber bundle over $\mathbb{R}$.
\end{proof}

As a result, if the end-periodic fiber bundle has $m_{0}(\con^{T^{*}Y}df) \equiv 0$, then the fibers of the infinite cyclic covering $\tildhat{Y} \to B$ are all cylinders and thus the periodic end is actually a cylindrical end. Combining this with Proposition \ref{prop:hessian}, we have shown that the end-periodic fiber bundle is cylindrical if and only if the Hessian $m_{0}(\con^{T^{*}Y}df)$ vanishes identically.

%%% Cylinder => [A(Y),m_{0}(df)] = 0 => Hess = 0
%%% Hess = 0 => Cylinder

%% file: sections/6-Chern-character.tex
\section{Chern character of the index bundle}\label{section:indexthm}
In this section we will integrate our transgression formula (\ref{eqn:finaltrans}) to obtain a local index formula in the end-periodic family setting, following the approach of \cite[\S 9.3]{bgv} and \cite[\S 15]{mp}.

Integrating the transgression formula (\ref{eqn:finaltrans}) yields
\begin{align*}\numberthis\label{integratedtrans}
\lim_{t \to \infty} &\rch(A_{t}) - \lim_{t \to 0^{+}} \rch(A_{t}) = \\
& \lim_{t \to \infty}(\dag) - \lim_{t \to 0^{+}}(\dag) - \int_{0}^{\infty}\frac{1}{2}\epeta(t)\,dt - d_{B}\left(\frac{1}{2\pi i}\int_{0}^{\infty}\oint_{|\xi|=1}\Str\left(\frac{\partial \mc{Q}_{\xi}}{\partial\xi}\right)  d\xi\,dt\right)
\end{align*}
where we have denoted
\[
(\dag) = \frac{1}{2\pi i}\oint_{|\xi|=1}\int_{W_{0}/B}f(x)\str\left(K_{e^{-A_{t}^{2}(Y)(\xi)}}(x,x)\right)dx\frac{d\xi}{\xi}.
\]
We will show that the long-time and short-time limits of the renormalized Chern character converge to
\begin{align*}
\lim_{t \to \infty} \rch(A_{t}) - \lim_{t \to 0^{+}} \rch(A_{t}) = \ch(\nabla_{0}) - \int_{Z/B}\AS(D(Z)).
\end{align*}
On the right-hand side of (\ref{integratedtrans}), we will show that the limits converge to 
\begin{align*}
\lim_{t \to \infty}(\dag) - \lim_{t \to 0^{+}}(\dag) = 0 - \int_{W_{0}/B}f\,\AS(D(Y)).
\end{align*}
We define the end-periodic eta form as the integral over $t \in (0,\infty)$ of $\epeta(t)$,
\[
\epeta = \int_{0}^{\infty} \epeta(t)\,dt.
\]
Finally, the remaining exact term in the transgression formula, integrated over $t \in (0,\infty)$, is obviously exact.

First we verify that the integral over $t \in (0,\infty)$ of $\epeta(t)$ converges to a well-defined differential form on $B$. Our proof is similar to the analogous fact \cite[Proposition 12]{mp} in the context of manifolds with boundary.
\begin{proposition}
The integral over $t \in (0,\infty)$ of the time-dependent end-periodic eta form $\epeta(t)$ converges in $\Omega^{\bigcdot}(B)$, so we have a well-defined differential form
\[
\epeta = \int_{0}^{\infty} \epeta(t)\,dt \in \Omega^{\bigcdot}(B).
\]
\end{proposition}
\begin{proof}
% Melrose-Piazza Proposition 12.
The time-dependent end-periodic eta form is defined as
\begin{align*}
\epeta(t) &= \frac{-1}{\pi i}\oint_{|\xi|=1}t^{1/2}\Str\left(\cl(d_{Y/B}f)\dt{A}_{t}(Y)(\xi)e^{-A_{t}^{2}(Y)(\xi)}\right)  \frac{d\xi}{\xi} \\
&\hspace{6mm} - \frac{1}{\pi i}\oint_{|\xi|=1}\Str\left(\frac{d}{dt}[A_{t}(Y)(\xi),\delta_{t}m_{0}(df)]\mc{H}_{t}(\xi)\right)\frac{d\xi}{\xi}.
\end{align*}
We examine the rate of decay/growth in $t$ for both of these integrands. Looking at the first term here, since $D^{z}(Y)(\xi)$ has trivial nullspace for each $z \in B$ and $\xi \in S^{1}$, we have
\[
|e^{-A_{t}^{2}(Y)(\xi)}| \leq Ce^{-\lambda t}
\]
where $\lambda > 0$ is the infimum of the smallest eigenvalues of $D^{z}(Y)(\xi)^{2}$ over $z \in B$, cf. \cite[Lemma 9.5]{bgv}. Moreover $\dt{A}_{t}(Y)(\xi)$ decays like $t^{-1/2}$ which simply balances with the extra $t^{1/2}$ coefficient, hence
\[
t^{1/2}\Str\left(\cl(d_{Y/B}f)\dt{A}_{t}(Y)(\xi)e^{-A_{t}^{2}(Y)(\xi)}\right)
\]
decays exponentially as $t \to \infty$ and is therefore integrable as $t \to \infty$. On the other hand, as $t \to 0^{+}$ we have
\[
\Str\left(\dt{A}_{t}(Y)(\xi)e^{-A_{t}^{2}(Y)(\xi)}\right) \leq O(t^{-1/2})
\]
by \cite[Theorem 9.23]{bgv}, so the integrand in the first term is also integrable as $t \to 0^{+}$.

Now looking at the second term
\[
\gamma_{2}(t) = \frac{1}{2\pi i}\oint_{|\xi|=1}\Str\left(\frac{d}{dt}[A_{t}(Y)(\xi),\delta_{t}m_{0}(df)]\mc{H}_{t}(\xi)\right)\frac{d\xi}{\xi}
\]
a direct calculation shows that
\begin{align*}
[A_{t}(Y)(\xi),\delta_{t}m_{0}(df)] &= t[D(Y)(\xi),\cl(d_{Y/B}f)] + t^{1/2}[D(Y)(\xi),\pi^{*}d_{B}f] \\
&\hspace{4mm} + t^{1/2}[A_{[1]}(Y)(\xi),\cl(d_{Y/B}f)] + [A_{[1]}(Y)(\xi),\pi^{*}d_{B}f] \\
&\hspace{4mm} + [A_{[2]}(Y),cl(d_{Y/B}f)] + t^{-1/2}[A_{[2]}(Y),\pi^{*}d_{B}f].
\end{align*}
and in fact $[A_{[2]}(Y),\pi^{*}d_{B}f] = 0$, which annihilates the term which is most problematic as $t \to 0^{+}$. To see this, one may take a local orthonormal frame $(e_{j})$ for $T(Y/B)$ and $(e_{\alpha})$ for $\pi^{*}TB$, and any $\omega \otimes u \in \Gamma(M;\mb{E}_{Y}) = \Gamma(M;\pi^{*}\Lambda T^{*}B \otimes E_{Y})$, and calculate
\begin{align*}
(A_{[2]}(Y)&\cdot\pi^{*}d_{B}f)(\omega\otimes u) \\
&= \frac{-1}{4}\sum_{\alpha < \beta}\sum_{j}e^{\alpha}\wedge e^{\beta} \cl(e^{j})\left(\Omega_{Y}(e_{\alpha},e_{\beta})(e_{j})(\pi^{*}d_{B}f\wedge \omega \otimes u)\right) \\
&= \frac{-1}{4}\sum_{\alpha < \beta}\sum_{j}\bigg(e^{\alpha}\wedge e^{\beta}\, \Omega_{Y}(e_{\alpha},e_{\beta})(e_{j}) \\
&\hspace{28mm} \cdot (-1)^{\deg(u)+1}(\pi^{*}d_{B}f \wedge \omega) \otimes \cl(e^{j})(u)\bigg)
\end{align*}
and on the other hand, since $\pi^{*}d_{B}f$ can be commuted past the 2-form $e^{\alpha}\wedge e^{\beta}$,
\begin{align*}
(\pi^{*}d_{B}f \cdot A_{[2]}(Y))(\omega\otimes u) &= \frac{-1}{4}\sum_{\alpha < \beta}\sum_{j}\bigg(\pi^{*}d_{B}f\wedge e^{\alpha}\wedge e^{\beta}\,\Omega_{Y}(e_{\alpha},e_{\beta})(e_{j}) \\
&\hspace{28mm} \cdot (-1)^{\deg(u)}\omega\otimes\cl(e^{j})(u)\bigg) \\
&= \frac{-1}{4}\sum_{\alpha < \beta}\sum_{j}\bigg(e^{\alpha}\wedge e^{\beta}\,\Omega_{Y}(e_{\alpha},e_{\beta})(e_{j}) \\
&\hspace{28mm} \cdot (-1)^{\deg(u)}(\pi^{*}d_{B}f\wedge\omega)\otimes\cl(e^{j})(u)\bigg) \\
&= -(A_{[2]}(Y)\cdot\pi^{*}d_{B}f)(\omega\otimes u),
\end{align*}
hence $[A_{2}(Y),\pi^{*}d_{B}f] = A_{2}(Y)(\pi^{*}d_{B}f) + (\pi^{*}d_{B}f)A_{[2]}(Y) = 0$. Therefore
\begin{align*}
\frac{d}{dt}[A_{t}(Y)(\xi),\delta_{t}m_{0}(df)] &= [D(Y)(\xi),\cl(d_{Y/B}f)] + \frac{1}{2}t^{-1/2}[D(Y)(\xi),\pi^{*}d_{B}f] \\
&\hspace{4mm} + \frac{1}{2}t^{-1/2}[A_{[1]}(Y)(\xi),\cl(d_{Y/B}f)].
\end{align*}
As $t \to \infty$ this operator does not slow down the overall rate of convergence of the integrand, and so integrability as $t \to \infty$ follows from the same argument we made before for the other integrand, namely the rapid decay of the heat operator $\mc{H}_{t}(\xi)$ as $t \to \infty$.

As for integrability when $t \to 0^{+}$, note that the operator $\partial_{t}[A_{t}(Y)(\xi),\delta_{t}m_{0}(df)]$ grows at worst like $t^{-1/2}$ as $t \to 0^{+}$. By Proposition \ref{twistedprop}, the indicial family $A(Y)(\xi)$ can be twisted by a flat connection to produce a Bismut superconnection $\mb{A}$ on a twisted bundle $\mb{E}_{Y}\otimes L_{\xi}$. The local heat invariants of $\mb{A}$ and $A(Y)(\xi)$ coincide exactly because $\mb{A}$ is obtained from $A(Y)(\xi)$ by twisting with a flat connection on a line bundle over $Y$, hence the coefficients of the short-time asymptotic expansion of the heat kernel for $A_{t}^{2}(Y)(\xi)$ coincide with those of $\mb{A}_{t}^{2}$. By \cite[Theorem 10.21]{bgv} the lowest order term (in $t$) in the short-time asymptotic expansion for the latter heat kernel is $t^{0}$. Thus, overall, the integrand in $\gamma_{2}(t)$ grows at worst like $t^{-1/2}$ as $t \to 0^{+}$, which is integrable as $t \to 0^{+}$. Since both integrands are integrable as $t \to \infty$ and $t \to 0^{+}$, the integral over $t \in (0,\infty)$ converges in $\Omega^{\bigcdot}(B)$.
\end{proof}

Before we state the index formula, we briefly recall the definition of the $K$-theoretic index bundle and its Chern character in de Rham  cohomology. It is convenient to consider two cases:
\begin{enumerate}[(i)]
\item The \emph{constant rank} case: where $\dim \ker (D^{z})^{+}$ is constant over $z \in B$, in which case the spaces $\ker (D^{z})^{+}$ and $\ker (D^{z})^{-}$ assemble vector bundles and hence $\Ker D = \Ker D^{+}\oplus\Ker D^{-}$ is a $\mathbb{Z}_{2}$-graded vector bundle.
\item The \emph{non-constant rank} case: where $\dim \ker (D^{z})^{+}$ is allowed to vary over $z \in B$, hence the spaces $\ker (D^{z})^{+}$ and $\ker (D^{z})^{-}$ do not assemble vector bundles.
\end{enumerate}
Recall that the $K$-theory group $K(B)$ is by definition the Grothendieck group generated by isomorphism classes of vector bundles over $B$. In the constant-rank case the index bundle is simply the $K$-theory element
\[
\Ind D = [\Ker D^{+}] - [\Ker D^{-}] \in K(B)
\]
and the Chern character of the index bundle is the cohomology class
\[
\ch(\Ind D) = [\ch(\con_{0})] = [\Str(e^{-\con_{0}^{2}})] \in H_{dR}^{\bigcdot}(B)
\]
where $\con_{0} = P_{0}A_{[1]}P_{0}$ is the induced connection on $\Ker D$, with $P_{0}$ the projection onto the kernel. The cohomology class is independent of the choice of connection on the kernel.

In the non-constant rank case the kernels $\ker D^{z}$ no longer assemble a vector bundle and we need another way to define the index bundle. In this case we follow the standard stabilization procedure, cf. \cite[\S 6]{mp} and \cite[\S 9.5]{bgv}. The stabilization procedure constructs a vertical family of odd smoothing operators $P \in \Psi^{-\infty}_{\pi}(M;\mb{E})$ so that the perturbed operator $D + P$ is injective on each fiber, hence $\Coker(D+P)$ is a well-defined vector bundle. Then the index bundle of $D$ is defined as the $K$-theory class
\begin{equation}\label{eqn:indexbundle}
\Ind D = [B\times \mb{C}^{N}] - [\Coker(D+P)] \in K(B)
\end{equation}
where $N$ is a positive integer such that
\[
\Gamma(M_{z};E_{z}^{+})/V_{z} \simeq \mathbb{C}^{N}
\]
for every $z \in B$, with a subspace $V_{z} \subseteq \Gamma(M_{z};E_{z}^{+})$ constructed in such a way as to essentially fulfill the role of $(\ker (D^{z})^{+})^{\perp}$, but yielding a dimension $N$ which is independent of $z$.

Write $\Ind D = [V^{+}] - [V^{-}]$ and think of $V^{+}$ and $V^{-}$ as the even and odd components of a superbundle $V = V^{+}\oplus V^{-} \to B$. Then the Chern character of $\Ind D$ is defined as the cohomology class
\[
\ch(\Ind D) = \ch(V^{+}) - \ch(V^{-}) = [\ch(A^{+})] - [\ch(A^{-})] \in H_{dR}^{\bigcdot}(B)
\]
where $A$ is any superconnection on $V$, with chiral parts $A^{+}$ and $A^{-}$. With this setup, our index formula can be stated as follows.
\begin{theorem}[Index formula]\label{mainthm2}
Let $M \to B$ be an end-periodic fiber bundle and let $E \to M$ be a family of end-periodic Clifford modules, with $D$ the associated family of end-periodic Dirac operators, and suppose that the family $D$ is Fredholm. Then $\ch(\Ind D) \in H_{dR}^{\mathrm{ev}}(B)$ is represented in de Rham cohomology by
\begin{align*}
\ch(\con_{0}) &= \int_{Z/B}\AS(D(Z)) - \int_{W_{0}/B}f\,\AS(D(Y)) - \frac{1}{2}\epeta \\
&\hspace{6mm} - d_{B}\left(\frac{1}{2\pi i}\int_{0}^{\infty}\oint_{|\xi|=1}\Str\left(\frac{\partial \mc{Q}_{\xi}}{\partial\xi}\right)  d\xi\,dt\right).
\end{align*}
Moreover, the degree zero part of the end-periodic eta form  $\epeta$ is the fiberwise end-periodic eta invariant, i.e., $(\epeta)_{[0]} \in C^{\infty}(B)$ given by $z \mapsto \eta_{\mathrm{ep}}(D^{z}(Y_{z}))$.
\end{theorem}

The statement about the degree zero component of the end-periodic eta form follows from equation (\ref{degreezero}). The proof Theorem \ref{mainthm2} will be completed by calculating the limits in the integrated transgression formula (\ref{integratedtrans}), which we will do in \S\ref{section:short-time}--\ref{section:NCRcase} below. Throughout \S\ref{section:short-time}--\ref{section:RHSlimits} we assume that the family of Dirac operators $D$ has constant rank nullspace over $z \in B$. In \S\ref{section:NCRcase} we show that the general case reduces to this case, and that the index formula holds as stated without any constant rank assumption.

\subsection{Short-time limit of the renormalized Chern character}\label{section:short-time}

We will calculate the limit of $\rch(A_{t})$ as $t \to 0^{+}$ using the short-time asymptotic expansion for the heat kernel. The relevant short-time expansions for us are given by \cite[Theorem 10.21]{bgv} (for a closed manifold, or a family of closed manifolds) and \cite[Proposition 2.11]{roe} (for a manifold of bounded geometry). Combining these two results gives us the short-time asymptotic expansion for a family of heat kernels associated with the end-periodic Bismut superconnection.

\begin{proposition}\label{asymptoticsprop}
Let $M^{n} \to B$ be an end-periodic fiber bundle and let $E \to M$ be a family of end-periodic Clifford modules, with $D = D(M)$ the associated family of end-periodic Dirac operators. Let $A$ denote the end-periodic Bismut superconnection adapted to $D$, and let $K_{t} = (K_{t}^{z})_{z\in B}$ denote the associated family of heat kernels, i.e., the family of integral kernels for the heat operator $e^{-A_{t}^{2}}$. Then we have an asymptotic expansion as $t \to 0^{+}$,
\[
K_{t}(x,x) \sim (4\pi t)^{-n/2} \sum_{i=0}^{\infty} t^{i}k_{i}(x)
\]
with coefficients $k_{i} \in \sum_{j \leq 2i} \Lambda^{j}T^{*}M \otimes \End_{\Cl(T^{*}(M/B))}(E)$, such that
\[
\sum_{i=0}^{n/2}(k_{i})_{[2i]} = \mathrm{\mathbf{I}}(D(M)).
\]
\end{proposition}

Using the short-time asymptotic expansion we will show that
\[
\lim_{t\to 0^{+}} \rch(A_{t}) = \int_{Z/B}\AS(D(Z))
\]
where $Z$ is the compact part of the end-periodic fiber bundle. Consider the definition of the renormalized Chern character
\begin{align*}
\rch(A_{t}) &= \rStr(e^{-A_{t}^{2}}) \\
&= \lim_{N\to\infty} \left(\int_{Z_{N}/B}\str(K_{t}(x,x))\,dx - (N+1)\int_{W_{0}/B}\str(\tildhat{K}_{t}(x,x))\,dx\right) \\
&= \lim_{N \to \infty} s_{N}(t).
\end{align*}
For each fixed $N \in \mb{N}$ we can deduce from Proposition \ref{asymptoticsprop} that
\begin{align*}
\lim_{t\to 0^{+}} s_{N}(t) &= \lim_{t\to 0^{+}}\left(\int_{Z_{N}/B}\str(K_{t}(x,x))\,dx - (N+1)\int_{W_{0}/B}\str(\tildhat{K}_{t}(x,x))\,dx\right) \\
&= \int_{Z_{N}/B}\AS(D(M))\,dx - (N+1)\int_{W_{0}/B} \AS(D(\tildhat{Y}))\,dx \\
&= \int_{Z/B}\AS(D(Z))\,dx.
\end{align*}
By Proposition \ref{uniformityprop} the limit $\lim_{N\to\infty} s_{N}(t)$ converges uniformly on bounded time intervals, and so we are justified in interchanging limits to calculate
\begin{align*}
\lim_{t\to 0^{+}} \rch(A_{t}) &= \lim_{t\to 0^{+}}\lim_{N\to\infty} s_{N}(t) \\
&= \lim_{N\to\infty}\lim_{t\to 0^{+}} s_{N}(t) \\
&= \lim_{N\to\infty}\left(\int_{Z/B}\AS(D(Z))\right) \\
&= \int_{Z/B}\AS(D(Z)).
\end{align*} 

\subsection{Long-time limit of the renormalized Chern character}\label{section:long-time}

Let $M \to B$ be an end-periodic fiber bundle and $E \to M$ a family of end-periodic Clifford modules, with associated family of end-periodic Dirac operators $D = (D^{z})_{z_{\in B}}$ and end-periodic Bismut superconnection $A$ adapted to $D$. We let $P_{0}$ denote the projection onto $\Ker D$, and $\con_{0} = \con^{\Ker D} = P_{0}A_{[1]}P_{0}$ the connection on $\Ker D$ induced by $A$. We recall that we are operating under the following assumptions:
\begin{enumerate}[(i)]
\item The operators $(D^{z})_{z\in B}$ are Fredholm on $H^{1}(M_{z};E_{z})$, which is equivalent to saying that the normal operators $\tildhat{D}^{z}$ are invertible for every $z \in B$ or that the indicial families $D(Y)^{z}(\xi)$ are invertible for every $z \in B$ and $\xi \in S^{1}$.
\item The family $D = (D^{z})_{z\in B}$ has nullspace of constant rank over $z \in B$.
\item The base $B$ is compact.
\end{enumerate}
Under these assumptions, the arguments of \cite[Proposition 9.20, Lemma 9.20, Lemma 9.21]{bgv} go through with little change, using the properties of the heat kernel on a smooth Riemannian manifold with bounded geometry, cf. \cite{roe}. By these arguments the heat operator $e^{-A_{t}^{2}}$ can be understood by diagonalizing and then exponentiating. Let $P_{0}$ and $P_{1}$ denote the projections onto $\Ker D$ and $\Rg D$. We can find $g$ such that
\[
A^2 = g^{-1}
\begin{bmatrix}
U & 0 \\
0 & V
\end{bmatrix}
g
\]
with $U = \con_{0}^{2} + U^{+},$ and $V = P_{1}AP_{1} + V^{+}$. Rescaling in $t > 0$ we get
\[
\delta_{t}(A^{2}) = \delta_{t}(g)^{-1}
\begin{bmatrix}
\delta_{t}(U) & 0 \\
0 & \delta_{t}(V)
\end{bmatrix}
\delta_{t}(g)
\]
which then yields
\[
e^{-A_{t}^{2}} = e^{-t\delta_{t}(A^{2})} = \delta_{t}(g)^{-1}
\begin{bmatrix}
e^{-t\delta_{t}(U)} & 0 \\
0 & e^{-t\delta_{t}(V)}
\end{bmatrix}
\delta_{t}(g),
\]
from which we deduce
\[
e^{-t\delta_{t}U} = e^{-\con_{0}^{2}} + O(t^{-1/2}).
\]
Using \cite[Lemma 9.21]{bgv} to handle the term $e^{-t\delta_{t}(V)}$, we choose a precompact open subset $S \subseteq B$ and for any $z \in S$ set $\lambda^{z} = $ smallest nonzero eigenvalue of $(D^{z})^2$. Then for $\lambda = \inf_{z \in S} \lambda^{z}$ we have
\[
P_{1}e^{-t\delta_{t}V}P_{1} = O(e^{-t\lambda/2})
\]
uniformly on precompact open subsets of $B$. These estimates, together with the Volterra series for the heat kernels of $U$ and $V$, cf. \cite[equation (11.2)]{mp}, imply that
\[
\rStr(e^{-A_{t}^{2}}) = \rStr(e^{-\con_{0}^{2}}) + O(t^{-1/2}).
\]
Note that $e^{-\con_{0}^{2}}$ is trace-class by the invertibilty of the normal operators $\tildhat{D}^{z}$ for every $z \in B$, hence the renormalization factor vanishes and $\rStr(e^{-\con_{0}^{2}}) = \Str(e^{-\con_{0}^{2}})$. Thus taking the limit as $t \to \infty$ yields
\begin{equation}\label{eq:longtimeconvergence}
\lim_{t\to\infty} \rch(A_{t}) = \ch(\con_{0}).
\end{equation}

\subsection{Limits on the right-hand side}\label{section:RHSlimits}

In this section we will calculate the long-time and short-time limits of the term $(\dag)$ in (\ref{integratedtrans}) by applying the framework of \S\ref{section:FLtransform}, where we showed that the indicial family $A_{t}^{2}(Y)(\xi)$ when twisted by a flat connection produces a Bismut superconnection on a twisted bundle.

\begin{proposition}
Under the hypotheses of Theorem \ref{mainthm2}, we define
\[
(\dag) = \frac{1}{2\pi i}\oint_{|\xi|=1}\int_{W_{0}/B}f(x)\str\left(K_{e^{-A_{t}^{2}(Y)(\xi)}}(x,x)\right)dx\frac{d\xi}{\xi}.
\]
Then we have
\[
\lim_{t \to \infty} (\dag) = 0 \mbox{ \ and \ } \lim_{t \to 0^{+}} (\dag) = \int_{W_{0}/B}f\cdot \mathrm{\mathbf{I}}(Y/B).
\]
\end{proposition}
\begin{proof}
For the long-time limit, by applying \cite[Theorem 9.19]{bgv} we calculate
\begin{align*}
\lim_{t \to \infty} \str\left(K_{\exp(-A_{t}^{2}(Y)(\xi))}(x,x)\right) = \str\left(K_{\exp(\con_{0}^{2})}(x,x)\right) = 0
\end{align*}
where $\con_{0} = P_{0}A_{[1]}(Y)(\xi)P_{0}$ is the induced connection on $\Ker D_{\xi}(Y)$, which is just the zero space under the assumption that $D = D(M)$ is Fredholm. Therefore $\lim_{t\to\infty} (\dag) = 0$.

For the short-time limit, consider the Bismut superconnection $\mb{A}$ on $\mb{E}_{Y}\otimes L_{\xi}$ defined in Proposition \ref{twistedprop} as the twist of $A(Y)(\xi)$ by the flat connection $d$ on the bundle $L_{\xi}$. Since $\mb{A}$ is obtained by twisting $A(Y)(\xi)$ by this flat connection on a rank-1 bundle, the local heat invariants in the short-time asymptotic expansion for the heat operators $e^{-A_{t}^{2}(Y)(\xi)}$ and $e^{-\mb{A}_{t}^{2}}$ coincide exactly. Therefore, applying \cite[Theorem 10.23]{bgv} to the rescaled Bismut superconnection $\mb{A}_{t}$ on $\mb{E}_{Y}\otimes L_{\xi}$ we calculate
\begin{align*}
\lim_{t \to 0^{+}} \str\left(K_{\exp(-A_{t}^{2}(Y)(\xi))}(x,x)\right) = \lim_{t \to 0^{+}} \str\left(K_{e^{-{\mb{A}_{t}}}}(x,x)\right)
= \mathrm{\mathbf{I}}(Y/B)\ch(L_{\xi}) 
= \mathrm{\mathbf{I}}(Y/B)
\end{align*}
since the Chern character is multiplicative over tensor products, and $L_{\xi}$ is equipped with a flat connection which implies $\ch(L_{\xi}) = 1$. Therefore
\begin{align*}
\lim_{t \to 0^{+}} (\dag) = \lim_{t\to 0^{+}} \frac{1}{2\pi i}\oint_{|\xi|=1}\int_{W_{0}/B}f(x)\str\left(K_{e^{-A_{t}^{2}(Y)(\xi)}}(x,x)\right)dx\frac{d\xi}{\xi}
= \int_{W_{0}/B}f\cdot \mathrm{\mathbf{I}}(Y/B).
\end{align*}
\vspace{-3mm}
\end{proof}

\subsection{Non-constant rank case}\label{section:NCRcase}
Now we remove the requirement that the family of Dirac operators $D = (D^{z})_{z\in B}$ have constant $\dim \ker (D^{z})^{+}$ over $z \in B$. In this case, the kernels $\ker D^{z}$ no longer assemble a vector bundle, and the standard stabilization procedure perturbs $D$ by a vertical family of odd smoothing operators $P \in \Psi^{-\infty}_{\pi}(M;\mb{E})$ so that $D + P$ is injective on each fiber, hence $\Coker(D+P) \to B$ is a well-defined vector bundle, and the index bundle $\Ind D \in K(B)$ is defined as in equation (\ref{eqn:indexbundle}).

Let $A^{P}$ denote the Bismut superconnection adapted to the perturbed family $D+P$ and let $A_{t}^{P}$ denote the rescaled superconnection. Then we know from the constant-rank case that
\begin{equation}
\lim_{t\to\infty}\rch(A_{t}^{P}) = \ch(\con^{\Ker(D+P)})\label{eq:longtime}
\end{equation}
cf. \cite[Theorem 9.26]{bgv}, and it is straightforward to calculate, using the same methods as before, that
\begin{equation}
\lim_{t\to 0^{+}}\rch(A_{t}^{P}) = \int_{Z/B}\AS(D(Z)) + \ch(B\times\mb{C}^{N}). \label{eq:shorttime}
\end{equation}
Then taking the transgression formula with the perturbed Bismut superconnection $A_{t}^{P}$ and integrating over $t \in (0,\infty)$ yields
\[
\lim_{t\to\infty}\rch(A_{t}^{P}) - \lim_{t\to 0^{+}}\rch(A_{t}^{P}) = \lim_{t \to \infty}(\dag^{P}) - \lim_{t\to 0^{+}}(\dag^{P}) - \frac{1}{2}\epeta^{P}
\]
where the superscript $P$ on the right-hand side indicates that the correction term involves the perturbed operator $A_{t}^{P}$ in its definition instead of $A_{t}$. Now from the definition of the index bundle we have
\[
\Ind(D) = [\Ker(D+P)] - [B\times\mb{C}^{N}]
\]
and thus, using \eqref{eq:longtime} and \eqref{eq:shorttime}, 
\begin{align*}
\ch(\Ind D) &= [\ch(\con^{\Ker(D+P)})]_{\mathrm{dR}} - N \\
&= \lim_{t \to \infty} [\rch(A_{t}^{P})]_{\mathrm{dR}} - N \\
&= \left[\int_{Z/B}\AS(D(Z)) + \lim_{t \to \infty}(\dag^{P}) - \lim_{t\to 0^{+}}(\dag^{P}) - \frac{1}{2}\epeta^{P}\right]_{\mathrm{dR}}
\end{align*}
where two factors of $N$ have canceled each other, and the notation $[\ \cdot \ ]_{\mathrm{dR}}$ indicates a de Rham cohomology class. Thus the index formula will follow from this equation as soon as we calculate these correction terms associated with the perturbed superconnection $A_{t}^{P}$. 

In fact, these correction terms depend only upon the indicial family of the superconnection $A_{t}$. We will show that, for a careful choice of perturbation $P$, the indicial family is unchanged by the perturbation, and so the limits $\lim_{t\to\infty} (\dag^{P})$ and $\lim_{t\to 0^{+}}(\dag^{P})$ are exactly the same as the limits we calculated in \S\ref{section:RHSlimits}. Similarly, the end-periodic eta form is unaffected by the perturbation, that is $\epeta^{P} = \epeta$.

Following this idea, we aim to construct a family of rank-stabilizing perturbations $P = (P^{z})_{z\in B}$ so that (i) $D+P$ is a family having nullspace of constant rank, and (ii) $N(P) = \tildhat{P} = 0$ (in particular, $P$ will be trace-class). Since the correction terms in the index formula depend only on the normal operator (or indicial family), the correction terms are unchanged when we replace $D$ with $D+P$. This will complete the proof of the index formula in the non-constant rank case. 

We will adapt the proof of \cite[Lemma 1.1]{mr} to our context. One of the key components is the parametrix construction for the Dirac operator. Under the assumption that $D$ is Fredholm (i.e., the normal operator $N(D)$ is invertible), there exists a parametrix $G$ such that $GD = I-R$ where the error $R$ is a compact operator. We additionally want to show that we can construct the parametrix so that the error is a compact \emph{smoothing} operator. In order to establish this we will make use of the uniform calculus of pseudodifferential operators $\mathcal{U}\Psi^{s}(M;E)$, cf. \cite[Definition 3.2]{shubin}. The uniform calculus can be defined for any manifold of bounded geometry, in particular for an end-periodic manifold.
\begin{lemma}\label{parametrixlemma}
Let $M$ be an end-periodic manifold and let $D$ be an end-periodic Dirac operator on $C_{c}^{\infty}(M;E)$. Suppose that the normal operator $N(D)$ is invertible. Then $D$ has a parametrix $G \in \mathcal{U}\Psi^{-1}(M;E)$ such that the error $R = I - GD$ is a compact smoothing operator.
\end{lemma}
\begin{proof}
First we will construct a parametrix for $D$ with compact error following the method of \cite[Lemma 5.5]{mazzeo}. Let $Z_{3} = Z\sqcup W_{0}\sqcup W_{1}\sqcup W_{2}$ and let $\widehat{Z}_{3}$ denote the double of $Z_{3}$ across the boundary. Let $D(\widehat{Z}_{3})$ denote the Dirac operator acting on sections of $E|_{\widehat{Z}_{3}}$, with respect to the obvious geometric data. Since $\widehat{Z}_{3}$ is a closed manifold, $D(\widehat{Z}_{3})$ is Fredholm, so we can find a parametrix $A$ such that $C = I - AD(\widehat{Z}_{3})$ is compact. Now by patching together the parametrix $A$ for $D(\widehat{Z}_{3})$ near $Z$, with the inverse normal operator $N(D)^{-1}$ along the end, we define an operator
\[
G_{1} = \psi_{1}A\phi_{1} + \psi_{2}N(D)^{-1}\phi_{2}.
\]
The smooth cutoffs $\psi_{j},\phi_{j} \in C^{\infty}(M)$ are defined so that:
\begin{itemize}
\item $\phi_{1}|_{Z\sqcup W_{0}} = 1$ and $\phi_{1}|_{M\setminus Z_{2}} = 0$,
\item $\phi_{2} = 1-\phi_{1}$,
\item $\supp(\psi_{1}) \subset Z_{3}$ and $\psi_{1}|_{\supp(\phi_{1})} = 1$,
\item $\supp(\psi_{2}) \subset \tildhat{Y}_{+}$ and $\psi_{2}|_{\supp(\phi_{2})} = 1$.
\end{itemize}
Thus the segment $W_{2}$ serves as the ``patching region''. A direct calculation shows that
\begin{align*}
G_{1}D &= [D,\psi_{1}]A\phi_{1} + \psi_{1}AD\phi_{1} + [D,\psi_{2}]N(D)^{-1}\phi_{2} +    \psi_{2}DN(D)^{-1}\phi_{2} \\
&= [D,\psi_{1}]A\phi_{1} + [D,\psi_{2}]N(D)^{-1}\phi_{2} - \psi_{1}C\phi_{1} + (\phi_{1}+\phi_{2}) \\
&= I - (\mathrm{compact \, operator}).
\end{align*}
In the last line one can argue exactly as in \cite[Lemma 5.5]{mazzeo} that the error is compact using the compact inclusion afforded by Rellich's lemma and the fact that the terms are smoothing of order at least $-1$.  Thus we have a parametrix $G_{1}$ for $D$ which is a (classical) pseudodifferential operator of order $-1$, such that the error $R_{1} = I - G_{1}D$ is compact.

Next we will show that $G_{1} \in \mathcal{U}\Psi^{-1}(M;E)$ is a uniform pseudodifferential operator, cf. \cite[Definition 3.2]{shubin}. From the patching construction we have $G_{1} = \psi_{1}A\phi_{1} + \psi_{2}N(D)^{-1}\phi_{2}$, and it is clear that $\psi_{1}A\phi_{1}$ is uniform because $A$ is a parametrix for a Fredholm operator on the closed manifold $\widehat{Z}_{3}$, so it will suffice to show that $\psi_{2}N(D)^{-1}\phi_{2} \in \mathcal{U}\Psi^{-1}(\tildhat{Y};\tildhat{E})$. It is easy to arrange that the integral kernel vanishes sufficiently far from the diagonal by adjusting the cutoff functions, so the only remaining point is to verify the uniformity condition on the symbol.

Using the bounded geometry of $\tildhat{Y}$ we fix a uniformly locally finite covering of $\tildhat{Y}$ by normal coordinate balls $\{B_{\epsilon}(x_{i})\}_{i\in\mathcal{I}}$, for some $\epsilon > 0$ which is bounded in terms of the injectivity radius of $\tildhat{Y}$. Let $\{\chi_{i}\}_{i\in\mathcal{I}}$ be a partition of unity subordinate to this open cover. By choosing $\epsilon > 0$ less than the injectivity radius of $Y$ we can ensure that the covering map restricts to an isometry $p: B_{\epsilon}(x_{i}) \to B_{\epsilon}(p(x_{i}))$ for every $i\in\mathcal{I}$. 

Let $q_{i}(x,\zeta)$ denote the symbol of the localized operator $\chi_{i}N(D)^{-1}\chi_{i}$ on $B_{\epsilon}(x_{i})$, and let $q_{i}(\xi,p(x),\zeta)$ denote the symbol of the localization of $D_{\xi}(Y)^{-1}$ on $B_{\epsilon}(p(x_{i}))$ (recall that the invertibility of $N(D)$ is equivalent to the invertibility of the indicial family $D_{\xi}(Y)$). The same calculation as in Proposition \ref{newprop} shows that the integral kernels are related by
\[
K_{N(D)^{-1}}(x,y) = \frac{1}{2\pi i}\oint_{S^{1}}\xi^{f(y)-f(x)}K_{D_{\xi}(Y)^{-1}}(p(x),p(y))\frac{d\xi}{\xi}
\]
and therefore the local symbols are related by
\[
q_{i}(x,\zeta) = \frac{1}{2\pi i}\oint_{S^{1}}q_{i}(\xi,p(x),\zeta)\frac{d\xi}{\xi}.
\]
Since $D_{\xi}(Y)^{-1}$ is a pseudodifferential operator of order $-1$ on the closed manifold $Y$ we have
\[
|D_{x}^{\alpha}D_{\zeta}^{\beta}q_{i}(\xi,p(x),\zeta)| \leq C_{i}^{\alpha\beta}(\xi)(1+|\zeta|)^{-1-|\beta|}
\]
for some $C_{i}^{\alpha\beta}(\xi) > 0$. Moreover, the symbols $q_{i}(\xi,p(x),\zeta)$ depend smoothly on $\xi \in S^{1}$, so the seminorm
\[
C_{i}^{\alpha\beta}(\xi) = \sup_{y,\zeta}|D_{y}D_{\zeta}^{\beta}q_{i}(\xi,y,\zeta)|(1+|\zeta|)^{1+|\beta|}
\]
is a smooth function of $\xi$. Thus we can find a constant
\[
C_{i}^{\alpha\beta} = \sup_{\xi\in S^{1}}C_{i}^{\alpha\beta}(\xi) < \infty
\]
so that
\begin{align*}
|D_{x}^{\alpha}D_{\zeta}^{\beta}q_{i}(x,\zeta)| \leq \frac{1}{2\pi i}\oint_{S^{1}}|D_{x}^{\alpha}D_{\zeta}^{\beta}q_{i}(\xi,p(x),\zeta)|\frac{d\xi}{\xi} \leq C_{i}^{\alpha\beta}(1+|\zeta|)^{-1-|\beta|}
\end{align*}
for any multi-indices $\alpha,\beta$ and $i \in \mathcal{I}$. Further, since $Y$ is compact we can find a finite subset of points $\{x_{j_{1}}\ldots,x_{j_{m}}\}$ such that the balls $\{B_{\epsilon}(p(x_{j_{\ell}}))\}_{1\leq \ell \leq m}$ cover $Y$. Choosing $C^{\alpha\beta} = \max\{C^{\alpha\beta}_{j_{\ell}}: 1\leq \ell \leq m\}$ then yields a uniform bound
\[
|D_{x}^{\alpha}D_{\zeta}^{\beta}q_{i}(x,\zeta)| \leq C^{\alpha\beta}(1+|\zeta|)^{-1-|\beta|}
\]
which holds for every $i\in\mathcal{I}$. We conclude that $\psi_{2}N(D)^{-1}\phi_{2} \in \mathcal{U}\Psi^{-1}(\tildhat{Y};\tildhat{E})$ and therefore $G_{1} \in \mathcal{U}\Psi^{-1}(M;E)$.

The uniform pseudodifferential operators form an algebra under composition, and the principal symbol is a homomorphism fitting into the short exact sequence
\[
0 \xrightarrow{} \Psi^{s-1}(M;E) \xrightarrow{} \Psi^{s}(M;E) \xrightarrow{\sigma} \mathcal{U}C^{\infty}(S^{*}M;\pi^{*}\End E) \xrightarrow{} 0
\]
where $\mathcal{U}C^{\infty}(S^{*}M;\pi^{*}\End E)$ denotes the space of smooth uniformly bounded sections of $\pi^{*}\End E$ over the cotangent sphere bundle $\pi: S^{*}M \to M$. Using the short exact sequence, we may apply the standard symbolic iterative inversion scheme (as exhibited in the proof of \cite[Theorem 7.24]{melrose} for example) to upgrade $G_{1}$ to a parametrix $G$ such that the error $R = I - GD$ is a compact smoothing operator.
\end{proof}

Consider the bundle of Hilbert spaces $H^{s}_{\pi}(M/B;E) \to B$ with fibers $H^{s}(M_{z};E_{z})$ for $z \in B$. Fix an orthonormal basis $(e_{j})_{j\geq 1}$ for $H^{s}_{\pi}(M/B;E)$ and let $\Pi_{N}^{L_{2}}$ denote the orthogonal projection onto the first $N$ terms of the orthonormal basis. Then $\Pi_{N}^{L_{2}}$ has an $L^{2}$-integral kernel of the form
\[
K_{{\Pi}_{N}^{L_{2}}}(x,y) = \sum_{1 \leq j \leq N} e_{j}(x)\otimes e_{j}^{*}(y).
\]
By approximating the $L^{2}$-sections $(e_{j})$ by compactly supported smooth sections, we may pass to a smooth integral kernel $K_{\Pi_{N}}$ which decays exponentially along $\en(M)$, and which induces a compact smoothing operator $\Pi_{N}$. The family $\Pi_{N} = (\Pi_{N}^{z})_{z \in B}$ constructed in this way consists of projections with constant rank over $z \in B$. Moreover, the family $\Pi_{N}$ is asymptotically periodic with normal operator $N(\Pi_{N}) = 0$ by virtue of the exponential decay along $\en(M)$. 

Applying Lemma \ref{parametrixlemma} to our family of end-periodic Dirac operators $D$ we can find a smooth family of parametrices $G = (G^{z})_{z\in B}$ so that $GD = I - R$ where $R$ is a family of compact smoothing operators. Since the sequence $\Pi_{N}$ converges strongly to the identity as $N \to \infty$ we can choose $N$ sufficiently large so that the compact smoothing operator $S = R(I - \Pi_{N})$ has uniformly small operator norm. Thus the Neumann series
\[
(I-S)^{-1} = \sum_{n \geq 0} S^{n} = I + \sum_{n\geq 1}S^{n}
\]
converges in operator norm. In fact, we can show that the corresponding Neumann series of smooth integral kernels $\sum_{n}K_{S^{n}}(x,y)$ converges in the $C^{\infty}(M\times M;E\boxtimes E^{*})$ (Frech\'et space) topology. It follows that $(I-S)^{-1} = I + S'$ where $S'$ is a smoothing operator.

\begin{lemma}\label{constantranklemma}
Let $(M,g)$ be a smooth Riemannian manifold without boundary and denote by $\con: \Gamma(M;\otimes^{\ell}T^{*}M)\to\Gamma(M;\otimes^{\ell+1}T^{*}M)$ the Levi-Civita connection acting on smooth covariant tensor fields, for any $\ell \in \mb{N}$. We also let $\con$ denote the Levi-Civita connection acting on tensor fields on the Riemannian product $(M\times M,g\oplus g)$. Let $|\cdot|_{g}$ denote the norm induced by $g$ on smooth covariant tensor fields. Let $S$ be a compact smoothing operator on $L^{2}(M)$ with operator norm $0 < \|S\| < 1$ so that the Neumann series 
\[
(I-S)^{-1} = \sum_{n \geq 0} S^{n} = I + S'
\]
converges in operator norm. Let $\{U_{j}\}_{j\in\mb{N}}$ be any exhaustion of $M \times M$ by compact subsets, and equip $C^{\infty}(M\times M)$ with the Frech\'et space topology induced by the family of seminorms
\[
p_{j,k}(v) = \sup_{(x,y)\in U_{j}}|\con^{k}v(x,y)|_{g}
\]
for $j,k \in \mb{N}$, and $v \in C^{\infty}(M\times M)$. Then the Neumann series $\sum_{n\geq 1}K_{S^{n}}(x,y)$ converges in $C^{\infty}(M\times M)$, hence $K_{S'}$ is smooth and $S'$ is a smoothing operator.
\end{lemma}
\begin{proof}
We will show that for any $j,k,n \in \mb{N}$ we have
\[
p_{j,k}(K_{S^{n}}) \leq C_{j,k}\|S^{n}\|
\]
for some constant $C_{j,k} > 0$. Since $S^{n}$ is a smoothing operator it induces a bounded map $S^{n}: L^{2}(M) \to C^{\infty}(M) \subseteq H^{m}(M)$ for every $m \in \mb{N}$, and if we fix a compact subset $U \subseteq M$ then we have $\|S^{n}u\|_{H^{m}(U)} \leq C_{m}\|S^{n}\|\cdot\|u\|_{L^{2}(U)}$ for any $u \in L^{2}(U)$. Let $\{(\phi_{i},V_{i})\}_{1\leq i \leq q}$ be a finite open cover of $U$ by smooth charts for $M$. Let $\{\chi_{i}\}_{1\leq i \leq q}$ be a partition of unity subordinated to this cover. Applying the Sobolev embedding theorem on each $V_{i}$, we have
\begin{align*}
\sum_{|\alpha|\leq k}\sup_{x\in V_{i}}|\partial^{\alpha}((\chi_{i}S^{n}u)\circ\phi_{i}^{-1})(\phi_{i}(x))| = \|\chi_{i}S^{n}u\|_{C^{k}(V_{i})} \leq C_{i}\|\chi_{i}S^{n}u\|_{H^{m}(V_{i})}
\end{align*}
for some constants $C_{i} > 0$ and $m \in \mb{N}$ chosen sufficiently large (e.g., $m > k + \dim(M)/2$) so that the Sobelev embedding can be applied.  Then summing over the partition of unity yields
\begin{align*}
\sup_{x\in U}|\con^{k}S^{n}u(x)|_{g} &\leq \sum_{1\leq i \leq q}\sup_{x\in V_{i}}|\con^{k}(\chi_{i}S^{n}u)(x)|_{g} \\
&\leq C\sum_{1\leq i \leq q}\sum_{|\alpha|\leq k}\sup_{x\in V_{i}}|\partial^{\alpha}((\chi_{i}S^{n}u)\circ\phi_{i}^{-1})(\phi_{i}(x))| \\
&= C\sum_{1\leq i \leq q}\|\chi_{i}S^{n}u\|_{C^{k}(V_{i})} \\
&\leq C\sum_{1\leq i \leq q} C_{i}\|\chi_{i}S^{n}u\|_{H^{m}(V_{i})} \\
&\leq C_{U,k}\|S^{n}u\|_{H^{m}(U')} \\
&\leq C_{U,k,m}\|S^{n}\|\cdot\|u\|_{L^{2}(U')}
\end{align*}
for some compact $U \subset U' \subset M$ and constants $C_{U,k,m}>0$. Fixing $y \in U$, we may construct a family of mollifiers $u_{\epsilon} \in C_{c}^{\infty}(M)$ supported near $y$ such that $\|u_{\epsilon}\|_{L^{2}} = 1$ for every $\epsilon > 0$ and $u_{\epsilon} \to \delta_{y} \in H^{-m}(U)$ in the weak topology as $\epsilon \to 0^{+}$. Applying the preceding inequality with these $u_{\epsilon}$ and then taking the limit as $\epsilon \to 0^{+}$ yields
\[
\sup_{x\in U}|\con_{x}^{k}K_{S^{n}}(x,y)|_{g} \leq C_{U,k,m}\|S^{n}\|.
\]
Applying the same argument with the adjoint $S^{*}$ and $u_{\epsilon} \to \delta_{x}$ yields the same bound on covariant derivatives of $K_{S^{n}}$ in the $y$-variable. From this we deduce
\[
p_{j,k}(K_{S^{n}}) = \sup_{(x,y)\in U_{j}}|\con^{k}K_{S^{n}}(x,y)|_{g} \leq C_{j,k}\|S^{n}\|
\]
and it follows that, for any $j,k \in \mb{N}$,
\[
\sum_{n=1}^{\infty}p_{j,k}(K_{S^{n}}) \leq C_{j,k}\sum_{n=1}^{\infty}\|S\|^{n} < \infty
\]
thus the Neumann series of integral kernels converges in $C^{\infty}(M\times M)$, and it is clear that it can only converge to $K_{S'}$. Hence $K_{S'}$ is smooth and $S'$ is a smoothing operator.
\end{proof}
It is straightforward to extend Lemma \ref{constantranklemma} to the case where $E \to M$ is a Hermitian vector bundle equipped with a metric connection $\con^{E}$, and $S$ is a compact smoothing operator acting on $L^{2}(M;E)$; the same proof goes through with only notational changes. Thus we conclude that our operator $S = R(I-\Pi_{N})$ satisfies $(I-S)^{-1} = I+S'$ where $S'$ is a smoothing operator.

We fix $N$ sufficiently large as above and set $\Pi = \Pi_{N}$, and define the perturbation $P = -D\Pi$. Since $D$ is end-periodic and $\Pi$ is asymptotically periodic with $N(\Pi) = 0$, it follows that $P$ is asymptotically periodic with normal operator $N(P) = 0$. We construct a new parametrix $G'$ so that
\[
G'(D+P) = I - \Pi.
\]
To construct this new parametrix $G'$, we note that
\[
GD(I-\Pi) = (I-R)(I - \Pi) = (I - R(I-\Pi))(I-\Pi) = (I-S)(I-\Pi)
\]
where $S = R(I-\Pi)$ as above. From the identity $(I-S)^{-1} = I+S'$ it follows that $G' = (I + S')G$ satisfies $G'(D+P) = I - \Pi$. By the construction of $\Pi$ as a family of projections of the same rank over $z \in B$, the equation $G'(D+P) = I - \Pi$ implies that $D+P$ has nullspace of constant dimension over $z \in B$.

In summary, we have constructed a family of asymptotically periodic smoothing perturbations $P$ such that $D+P$ has constant rank nullspace and $N(P) = 0$. As a result, the normal operators and indicial families of $D$ coincide with those of $D+P$, and the same goes for the Bismut superconnections $A_{t}$ and $A_{t}^{P}$. We conclude that
\[
\lim_{t \to \infty}(\dag^{P}) = \lim_{t \to \infty}(\dag), \hspace{3mm} \lim_{t\to 0^{+}}(\dag^{P}) = \lim_{t \to 0^{+}}(\dag), \hspace{3mm} \epeta^{P} = \epeta.
\]
This completes the proof of Theorem \ref{mainthm2} in the non-constant rank case.